\documentclass{article}
\usepackage{graphicx} %
\usepackage{subcaption}
\usepackage[
    a4paper,
    left=2.8cm,
    right=2.5cm,
    top=2.5cm,
    bottom=2.5cm,
    headsep=0.8cm,
    footskip=1cm
]{geometry}

\newcounter{case}
\newcommand{\mycase}{\refstepcounter{case}\textbf{Case \thecase.}}

\usepackage{graphicx} 
\usepackage{amsmath,amssymb}
\usepackage{mathrsfs}
\usepackage{amsthm}
\usepackage{subcaption}
\usepackage{comment}
\usepackage{float}
\usepackage{xcolor}
\usepackage{hyperref}
\usepackage{cleveref}
\usepackage[square,numbers]{natbib} %

\usepackage{thmtools}
\usepackage{thm-restate}
\usepackage{bm}

\newcommand{\intfour}{[4]}
\newcommand{\inttwo}{[2]}
\newcommand{\intthree}{[3]}

\newtheorem{theorem}{Theorem}[section]
\newtheorem{Definition}{Definition}[section]
\newtheorem{lemma}[theorem]{Lemma}
\newtheorem{proposition}[theorem]{Proposition}
\newtheorem{problem}{Problem}[section] 
\newtheorem{claim}{Claim}[section]
\newtheorem{corollary}{Corollary}[section]
\newtheorem{conjecture}{Conjecture}[section]

\usepackage{tikz}
\usetikzlibrary{arrows.meta,positioning}
\usepackage{booktabs}
\usetikzlibrary{backgrounds}

\usetikzlibrary{arrows.meta,shapes.geometric}

\definecolor{starfill}{RGB}{0,0,0}
\definecolor{diamondfill}{RGB}{0,0,0}
\definecolor{pathgray}{RGB}{0,0,0}
\definecolor{qgray}{RGB}{0,0,0}
\definecolor{highlightgray}{RGB}{210,210,210}
\definecolor{highlightdark}{RGB}{0,0,0}
\title{The unilateral 2-linkage problem}

\begin{document}

\author{
J{\o}rgen Bang-Jensen$^{1,2}$\;\quad Xiaoying Chen$^{1}$\;\quad  Panpan Cheng$^{1}$\thanks{Corresponding author: Panpan Cheng (ppancheng@126.com).} \;\quad Jin Yan$^{1}$
\\[0.4em]
\begin{minipage}{0.92\textwidth}
\centering
\small
$^{1}$ School of Mathematics, Shandong University, Jinan 250100, China\\[2pt]
$^{2}$ Department of Mathematics and Computer Science, University of Southern Denmark, Odense DK-5230, Denmark
\end{minipage}
}

\date{}

\maketitle 

\begin{abstract}
    We study a new variant of the classical linkage problem for digraphs, which we call {\bf unilateral linkages.} A digraph $D$ is {\bf unilaterally $\bm{k}$-linked} if it satisfies that for every choice of $2k$ distinct vertices 
$s_1,s_2,\ldots{},s_k,t_1,t_2,\ldots{},t_k$ there are vertex disjoint directed paths $P_1,\ldots{},P_k$ in $D$ such that  $P_i$ is either an $(s_i,t_i)$-path or a $(t_i,s_i)$-path for each $i\in [k]$.
This is a relaxation of the property of $D$ being $k$-linked, where we insist that each $P_i$ must go from $s_i$ to $t_i$. Clearly every $k$-linked digraph is also unilaterally $k$-linked. 

It is easy to see that the unilateral linkage problem is NP-complete for general digraphs and that it is trivial for semicomplete digraphs (digraphs with no pair of non-adjacent vertices). We show how Thomassen's construction of highly connected digraphs which are not 2-linked can be modified to show that there is no connectivity which guarantees that a digraph is unilaterally 2-linked. 
We show that every 4-strong split digraph and every 5-strong semicomplete multipartite digraph is unilaterally 2-linked. Both bounds are best possible. A digraph is {\bf co-bipartite} if its vertex set can be covered by two induced semicomplete digraphs. We prove that every 6-strong co-bipartite digraph is unilaterally 2-linked. Here we do not know whether the correct answer is 5. We also prove that every 16-strong digraph with independence number at most 2 is unilaterally 2-linked.
\end{abstract}

\noindent{}{\bf Keywords:} Unilateral linkages; linkage problem; split digraph; semicomplete multipartite digraph; co-bipartite digraph; independence number; NP-complete
\section{Introduction}

Notation and terminology in this paper is consistent with \cite{bang2009}. 
A {\bf $k$-linkage} from $s_1,s_2,\ldots{},s_k$ to $t_1,t_2,\ldots{},t_k$ in a (di)graph $G$ is a collection of $k$ vertex disjoint paths $P_1,P_2,\ldots{},P_k$ such that $P_i$ is an $(s_i,t_i)$-path for $i\in [k]$. The {\bf $k$-linkage problem} is the problem of deciding for a given (di)graph $G$ and ordered sets of vertices $(s_1,s_2,\ldots{},s_k)$ and $(t_1,t_2,\ldots{},t_k)$; whether $G$ has a $k$-linkage from $(s_1,s_2,\ldots{},s_k)$ to $(t_1,t_2,\ldots{},t_k)$.
A (di)graph $G$ is {\bf $k$-linked}  if it has a $k$-linkage from $(s_1,s_2,\ldots{},s_k)$ to $(t_1,t_2,\ldots{},t_k)$ for every choice of distinct vertices $s_1,s_2,\ldots{},s_k,t_1,t_2,\ldots{},t_k$.

Linkage problems are well understood for undirected graphs.
Robertson and Seymour proved that the $k$-linkage problem is polynomially solvable for every fixed $k$ for undirected graphs \cite{robertsonJCT63}. Many people worked on sufficient conditions for a graph to be $k$-linked. The strongest result is due to Thomas and Wollan who proved that every $2k$-connected graph on $n$ vertices with at least $5kn$ edges is $k$-linked  \cite{thomasEJC26}.

For digraphs the situation is very different: Fortune, Hopcroft and Wyllie proved that the $k$-linkage problem is NP-complete for all $k\geq 2$ \cite{fortuneTCS10} and  
Thomassen \cite{thomassenC11} proved that there exist $K$-strong non-2-linked digraphs for every natural number $K$ (see Theorem \ref{thm:thomassen} below).
This has inspired research on classes of digraphs for which one can solve the $k$-linkage problem in polynomial time as well as finding sufficient conditions for being $k$-linked in specific classes of digraphs. We list some of the results below.

\begin{itemize}

\item[(A)] In their seminal paper \cite{fortuneTCS10} Fortune Hopcroft and Wyllie proved that for every fixed $k$ the $k$-linkage problem is polynomial for acyclic digraphs.
\item[(B)] Bang-Jensen and Thomassen \cite{bangSJDM5}
proved that the 2-linkage problem is polynomial for semicomplete digraphs.

\item[(C)] Chudnovsky, Scott and Seymour \cite{chudnovskyAM270} generalized the result in (B) to all fixed $k$ and  more general classes of digraphs called $d$-path dominant digraphs. For a fixed integer $d$ a digraph $D$ is {\bf $\bm{d}$-path dominant} if the vertex set of every minimal path $P$ in $D$ with $d$ vertices is adjacent to all vertices of $V(D)-V(P)$. That is, every vertex of $V(D)-V(P)$ has at least one neighbour in $V(P)$.
Thus semicomplete digraphs correspond to the case $d=1$ and semicomplete multipartite digraphs to the case $d=2$.

\item[(D)]Bang-Jensen \cite{bangADM41} proved that every 5-strong semicomplete digraph is 2-linked and constructed 4-strong tournaments that are not 2-linked. 

\item[(E)]In a highly influential paper Pokrovskiy \cite{pokrovskiyJCT115} proved that every 452$k$-strong tournament is $k$-linked and conjectured the existence of a function $g(k)$ such that every $2k$-strong tournament with minimum in- and out-degree at least $g(k)$ is $k$-linked. 
\item[(F)]This was almost confirmed by Gir{\~{a}}o, Popielarz and Snyder \cite{giraoC41} who proved the existence of $f(k)$ so that every $(2k+1)$-strong tournament with minimum out-degree $f(k)$ is $k$-linked. 
\item[(G)] Finally Zhou and Yan \cite{zhouJCT176} disproved Pokrovskiy's conjecture by constructing $2k$-strong tournaments on $n\geq 14k^2$ vertices with arbitrarily high minimum in- and out-degrees (depending on $n$) which are not $k$-linked.
\end{itemize}

In  this paper we introduce a natural relaxation of the $k$-linkage problem for digraphs and prove its relevance through a series of results.

\begin{Definition}
Let $D=(V,A)$ be a digraph. Two disjoint ordered $k$-element subsets $S=(s_1,s_2,\dots,s_k)$ and $T=(t_1,t_2,\dots,t_k)$ of $V$ are {\bf unilaterally linked} if there exist vertex‑disjoint directed paths $P_1,P_2,\dots,P_k$ such that for each index $i$, $P_i$ is either an $(s_i,t_i)$‑path or a $(t_i,s_i)$‑path in $D$.
A digraph $D$ is {\bf unilaterally $k$-linked} if for every choice of two disjoint ordered $k$-element subsets $S,T\subseteq V(D)$, the subsets $S$ and $T$ are unilaterally linked.
\end{Definition}

\begin{theorem}
    The unilateral $k$-linkage problem is NP-complete for every $k\geq 2$.
\end{theorem}
\begin{proof}
    Let $(D,s_1,s_2,\ldots{},s_k,t_1,t_2,\ldots{},t_k)$ be an instance of the $k$-linkage problem and construct a new digraph $D'$ by adding new vertices $s'_1,\ldots{},s'_k$ and  arcs $s'_1s_1,s'_2s_2,\ldots{},s'_ks_k$. In the new digraph $D'$ the ordered sets $s'_1,\ldots{},s'_k$ and $t_1,t_2,\ldots{},t_k$ are unilaterally linked if and only if $D$ has a $k$-linkage from $(s_1,s_2,\ldots{},s_k)$ to $(t_1,t_2,\ldots{},t_k)$. 
\end{proof}

For an acyclic digraph $D$  and an acyclic ordering $v_1,v_2,\ldots{},v_n$ of its vertices, the unilateral $k$-linkage problem and the normal $k$-linkage problem are 'equivalent' in the sense that there is precisely one possible direction of the path between two terminals $s_i,t_i$ when we make the natural assumption that  $D$ has a path between $s_i$ and $t_i$ for all $i\in [k]$.

\begin{theorem}[Thomassen~\cite{thomassenC11}]\label{thm:thomassen}
For every positive integer $m$, there exists an $m$-strong digraph $H$
with two distinct vertices $u,v$ such that no directed cycle of $H$
contains both $u$ and $v$.
\end{theorem}

Using Thomassen's result we can now show that no degree of strong connectivity suffices to guarantee that a digraph is unilaterally 2-linked.
\begin{theorem}\label{thm:u2lnknoK}
For every positive integer $k$, there exists a $k$-strong digraph that is
not unilaterally $2$-linked.
\end{theorem}

\begin{proof}
By Theorem~\ref{thm:thomassen}, there is a $(k+2)$-strong digraph $H$
with distinct vertices $u,v$ that lie on no common directed cycle.
Observe that $u$ and $v$ are nonadjacent. Indeed, if $uv\in A(H)$,
then a directed $(v,u)$-path together with $uv$ would give such a cycle.
The case $vu\in A(H)$ is analogous.

Let $R=H-\{u,v\}$. We construct a digraph $D$ by adding four new
vertices $s_1,t_1,s_2,t_2$ to $R$. We retain all arcs of $R$, make
$\{s_1,t_1,s_2,t_2\}$ an independent set, and specify all arcs incident
with the new vertices by
\begin{align}
N_D^+(s_1)=N_D^+(t_1)=N_H^+(u),\;
N_D^-(s_1)=N_D^-(t_1)=N_H^-(v),\label{eq:first}\\
N_D^+(s_2)=N_D^+(t_2)=N_H^+(v),\;
N_D^-(s_2)=N_D^-(t_2)=N_H^-(u).\label{eq:second}
\end{align}
Since $u$ and $v$ are nonadjacent, all four neighbourhoods of $u$ and $v$ on the right-hand sides are subsets of $V(R)$. This defines $D$ completely, and $|V(D)|=|V(H)|+2$. 

Observe that $R$ is $k$-strong and since all vertices in the set $\{s_1,s_2,t_1,t_2\}$ have at least $k+2$ in- and out-neighbours in $V(R)$, it follows from Proposition 5.8.5 in \cite{bang2009} that $D$ is $k$-strong.

We claim that the ordered sets $(s_1,s_2)$ and $(t_1,t_2)$ are not unilaterally 2-linked. 
Suppose, to the contrary, that $P_1,P_2$ are vertex-disjoint directed
paths whose endpoint sets are $\{s_1,t_1\}$ and $\{s_2,t_2\}$,
respectively. Their directions are arbitrary. Since all four new
vertices occur as endpoints of these paths, every internal vertex of
either path belongs to $R$. Moreover, each path has an internal vertex,
because its endpoints are nonadjacent.

Write $P_1=p_0p_1\cdots p_\ell$ in its directed order, so that
$\{p_0,p_\ell\}=\{s_1,t_1\}$. By~\eqref{eq:first}, we have $up_1\in A(H)$ and $p_{\ell-1}v\in A(H)$.
Thus $u\,p_1\cdots p_{\ell-1}\,v$
is a directed $(u,v)$-path in $H$, regardless of the direction in
which $P_1$ joins $s_1$ and $t_1$. Similarly, by~\eqref{eq:second},
replacing the initial and terminal vertices of $P_2$ by $v$ and $u$,
respectively, gives a directed $(v,u)$-path in $H$.

The two resulting paths have disjoint sets of internal vertices,
because $P_1$ and $P_2$ are vertex-disjoint. Their union is therefore
a directed cycle containing both $u$ and $v$, contrary to the choice
of $H$. Hence $D$ is not unilaterally $2$-linked.
\end{proof}

Because of Theorem \ref{thm:u2lnknoK} it is interesting to find classes of digraphs for which there is some degree of connectivity that guarantees unilateral 2-linkages for every choice of terminals.

Note that every semicomplete digraph on at least $2k$ vertices is unilaterally $k$-linked so the unilateral $k$-linkage problem is trivial  for semicomplete digraphs. This is by no means the case for the classes of semicomplete multipartite digraphs, split digraphs and co-bipartite digraphs. A digraph $D=(V,A)$ is {\bf semicomplete multipartite} if its underlying undirected graph is a complete multipartite graph. A digraph $D$ is a {\bf split digraph} if we can partition $V$ into two sets $V_1,V_2$ so that no arc of $D$ is inside $V_1$ and $V_2$ induces a semicomplete subdigraph of $D$.  We use the notation $D=(V_1,V_2,A)$ for a split digraph, where $V_1$ is the independent set. A split digraph $D=(V_1,V_2,A)$ is semicomplete if every vertex of $V_1$ is adjacent to every vertex of $V_2$. Semicomplete split digraphs form a subclass of the semicomplete multipartite digraphs (only one partite set may have  size more than one). Finally, a digraph $D=(V,A)$ is {\bf co-bipartite} if it has a vertex-partition $V=V_1\cup V_2$ so that each of the induced subdigraphs $D[V_1], D[V_2]$ are semicomplete.

In this paper we establish exact connectivity bounds for unilateral 2-linkages  in the case of split digraphs and semicomplete multipartite digraphs, as well as an almost tight bound for co-bipartite digraphs. By combining several of our intermediate results for some of the classes above, we  also prove that every 16-strong digraph of independence number 2 is unilaterally 2-linked. 
Problems on split digraphs are often much harder than the corresponding problem for semicomplete digraphs: for instance, by Camion's theorem \cite{camionCRASP249}, every strongly connected semicomplete digraph has a hamiltonian cycle, whereas the hamiltonian cycle problem is NP-complete for general split digraphs \cite{bangJGT108}. The hamiltonian cycle problem is polynomial for semicomplete split digraphs. The latter follows from \cite{bangJGT29} and the fact that semicomplete split digraphs are also semicomplete multipartite. As another example, the  2-linkage problem admits a polynomial-time algorithm for semicomplete multipartite digraphs (and hence also for semicomplete split digraphs) by (C), while its complexity remains open for general split digraphs. %

\begin{proposition}\label{prop:samecomplexity}
For split digraphs and fixed $k$ the $k$-linkage problem is polynomially equivalent to the 
    unilateral $k$-linkage problem.
\end{proposition}

\begin{proof}

If we have an algorithm for the normal $k$-linkage problem, we can just try the $2^k$ possibilities for a unilateral linkage between $\{s_1,s_2,\ldots{},s_k\}$ and $\{t_1,t_2,\ldots{},t_k\}$. %
We prove the other direction for the case $k=2$. The construction can easily be modified to work for general $k$.
Suppose that we have an algorithm for the unilateral 2-linkage problem. Then we can check the existence of the 2-linkage from $(s_1,s_2)$ to $(t_1,t_2)$ for a split digraph $D=(V_1,V_2,A)$ as follows:\\

We first construct a new split digraph $D'$ by adding some new vertices. Here we first consider $i=1$ and then $i=2$. 

\begin{itemize}

\item If $s_i$ is in $V_2$ we add a new vertex $s'_i$ and the arc $s'_is_i$ and add $s'_i$ to $V_1$.
\item If $s_i$ is in $V_1$, we add two new vertices $s'_i, s''_i$, the arcs $s'_is''_i, s''_is_i$ and all possible arcs from vertices in $V_2$ to $s''_i$. Add $s'_i$ to $V_1$ and $s''_i$ to $V_2$.
\item If $t_j$ is in $V_2$ we add a new vertex $t'_j$ and the arc $t_jt'_j$ and add $t'_j$ to $V_1$.
\item If $t_j$ is in $V_1$ we add two new vertices $t'_j,t''_j$, the arcs $t_jt''_j,t''_jt'_j$ and all possible arcs from $t''_j$ to $V_2$. Add $t'_j$ to $V_1$ and $t''_j$ to $V_2$.

\end{itemize}

Note that if both $s''_1$ and $s''_2$ were created above then we added the arc $s''_1s''_2$.
Similarly, if both $t''_1$ and $t''_2$ were created above, then we add the arc $t''_2t''_1$. Finally if some $t''_j$ and $s''_r$ was created, then we added the arc $t''_js''_r$.

Now  it is easy to check that the new digraph $D'$ is a split digraph and it has a pair of disjoint paths between $s'_1$ and $t'_1$ respectively between $s'_2$ and $t'_2$ if and only if the original split digraph $D$ has disjoint $(s_1,t_1)-$, $(s_2,t_2)$-paths.

The construction above also works for larger $k$. We just process the vertices $s_1,s_2,\ldots{},s_k$ in that order and then the vertices $t_1,t_2,\ldots{},t_k$ in the same order.
\end{proof}

\begin{conjecture}\label{conj:2linksplitpol}
    The unilateral 2-linkage problem (and hence also the  2-linkage problem) can be solved in polynomial time for split digraphs.
\end{conjecture}

The main results of this paper are as follows. We show by constructions that the first bounds are best possible and the third one is within one of optimal.

\begin{restatable}{theorem}{ThmUnilateralsplit}
\label{thm:mainsplit}
Every 4-strong split digraph \(D=(V_1,V_2;A)\) is {unilaterally} 2-linked. 
\end{restatable}

\begin{restatable}{theorem}{ThmUnilateralsmd}\label{the3.6}
Every $5$-strong semicomplete multipartite digraph is unilaterally
$2$-linked.
\end{restatable}

\begin{restatable}{theorem}{Thmunilateralcobip}\label{thm:6-strongcobip}
Every $6$-strong co-bipartite digraph is unilaterally
$2$-linked.
\end{restatable}

\begin{restatable}{theorem}{Thmunilateralalpha}\label{thm:alpha-two-unilateral}
    Every $16$-strong digraph $D$ with $\alpha(D)\leq2$ is unilaterally
$2$-linked.
\end{restatable}

In Section \ref{sec:prelimsec} we provide some useful notation and two important theorems on 2-linkages that we will use. In Section \ref{sec:split2linksec} we prove some new results on 2-linkages in split digraphs that will be used in the proof of Theorem \ref{thm:mainsplit}. In Sections \ref{sec:splitunilateralsec}, \ref{sec:smdsec}, \ref{sec:cobiplinksec}, \ref{sec:alpha=2linksec} we prove each of the four theorems above. %
Finally in Section \ref{sec:remarks} we mention some open problems.

\section{Preliminaries}\label{sec:prelimsec}
For an arc \((u,v)\), abbreviated as \(uv\), we say that \(u\) {\bf dominates} \(v\) (\(v\) {\bf is dominated by} \(u\)) and sometimes write \(u \to v\) to emphasize the direction of the arc. For vertex sets \(X,Y\subseteq V(D)\):
\begin{itemize}
\item \(X \to Y\) means every vertex in \(X\) dominates every vertex in \(Y\);
\item \(X \Rightarrow Y\) means \((Y,X)_D = \emptyset\);
\item \(X \mapsto Y\) means both \(X\to Y\) and \(X\Rightarrow Y\) hold.
\end{itemize}

For a vertex $v$ in digraph $D$, we use the following notation:
\[
N_D^+(v) = \{u \in V - v : vu \in A\}, \quad N_D^-(v) = \{w \in V - v : wv \in A\}.
\]
The sets $N_D^+(v)$, $N_D^-(v)$ are called the {\bf out-neighbourhood}, {\bf in-neighbourhood} of $v$.\\

A directed path $P=v_1v_2\ldots{}v_r$ in a digraph $D$ is {\bf minimal} if $D$ does not contain any arc $v_iv_j$ such that $i<j-1$.\\

\begin{Definition}\label{definition:good}
Let $D$ be a digraph, and let $s_1, t_1, s_2, t_2$ be four distinct vertices of $D$. 
We say that the triple $(D, (s_1,t_1), (s_2,t_2))$ is {\bf good} if there exist two disjoint paths $P_1,P_2$ in $D$ so that $P_i$ is either an $(s_i,t_i)$-path or a $(t_i,s_i)$-path for $i\in [2]$. Thus a digraph $D$ is unilaterally 2-linked if and only if the triple $(D, (s_1,t_1), (s_2,t_2))$ is good for every choice of  four distinct vertices $s_1,t_1,s_2,t_2$ of $D$.
\end{Definition}

We are going to use the following two results in Section \ref{sec:split2linksec} and Section \ref{sec:alpha=2linksec}

\begin{theorem}[Local linkage theorem for split digraphs]\label{split digraph:theorem}
  \cite{chen2linksplit} Let \(D=(V_1,V_2;A)\) be a split digraph and \(s_1,s_2,t_1,t_2\) be distinct vertices of \(D\). If \(D-\{s_{3-i},t_{3-i}\}\) has three internally disjoint \((s_i,t_i)\)-paths and \(D-\{s_{i},t_{i}\}\) has four internally disjoint \((s_{3-i},t_{3-i})\)-paths for \(i=1\) or \(2\), then \(D\) has a pair of disjoint \((s_1,t_1)\)- and \((s_2,t_2)\)-paths.
\end{theorem}

\begin{theorem}[Local linkage theorem for semicomplete digraphs]
\label{thm:local}\cite{bangADM41}
Let $T$ be semicomplete and let $s_1,s_2,t_1,t_2$ be distinct vertices.
If $T-\{s_2,t_2\}$ contains three internally vertex-disjoint
$(s_1,t_1)$-paths and $T-\{s_1,t_1\}$ contains two internally
vertex-disjoint $(s_2,t_2)$-paths, then $T$ contains vertex-disjoint
$(s_1,t_1)$- and $(s_2,t_2)$-paths.
\end{theorem}

We shall use the following lemma in Section \ref{sec:alpha=2linksec}.

\begin{lemma}\label{lem:degree-two-selection}
Let $G$ be a bipartite graph with parts $L,R$, where
$|L|=|R|=N\geq1$, and suppose that every vertex of $L$ has degree
at most two.
For every integer $p$ with $1\leq p\leq N$, there is a set
$L_0\subseteq L$ such that $|L_0|=p$ and
$|N_G(L_0)|\leq p+1$.
\end{lemma}

\begin{proof}
We represent the vertices of $L$ by the edges of a multigraph
$M$ on vertex set $R$.
A vertex of $L$ with two neighbours gives an edge joining them,
and a vertex with one neighbour gives a loop at that neighbour.
For a vertex with no neighbours, use a loop at an arbitrary
vertex of $R$.
Edges are kept distinct, even when their endpoints coincide.
Thus $M$ has $N$ vertices and $N$ edges.
The endpoints of each edge contain all neighbours of the
corresponding vertex of $L$.
It is therefore enough to find $p$ edges of $M$ that are incident
with at most $p+1$ vertices. Indeed, let $U$ be the set of endpoints
of such edges and let $L_0$ be the set of corresponding vertices
of $L$. Then $|L_0|=p$ and $|N_G(L_0)|\leq |U|\leq p+1$.

We order the edges of $M$ component by component, ignoring
isolated vertices.
First list the components having at least as many edges as
vertices, and then list the nontrivial tree components.
Within each component, order the edges so that every nonempty
initial segment is connected on its incident vertices.
For a component with at least two vertices, this can be done
by growing a spanning tree and then listing the remaining edges.
For a component consisting of one vertex and loops, any order
has the required property.

Let $q$ be the number of isolated vertices and $t$ the number
of nontrivial tree components.
For each component $C$ of $M$, we consider the difference
$|E(C)|-|V(C)|$. Each nontrivial tree component has one fewer edge than vertices,
and each isolated vertex has one vertex and no edges.
Thus these components contribute $-t-q$ in total to the sum
of $|E(C)|-|V(C)|$ over all components of $M$.
Since
\[
\sum_{C}
\bigl(|E(C)|-|V(C)|\bigr)
=|E(M)|-|V(M)|=0,
\]
the sum over the components listed first equals $t+q$.
Consequently, after any whole component has been listed, the
number of incident vertices is at most the number of listed
edges.
Indeed, each component in the first group has at least as many
edges as vertices, and each subsequent tree component reduces
the accumulated difference by one; there are only $t$ such trees.

Consider the first $p$ edges in this ordering.
The completed components together are incident with at most as
many vertices as they have edges.
If the next component is only partly listed, its initial segment
of $r$ edges is connected and hence uses at most $r+1$ vertices.
Thus the first $p$ edges are incident with at most $p+1$ vertices.
The corresponding $p$ vertices of $L$ give the required set $L_0$.
\end{proof}

\section{2-linkages in split digraphs}\label{sec:split2linksec}

In this section we establish some results on 2-linkages in split digraphs that complement those obtained in \cite{chen2linksplit} and which will be used in the proof of Theorem \ref{thm:mainsplit} and Theorem \ref{thm:alpha-two-unilateral}.

\begin{lemma}\label{Lemma:{s_1,t_1} in V_1 : (3,3)}
Let \(D=(V_1,V_2;A)\) be a split digraph, and let \(s_1,s_2,t_1,t_2\) be distinct vertices of \(D\) such that \(\{s_1,t_1\} \subseteq V_1\) or \(\{s_2,t_2\} \subseteq V_1\). If \(D-\{s_{3-i},t_{3-i}\}\) contains three internally disjoint \((s_i,t_i)\)-paths for each \(i \in \{1,2\}\), then \(D\) contains a pair of disjoint \((s_1,t_1)\)-path and \((s_2,t_2)\)-path.
\end{lemma}

\begin{proof}
W.l.o.g., suppose that \(\{s_1,t_1\} \subseteq V_1\).
Let \(Q_1,Q_2,Q_3\) be three internally disjoint \((s_2,t_2)\)-paths in \(D-\{s_1,t_1\}\), and let \(P_1,P_2,P_3\) be three internally disjoint \((s_1,t_1)\)-paths in \(D-\{s_2,t_2\}\).
For each \(i \in \intthree\), let \(s_1^i\) be the successor of \(s_1\) on \(P_i\), and let \(t_1^i\) be the predecessor of \(t_1\) on \(P_i\). Since \(\{s_1,t_1\} \subseteq V_1\), we have \(s_1^i, t_1^i \in {V_2}\) for all \(i \in \intthree\).
Let \(z\) be the first vertex and \(w\) be the last vertex at which \(Q_1\) intersects \(V(P_1)\cup V(P_2)\cup V(P_3)\).
Then \(z \in V(P_{i'})\) and \(w \in V(P_{j'})\) for some indices \(i',j' \in\intthree\), where \(i'\) and \(j'\) may be identical.

If \(t_1^{i'} \to s_1^{j'}\) is an arc in \(D\), then the \((s_2,t_2)\)-path
\(
Q=Q_1[s_2,z] \circ P_{i'}[z,t_1^{i'}] \circ t_1^{i'}s_1^{j'} \circ P_{j'}[s_1^{j'},w] \circ Q_1[w,t_2]
\)
is disjoint from \(P_i\) for any \(i \in \{1,2,3\}\setminus\{i',j'\}\).
Otherwise, \(s_1^{j'} \to t_1^{i'}\) is an arc in \(D\), and then the path \(P = s_1 \to s_1^{j'} \to t_1^{i'} \to t_1\) is an \((s_1,t_1)\)-path that is disjoint from \(Q_j\) for some \(j \in \intthree\).
This completes the proof.
\end{proof}

\begin{lemma}\label{Lemma:{s_1,s_2} in V_1 : (3,3)}
    Let \(D=(V_1,V_2;A)\) be a split digraph, and let \(s_1,s_2,t_1,t_2\) be distinct vertices of \(D\) such that \(\{s_1,s_2\} \subseteq V_1\) and \(\{t_1,t_2\} \subseteq V_2\). If \(D-\{s_{3-i},t_{3-i}\}\) contains three internally disjoint \((s_i,t_i)\)-paths for each \(i \in \{1,2\}\), then {$D$ has a pair of disjoint paths $P_1,P_2$ such that $P_i$ is an $(s_i,t_i)$-path. In particular \((D,\{s_1,t_1\},\{s_2,t_2\})\) is good}. 
\end{lemma} 

\begin{proof}
On the contrary, suppose that
\begin{equation}\label{eq:contradiction-1}
\text{$D$ contains no vertex-disjoint
$(s_1,t_1)$- and $(s_2,t_2)$-paths.}
\end{equation}
Let $P_1,P_2,P_3$ be three internally disjoint $(s_1,t_1)$-paths in $D-\{s_2,t_2\}$ chosen to be minimal. Similarly, let $Q_1,Q_2,Q_3$ be three internally disjoint minimal $(s_2,t_2)$-paths in $D-\{s_1,t_1\}$. Clearly, each of \(P_i\), \(i \in \intthree\) intersects all of \(Q_j\) for \(j \in \intthree\). Let $S = V(P_1)\cup V(P_2)\cup V(P_3)$. For each $i\in \intthree$, let $z_i$ be the first vertex of $Q_i$ intersecting $S$ and $w_i$ the last such vertex. Furthermore, let \(s_1^i(s_2^i)\) be the successor of \(s_1(s_2)\) on \(P_i(Q_i)\) for \(i \in \intthree\), and let \(t_1^i(t_2^i)\) be the predecessor of \(t_1(t_2)\) on \(P_i(Q_i)\) for \(i \in \intthree\).%

\begin{claim}\label{claim:t_1^j in V_1}
If \(z_i\) lies on \(P_j\) for some \(i,j\in \intthree\), then \(t_1^j \in V_1\).
\end{claim}

\begin{proof}
Suppose, for contradiction, that $t_1^j \in V_2$. 
Let $w_i$ lie on $P_k$ for some $k\in \intthree$. Since $s_1\in V_1$, it follows that $s_1^k\in V_2$. If $t_1^j\to s_1^k$, then the $(s_2,t_2)$-path
\[
Q_i[s_2,z_i] \circ P_j[z_i,t_1^j] \circ t_1^js_1^k \circ P_k[s_1^k,w_i] \circ Q_i[w_i,t_2]
\]
is disjoint from $P_l$ for some $l\in \intthree\setminus\{j,k\}$, contradicting \eqref{eq:contradiction-1}. Consequently, we must have $s_1^k\to t_1^j$. However, the $(s_1,t_1)$-path $s_1s_1^kt_1^jt_1$ is disjoint from $Q_l$ for some $l\in \intthree$, which again contradicts \eqref{eq:contradiction-1}. 
\end{proof}

\vspace{2mm}
\noindent \mycase\label{Lemma:case1} There exist distinct indices \(i_1,j_1 \in \intthree\) such that \(z_{i_1}, z_{j_1} \in V(P_i)\) for some \(i \in \intthree\).
\vspace{2mm}

W.l.o.g., let \(i=1\). Let \(w_1\) lie on \(P_{k_1}\) for some \(k_1 \in \intthree\), and let \(v\) be the predecessor of \(t_1^1\) on \(P_1\). By Claim \ref{claim:t_1^j in V_1}, we have \(t_1^1 \in V_1\), so \(v \in V_2\). Since \(s_1 \in V_1\), it follows that \(s_1^{k_1} \in V_2\).
If \(v \to s_1^{k_1}\), then there exists an index \(i \in \{i_1,j_1\}\) such that the \((s_2,t_2)\)-path
\(
Q_i[s_2,z_i] \circ P_1[z_i,v] \circ vs_1^{k_1} \circ P_{k_1}[s_1^{k_1},w_1] \circ Q_1[w_1,t_2]
\)
is disjoint from \(P_j\) for some \(j \in \intthree \setminus \{1,k_1\}\). Consequently, we must have  \(s_1^{k_1} \to v\) {and then \(P = s_1 s_1^{k_1} v t_1^1 t_1\) is an \((s_1, t_1)\)-path.} (See Figure \ref{Figure-1}).

\begin{figure}[H]
    \centering
\tikzset{pics/manuscriptdirection/.style={code={
  \fill[black] (0pt,0pt) -- (-6pt,2pt) -- (-6pt,-2pt) -- cycle;
}}}

\begin{tikzpicture}[
    scale=0.8,
    transform shape,
    font=\large,
    line cap=round,
    line join=round,
    ppath/.style={draw=pathgray,line width=1.2pt},
    qpath/.style={draw=qgray,line width=1.2pt},
    selected/.style={draw=highlightdark,line width=1.2pt},
    arrow/.style={-{Triangle[length=6pt,width=3pt]},shorten >=0pt},
    starvertex/.style={shape=star,star points=5,star point ratio=2.35,
        draw=none,fill=starfill,minimum size=12pt,inner sep=0pt},
    diamondvertex/.style={diamond,draw=none,fill=diamondfill,
        minimum size=7.5pt,inner sep=0pt},
    blackvertex/.style={circle,draw=black,fill=black,
        minimum size=5.5pt,inner sep=0pt},
    lab/.style={inner sep=1pt}
]

\coordinate (s1)  at ( 0.00,6.50);
\coordinate (t1)  at ( 0.00,0.00);
\coordinate (s2)  at (-4.10,3.25);
\coordinate (t2)  at ( 3.75,3.25);
\coordinate (s11) at (-0.79,5.72);
\coordinate (s12) at ( 0.00,5.68);
\coordinate (s13) at ( 0.78,5.68);
\coordinate (zi)  at (-1.61,3.02);
\coordinate (zj)  at (-1.38,1.88);
\coordinate (v)   at (-1.05,1.17);
\coordinate (t11) at (-0.56,0.52);
\coordinate (w1)  at ( 1.59,3.32);

\draw[highlightgray,line width=4.5pt]
    (s1) .. controls (0.30,6.24) and (0.56,5.99) .. (s13)
    -- (v)
    .. controls (-0.92,0.93) and (-0.73,0.69) .. (t11)
    .. controls (-0.40,0.35) and (-0.19,0.16) .. (t1);

\draw[ppath,arrow]
    (s1) .. controls (-0.30,6.24) and (-0.55,6.00) .. (s11);
\draw[ppath]
    (s11) .. controls (-1.25,5.11) and (-1.72,3.98) .. (zi)
    pic[pos=0.50,sloped,allow upside down] {manuscriptdirection}
    .. controls (-1.61,2.61) and (-1.52,2.18) .. (zj)
    .. controls (-1.33,1.64) and (-1.17,1.33) .. (v);

\draw[ppath]
    (s13) .. controls (1.22,5.08) and (1.62,4.24) .. (w1)
    pic[pos=0.48,sloped,allow upside down] {manuscriptdirection}
    .. controls (1.69,1.88) and (0.96,0.76) .. (t1)
    pic[pos=0.80,sloped,allow upside down] {manuscriptdirection};

\draw[selected,arrow]
    (s1) .. controls (0.30,6.24) and (0.56,5.99) .. (s13);
\draw[selected,arrow] (s13) -- (v);
\draw[selected,arrow]
    (v) .. controls (-0.92,0.93) and (-0.73,0.69) .. (t11);
\draw[selected,arrow]
    (t11) .. controls (-0.40,0.35) and (-0.19,0.16) .. (t1);

\draw[ppath,arrow] (s1) -- (s12);
\draw[ppath] (s12) -- (t1)
    pic[pos=0.88,sloped,allow upside down] {manuscriptdirection};

\draw[qpath]
    (s2) .. controls (-3.10,3.68) and (-2.76,3.60) .. (zi)
    pic[pos=0.35,sloped,allow upside down] {manuscriptdirection};
    \draw[qpath]
    (s2) .. controls (-3.93,2.49) and (-3.68,2.36) ..
    (-2.95,2.43)
    pic[pos=0.90,sloped,allow upside down]
    {manuscriptdirection}
    .. controls (-2.28,2.51) and (-1.87,2.28) .. (zj);
\draw[qpath]
    (w1) .. controls (2.18,3.68) and (2.43,3.81) .. (2.75,3.11)
    .. controls (2.93,2.63) and (3.23,2.82) .. (t2)
    pic[pos=0.76,sloped,allow upside down] {manuscriptdirection};

\foreach \p in {s1,s2,t11}
    \node[starvertex] at (\p) {};
\foreach \p in {s11,s12,s13,v,t1,t2}
    \node[diamondvertex] at (\p) {};
\foreach \p in {zi,zj,w1}
    \node[blackvertex] at (\p) {};

\node[lab,above=5pt] at (s1) {$s_1$};
\node[lab,below=5pt] at (t1) {$t_1$};
\node[lab,left=8pt]  at (s2) {$s_2$};
\node[lab,right=5pt] at (t2) {$t_2$};
\node[lab,left=6pt]  at (s11) {$s_1^1$};
\node[lab,left=5pt]  at (s12) {$s_1^2$};
\node[lab,right=5pt] at (s13) {$s_1^3$};
\node[lab,left=5pt]  at (zi) {$z_{i_1}$};
\node[lab,left=5pt]  at (zj) {$z_{j_1}$};
\node[lab,left=5pt]  at (w1) {$w_1$};
\node[lab,left=5pt]  at (v) {$v$};
\node[lab,left=6pt,yshift=-2pt] at (t11) {$t_1^1$};

\node[lab] at (-0.65,1.00) {$P_1$};
\node[lab] at ( 0.24,1.00) {$P_2$};
\node[lab] at ( 1.35,0.98) {$P_3$};
\node[lab] at (-2.96,3.83) {$Q_{i_1}$};
\node[lab] at (-3.13,2.06) {$Q_{j_1}$};

\node[starvertex]  at (2.88,6.11) {};
\node[diamondvertex] at (2.88,5.59) {};
\node[lab,anchor=west] at (3.15,6.11) {$V_1$};
\node[lab,anchor=west] at (3.15,5.59) {$V_2$};

\end{tikzpicture}

\caption{Example with \({k_1=3}\). {The path $P$ is highlighted.} Stars represent vertices in \(V_1\), diamonds represent vertices in \(V_2\) and filled circles represent a vertex whose location in $V_1,V_2$ is not determined. This rule is used for all following figures (without further mentioning).}
\label{Figure-1}
\end{figure}
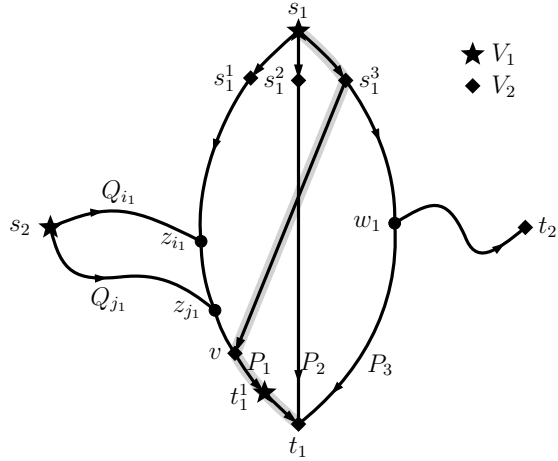

By \eqref{eq:contradiction-1}, each $Q_i$, $i\in [3]$ must intersect $P$. Thus, we may assume \(|V(Q_i) \cap \{t_1^1,v,s_1^{k_1}\}| = 1\) for all \(i \in \intthree\). %
W.l.o.g., {assume that} \(v \in V(Q_1)\), \(t_1^1 \in V(Q_2)\) and \(s_1^{k_1} \in V(Q_3)\). {We} now analyze the properties of several subpaths.

\begin{claim}\label{Claim:v(Q_1)[v,t_2] cap V(P_i) notempty}
\begin{align*}
& V(Q_1[v,t_2]) \cap V(P_i) \neq \emptyset \quad \text{for every } i\in \intthree \setminus \{1\},\\
& V(Q_2[t_1^1,t_2]) \cap V(P_i) \neq \emptyset \quad \text{for every } i\in \intthree \setminus \{1\},\\
& V(Q_3[s_2,s_1^{k_1}]) \cap V(P_i) \neq \emptyset \quad \text{for every } i\in \intthree \setminus \{k_1\}.
\end{align*}
\end{claim}

\begin{proof}
Suppose that there exists some \(i \in \intthree \setminus \{1\}\) satisfying
\(
V(Q_1[v,t_2]) \cap V(P_i) = \emptyset.
\)
Then there exists \(j \in \{i_1,j_1\}\) such that the \((s_2,t_2)\)-path
\(
Q_j[s_2,z_j] \circ P_1[z_j,v] \circ Q_1[v,t_2]
\)
is disjoint from \(P_i\). This contradicts \eqref{eq:contradiction-1}.
Hence, \(V(Q_1[v,t_2]) \cap V(P_i) \neq \emptyset\) for all \(i\in \intthree \setminus \{1\}\).

{An identical argument with $v$ replaced by $t_1^1$ shows that } %
\(V(Q_2[t_1^1,t_2]) \cap V(P_i) \neq \emptyset\) for all \(i\in \intthree \setminus \{1\}\).
Finally, assume that there exists some \(i \in \intthree \setminus \{k_1\}\) with
\(
V(Q_3[s_2,s_1^{k_1}]) \cap V(P_i) = \emptyset.
\)
Then the \((s_2,t_2)\)-path
\(
Q_3[s_2,s_1^{k_1}] \circ P_{k_1}[s_1^{k_1},w_1] \circ Q_1[w_1,t_2]
\)
is disjoint from \(P_i\), contradicting \eqref{eq:contradiction-1}.
We conclude that \(V(Q_3[s_2,s_1^{k_1}]) \cap V(P_i) \neq \emptyset\) for all \(i\in \intthree \setminus \{k_1\}\).
\end{proof}

{Let $i'$ be an index in $\{2,3\}-\{k_1\}$}. By Claim \ref{Claim:v(Q_1)[v,t_2] cap V(P_i) notempty}, %
{we can choose vertices \(x_1\), \(x_2\) and \(x_3\)} such that
\[
x_1 \in V(Q_1[v,t_2]) \cap V(P_{i'}),\;
x_2 \in V(Q_2[t_1^1,t_2]) \cap V(P_{i'}),\;
x_3 \in V(Q_3[s_2,s_1^{k_1}]) \cap V(P_{i'}).
\]
We now {determine}  the order of vertices \(x_1,x_2,x_3\) along the path \(P_{i'}\).
If there exists an index \({a} \in \{1,2\}\) %
such that \(|P_{i'}[s_1,x_{{a}}]| > |P_{i'}[s_1,x_3]|\), then the \((s_2,t_2)\)-path \(Q_3[s_2,x_3] \circ P_{i'}[x_3,x_{{a}}] \circ Q_{{a}}[x_{{a}},t_2]\) is disjoint from \(P\), contradicting \eqref{eq:contradiction-1}, so we conclude 

\[
\max\{|P_{i'}[s_1,x_1]|, |P_{i'}[s_1,x_2]|\} < |P_{i'}[s_1,x_3]|.
\]
For convenience, we assume that \(|P_{i'}[s_1,x_1]| < |P_{i'}[s_1,x_2]|\). Thus, we obtain \(
|P_{i'}[s_1,x_1]| < |P_{i'}[s_1,x_2]| < |P_{i'}[s_1,x_3]|.\)

\begin{claim}\label{claim:x_3 in V_1}
\(x_3 \in V_1\) and \(P_{i'}[x_3,t_1]=x_3t_1\).
\end{claim}

\begin{proof}
Suppose \(x_3 \notin V_1\), so \(x_3 \in V_2\). Since \(s_1^{i'} \in V_2\) and by the minimality of \(P_{i'}\), we have \(x_3 \to s_1^{i'}\).
Consider the \((s_2,t_2)\)-path
\(
Q_3[s_2,x_3] \circ x_3s_1^{i'} \circ P_{i'}[s_1^{i'},x_1] \circ Q_1[x_1,t_2].
\)
This path is disjoint from \(P\), which contradicts \eqref{eq:contradiction-1} (See Figure \ref{fig2a:lemma}).
Hence, \(x_3 \in V_1\). Now suppose \(|P_{i'}[x_3,t_1]| \ge 3\). Let \(x_3^+\) denote the successor of \(x_3\) on \(P_{i'}[x_3,t_1]\); clearly \(x_3^+ \in V_2\). By the minimality of \(P_{i'}\), we obtain \(x_3^+ \to s_1^{i'}\).
Then the \((s_2,t_2)\)-path
\[
Q_3[s_2,x_3] \circ x_3x_3^+ \circ x_3^+s_1^{i'} \circ P_{i'}[s_1^{i'},x_1] \circ Q_1[x_1,t_2]
\]
is disjoint from \(P\), contradicting \eqref{eq:contradiction-1} (See Figure \ref{fig2b:lemma}).
Consequently, \(P_{i'}[x_3,t_1]\) is exactly the arc \(x_3t_1\).
\end{proof}

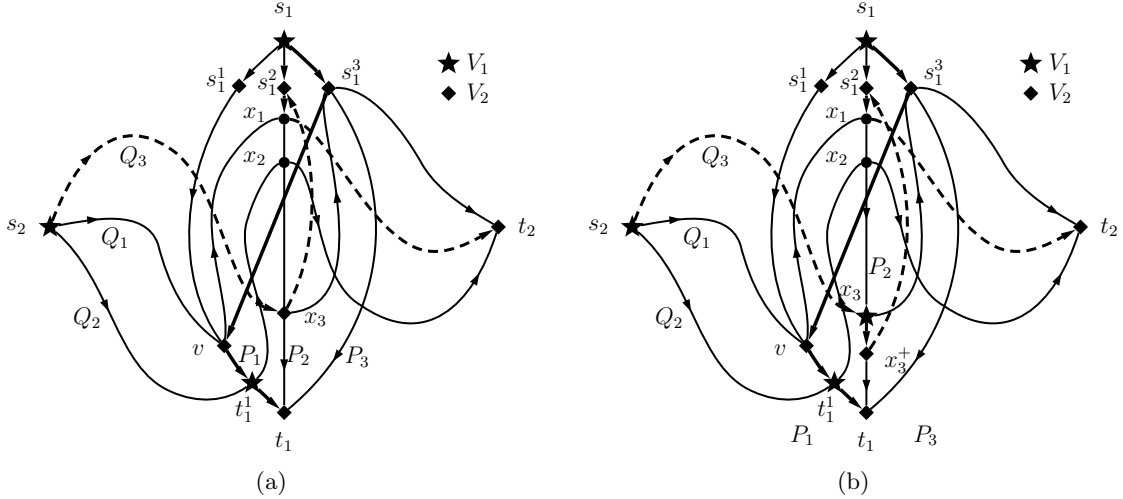
\begin{figure}[H]
\centering
\tikzset{pics/manuscriptdirection/.style={code={
  \fill[black] (0pt,0pt) -- (-6pt,2pt) -- (-6pt,-2pt) -- cycle;
}}}

\begin{subfigure}{0.45\textwidth}
\centering
\resizebox{\linewidth}{!}{%
\begin{tikzpicture}[
    scale=0.8,                 
    transform shape,
    font=\large,
    line cap=round,
    line join=round,
    pathline/.style={draw=black,line width=0.8pt},
    solidroute/.style={draw=black,line width=1.4pt,solid},
    dashedroute/.style={draw=black,line width=1.2pt,
        dash pattern=on 4pt off 2.5pt,line cap=butt},
    arrow/.style={-{Triangle[length=6pt,width=3pt,line width=0pt]},shorten >=0pt},
    starvertex/.style={shape=star,star points=5,star point ratio=2.35,
        draw=none,fill=black,minimum size=12pt,inner sep=0pt,outer sep=0pt},
    diamondvertex/.style={diamond,draw=none,fill=black,
        minimum size=7.5pt,inner sep=0pt,outer sep=0pt},
    blackvertex/.style={circle,draw=none,fill=black,
        minimum size=5.5pt,inner sep=0pt,outer sep=0pt},
    lab/.style={inner sep=1pt}
]

\node[starvertex]    (s1)  at ( 0.00,6.50) {};
\node[diamondvertex] (t1)  at ( 0.00,0.00) {};
\node[starvertex]    (s2)  at (-4.10,3.25) {};
\node[diamondvertex] (t2)  at ( 3.75,3.25) {};
\node[diamondvertex] (s11) at (-0.79,5.72) {};
\node[diamondvertex] (s12) at ( 0.00,5.68) {};
\node[diamondvertex] (s13) at ( 0.78,5.68) {};
\node[diamondvertex] (v)   at (-1.05,1.17) {};
\node[starvertex]    (t11) at (-0.56,0.52) {};
\node[blackvertex]   (x1)  at ( 0.00,5.14) {};
\node[blackvertex]   (x2)  at ( 0.00,4.38) {};
\node[diamondvertex] (x3) at (0.00,1.74) {};

\begin{scope}[on background layer]

\draw[pathline,arrow]
    (s1) .. controls (-0.30,6.24) and (-0.55,6.00) .. (s11);
\draw[pathline]
    (s11) .. controls (-1.85,4.24) and (-1.92,2.33) .. (v)
    pic[pos=0.42,sloped,allow upside down] {manuscriptdirection};
\draw[pathline,arrow] (s1) -- (s12);
\draw[dashedroute,arrow] (s12) -- (x1);
\draw[pathline] (x1) -- (t1)
    pic[pos=0.88,sloped,allow upside down] {manuscriptdirection};
\draw[pathline]
    (s13) .. controls (2.12,3.75) and (1.76,1.46) .. (t1)
    pic[pos=0.82,sloped,allow upside down] {manuscriptdirection};

\draw[pathline]
    (s2) .. controls (-3.38,3.38) and (-2.44,3.59) .. (-2.21,2.95)
    pic[pos=0.36,sloped,allow upside down] {manuscriptdirection}
    .. controls (-2.03,2.25) and (-1.75,1.66) .. (v);
\draw[pathline]
    (v) .. controls (-0.94,2.09) and (-1.63,3.39) .. (-1.01,4.35)
    pic[pos=0.55,sloped,allow upside down] {manuscriptdirection}
    .. controls (-0.67,4.88) and (-0.25,5.16) .. (x1);
\draw[dashedroute,arrow] (x1)
    .. controls (0.64,5.10) and (1.27,3.58) .. (2.04,3.04)
    .. controls (2.48,2.70) and (2.91,2.79) .. (t2);

\draw[pathline]
    (s2) .. controls (-3.03,2.32) and (-2.95,0.78) .. (-1.99,0.37)
    pic[pos=0.45,sloped,allow upside down] {manuscriptdirection}
    .. controls (-1.51,0.15) and (-1.02,0.19) .. (t11);
\draw[pathline]
    (t11) .. controls (0.12,1.14) and (-0.92,2.65) .. (-0.64,3.39)
    pic[pos=0.55,sloped,allow upside down] {manuscriptdirection}
    .. controls (-0.50,3.87) and (-0.22,4.32) .. (x2);
\draw[pathline]
    (x2) .. controls (0.63,4.37) and (0.47,2.37) .. (1.34,1.88)
    pic[pos=0.52,sloped,allow upside down] {manuscriptdirection}
    .. controls (2.46,1.10) and (3.30,1.84) .. (t2)
    pic[pos=0.80,sloped,allow upside down] {manuscriptdirection};

\draw[dashedroute,arrow] (s2)
    .. controls (-3.47,4.91) and (-2.66,5.14) .. (-1.90,4.55)
    pic[pos=0.35,sloped,allow upside down] {manuscriptdirection}
    .. controls (-1.12,3.91) and (-0.93,2.03) .. (x3);
\draw[pathline]
    (x3) .. controls (1.76,1.91) and (0.46,5.27) .. (s13)
    pic[pos=0.55,sloped,allow upside down] {manuscriptdirection};
\draw[pathline]
    (s13) .. controls (1.51,5.71) and (2.11,4.90) .. (2.42,4.34)
    .. controls (2.73,3.82) and (3.34,3.49) .. (t2)
    pic[pos=0.75,sloped,allow upside down] {manuscriptdirection};


\draw[dashedroute,arrow] (x3) .. controls (0.73,3.11) and (0.50,4.68) .. (s12);

\draw[solidroute,arrow] (s1) .. controls (0.30,6.24) and (0.56,5.99) .. (s13);
\draw[solidroute,arrow] (s13) -- (v);
\draw[solidroute,arrow]
    (v) .. controls (-0.92,0.93) and (-0.73,0.69) .. (t11);
\draw[solidroute,arrow]
    (t11) .. controls (-0.40,0.35) and (-0.19,0.16) .. (t1);
\end{scope}

\node[lab,above=5pt] at (s1.north) {$s_1$};
\node[lab,below=5pt] at (t1.south) {$t_1$};
\node[lab,left=7pt] at (s2.west) {$s_2$};
\node[lab,right=5pt] at (t2.east) {$t_2$};
\node[lab,left=5pt,yshift=3pt] at (s11.center) {$s_1^1$};
\node[lab,left=2pt,yshift=3pt] at (s12.center) {$s_1^2$};
\node[lab,right=5pt,yshift=8pt] at (s13.center) {$s_1^3$};
\node[lab,left=5pt,yshift=2pt] at (x1.west) {$x_1$};
\node[lab,left=5pt,yshift=2pt] at (x2.west) {$x_2$};
\node[lab,left=5pt,yshift=-2pt] at (v.west) {$v$};
\node[lab] at (-0.70,0.06) {$t_1^1$};

\node[lab] at (-2.96,3.05) {$Q_1$};
\node[lab] at (-3.45,1.65) {$Q_2$};
\node[lab] at (-2.64,4.48) {$Q_3$};

\node[starvertex] at (2.88,6.11) {};
\node[diamondvertex] at (2.88,5.59) {};
\node[lab,anchor=west] at (3.15,6.11) {$V_1$};
\node[lab,anchor=west] at (3.15,5.59) {$V_2$};

\node[lab,right=5pt,yshift=-3pt] at (x3.east) {$x_3$};
\node[lab] at (-0.60,1.00) {$P_1$};
\node[lab] at ( 0.24,1.00) {$P_2$};
\node[lab] at ( 1.28,0.98) {$P_3$};
\end{tikzpicture}%
}
\caption{}
\label{fig2a:lemma}
\end{subfigure}
\hspace{4mm}
\begin{subfigure}{0.45\textwidth}
\centering
\resizebox{\linewidth}{!}{%
\begin{tikzpicture}[
    scale=0.8,                 
    transform shape,
    font=\large,
    line cap=round,
    line join=round,
    pathline/.style={draw=black,line width=0.8pt},
    solidroute/.style={draw=black,line width=1.4pt,solid},
    dashedroute/.style={draw=black,line width=1.2pt,
        dash pattern=on 4pt off 2.5pt,line cap=butt},
    arrow/.style={-{Triangle[length=6pt,width=3pt,line width=0pt]},shorten >=0pt},
    starvertex/.style={shape=star,star points=5,star point ratio=2.35,
        draw=none,fill=black,minimum size=12pt,inner sep=0pt,outer sep=0pt},
    diamondvertex/.style={diamond,draw=none,fill=black,
        minimum size=7.5pt,inner sep=0pt,outer sep=0pt},
    blackvertex/.style={circle,draw=none,fill=black,
        minimum size=5.5pt,inner sep=0pt,outer sep=0pt},
    lab/.style={inner sep=1pt}
]

\node[starvertex]    (s1)  at ( 0.00,6.50) {};
\node[diamondvertex] (t1)  at ( 0.00,0.00) {};
\node[starvertex]    (s2)  at (-4.10,3.25) {};
\node[diamondvertex] (t2)  at ( 3.75,3.25) {};
\node[diamondvertex] (s11) at (-0.79,5.72) {};
\node[diamondvertex] (s12) at ( 0.00,5.68) {};
\node[diamondvertex] (s13) at ( 0.78,5.68) {};
\node[diamondvertex] (v)   at (-1.05,1.17) {};
\node[starvertex]    (t11) at (-0.56,0.52) {};
\node[blackvertex]   (x1)  at ( 0.00,5.14) {};
\node[blackvertex]   (x2)  at ( 0.00,4.38) {};
\node[starvertex] (x3) at (0.00,1.68) {};
\node[diamondvertex] (x3plus) at (0.00,1.03) {};

\begin{scope}[on background layer]

\draw[pathline,arrow]
    (s1) .. controls (-0.30,6.24) and (-0.55,6.00) .. (s11);
\draw[pathline]
    (s11) .. controls (-1.85,4.24) and (-1.92,2.33) .. (v)
    pic[pos=0.42,sloped,allow upside down] {manuscriptdirection};
\draw[pathline,arrow] (s1) -- (s12);
\draw[dashedroute,arrow] (s12) -- (x1);
\draw[pathline] (x1) -- (x3)
    pic[pos=0.53,sloped,allow upside down] {manuscriptdirection};
\draw[pathline] (x3plus) -- (t1)
    pic[pos=0.70,sloped,allow upside down] {manuscriptdirection};
\draw[pathline]
    (s13) .. controls (2.12,3.75) and (1.76,1.46) .. (t1)
    pic[pos=0.82,sloped,allow upside down] {manuscriptdirection};

\draw[pathline]
    (s2) .. controls (-3.38,3.38) and (-2.44,3.59) .. (-2.21,2.95)
    pic[pos=0.36,sloped,allow upside down] {manuscriptdirection}
    .. controls (-2.03,2.25) and (-1.75,1.66) .. (v);
\draw[pathline]
    (v) .. controls (-0.94,2.09) and (-1.63,3.39) .. (-1.01,4.35)
    pic[pos=0.55,sloped,allow upside down] {manuscriptdirection}
    .. controls (-0.67,4.88) and (-0.25,5.16) .. (x1);
\draw[dashedroute,arrow] (x1)
    .. controls (0.64,5.10) and (1.27,3.58) .. (2.04,3.04)
    .. controls (2.48,2.70) and (2.91,2.79) .. (t2);

\draw[pathline]
    (s2) .. controls (-3.03,2.32) and (-2.95,0.78) .. (-1.99,0.37)
    pic[pos=0.45,sloped,allow upside down] {manuscriptdirection}
    .. controls (-1.51,0.15) and (-1.02,0.19) .. (t11);
\draw[pathline]
    (t11) .. controls (0.12,1.14) and (-0.92,2.65) .. (-0.64,3.39)
    pic[pos=0.55,sloped,allow upside down] {manuscriptdirection}
    .. controls (-0.50,3.87) and (-0.22,4.32) .. (x2);
\draw[pathline]
    (x2) .. controls (0.63,4.37) and (0.47,2.37) .. (1.34,1.88)
    pic[pos=0.52,sloped,allow upside down] {manuscriptdirection}
    .. controls (2.46,1.10) and (3.30,1.84) .. (t2)
    pic[pos=0.80,sloped,allow upside down] {manuscriptdirection};

\draw[dashedroute,arrow] (s2)
    .. controls (-3.47,4.91) and (-2.66,5.14) .. (-1.90,4.55)
    pic[pos=0.35,sloped,allow upside down] {manuscriptdirection}
    .. controls (-1.12,3.91) and (-0.93,2.03) .. (x3);
\draw[pathline]
    (x3) .. controls (1.76,1.91) and (0.46,5.27) .. (s13)
    pic[pos=0.55,sloped,allow upside down] {manuscriptdirection};
\draw[pathline]
    (s13) .. controls (1.51,5.71) and (2.11,4.90) .. (2.42,4.34)
    .. controls (2.73,3.82) and (3.34,3.49) .. (t2)
    pic[pos=0.75,sloped,allow upside down] {manuscriptdirection};

\draw[dashedroute,arrow] (x3) -- (x3plus);
\draw[dashedroute,arrow] (x3plus) .. controls (1.06,2.53) and (0.77,4.39) .. (s12);

\draw[solidroute,arrow] (s1) .. controls (0.30,6.24) and (0.56,5.99) .. (s13);
\draw[solidroute,arrow] (s13) -- (v);
\draw[solidroute,arrow]
    (v) .. controls (-0.92,0.93) and (-0.73,0.69) .. (t11);
\draw[solidroute,arrow]
    (t11) .. controls (-0.40,0.35) and (-0.19,0.16) .. (t1);
\end{scope}

\node[lab,above=5pt] at (s1.north) {$s_1$};
\node[lab,below=5pt] at (t1.south) {$t_1$};
\node[lab,left=7pt] at (s2.west) {$s_2$};
\node[lab,right=5pt] at (t2.east) {$t_2$};
\node[lab,left=5pt,yshift=3pt] at (s11.center) {$s_1^1$};
\node[lab,left=2pt,yshift=3pt] at (s12.center) {$s_1^2$};
\node[lab,right=5pt,yshift=8pt] at (s13.center) {$s_1^3$};
\node[lab,left=5pt,yshift=2pt] at (x1.west) {$x_1$};
\node[lab,left=5pt,yshift=2pt] at (x2.west) {$x_2$};
\node[lab,left=5pt,yshift=-2pt] at (v.west) {$v$};
\node[lab] at (-0.70,0.06) {$t_1^1$};

\node[lab] at (-2.96,3.05) {$Q_1$};
\node[lab] at (-3.45,1.65) {$Q_2$};
\node[lab] at (-2.64,4.48) {$Q_3$};

\node[starvertex] at (2.88,6.11) {};
\node[diamondvertex] at (2.88,5.59) {};
\node[lab,anchor=west] at (3.15,6.11) {$V_1$};
\node[lab,anchor=west] at (3.15,5.59) {$V_2$};

\node[lab] at (-0.28,2.1) {$x_3$};
\node[lab,right=4pt,yshift=-3pt] at (x3plus.east) {$x_3^+$};
\node[lab] at (-1.12,-0.40) {$P_1$};
\node[lab] at ( 0.24,2.54) {$P_2$};
\node[lab] at ( 1.04,-0.40) {$P_3$};
\end{tikzpicture}%
}
\caption{}
\label{fig2b:lemma}
\end{subfigure}
\caption{Example with \(k_1 = 3\) and \(i' = 2\): (a) \(x_3 \in V_2\); (b) \(x_3 \in V_1\) and \(|P_{i'}[x_3,t_1]| \geq 3\).}
\label{fig2:lemma}
\end{figure}

 Suppose for contradiction that there exists another vertex \(x_3'\in V(Q_3[s_2,x_3])\cap V(P_{i'})\). {By Claim \ref{claim:x_3 in V_1} we must have} \(|P_{i'}[s_1,x_3']| < |P_{i'}[s_1,x_3]|\). {However, since we could have chosen $x_3'$ instead of $x_3$, we obtain a contradiction to Claim \ref{claim:x_3 in V_1}}. Hence $x_3$ is the only vertex of $P_{i'}$ that is also on $V(Q_3[s_2,x_3])$.

Let \(x_3^-\) denote the predecessor of \(x_3\) along \(Q_3[s_2,x_3]\). Consequently, \(x_3^- \in V_2\). Recall that \(s_1^{k_1} \in V_2\).
If \(s_1^{k_1} \to x_3^-\), then the \((s_1,t_1)\)-path \(s_1 s_1^{k_1}  x_3^-  x_3  t_1\) is disjoint from \(Q_i\) for all \(i \in \inttwo\). Hence \(x_3^- \to s_1^{k_1}\), and now the \((s_2,t_2)\)-path
\[
Q_3[s_2,x_3^-] \circ x_3^-s_1^{k_1} \circ P_{k_1}[s_1^{k_1},w_1] \circ Q_1[w_1,t_2]
\]
is disjoint from $P_{i'}$ (recall that we showed above that $x_3$ is the only vertex of $P_{i'}$ that is also on $V(Q_3[s_2,x_3])$).

\vspace{2mm}
\noindent \mycase\; Precisely one of the vertices \(z_1,z_2,z_3\) belong to $P_r$ for each $r\in [3]$.
\vspace{2mm}

W.l.o.g., we assume \(z_1 \in V(P_1)\), \(z_2 \in V(P_2)\) and \(z_3 \in V(P_3)\). By Claim \ref{claim:t_1^j in V_1}, \(t_1^1,t_1^2,t_1^3 \in V_1\). Let \(w_1 \in V(P_{i''})\) for some \(i'' \in \intthree\), and let \(v\) be the predecessor of \(t_1^1\) on \(P_1\).

\begin{claim}\label{claim:z_1=t_1^1}
    \(z_1 = t_1^1\).
\end{claim}

\begin{proof}
    Suppose for contradiction that \(z_1 \neq t_1^1\). Since \(t_1^1, s_1 \in V_1\), we clearly have \(v, s_1^{i''} \in V_2\). {We must have \(s_1^{i''} \to v\) since otherwise} the \((s_2,t_2)\)-path
    \(
    Q_1[s_2,z_1] \circ P_1[z_1,v] \circ vs_1^{i''} \circ P_{i''}[s_1^{i''},w_1] \circ Q_1[w_1,t_2]
    \)
    is disjoint from \(P_i\) for some \(i \in \intthree \setminus \{1,i''\}\).  Define the \((s_1,t_1)\)-path
    \[
    {P^*} = s_1 s_1^{i''} v t_1^1 t_1.
    \]
    The argument proceeds similarly to the proof of Case \ref{Lemma:case1}. Replacing \(i''\) with \(k_1\) transforms the {path $P^*$ above } into the corresponding path {$P$} in Case \ref{Lemma:case1}. This again yields a contradiction to \eqref{eq:contradiction-1}, so we conclude \(z_1 = t_1^1\).
       Note that in the proof in Case \ref{Lemma:case1} we  did exploit that there were two of the vertices $z_1,z_2,z_3$ on the same path but that was only used to make sure one of these vertices could reach  the vertex $v$ along $P_1$. By the assumption that $z_1\neq t_1^1$, the vertex $z_1$ can, in fact, reach $v$ along $P_1$ and now the remaining argument is the same as in the proof of Case \ref{Lemma:case1}.

\end{proof}

By Claim \ref{claim:z_1=t_1^1}, we have \(z_1 = t_1^1 \in V_1\). {Let } \(v' \in V_2\) be the predecessor of \(z_1\) along \(Q_1[s_2,z_1]\). {As  \(s_1^{i''} \in V_2\) it is adjacent to $v'$.}  If \(v' \to s_1^{i''}\), then the \((s_2,t_2)\)-path
\(
Q_1[s_2,v'] \circ v's_1^{i''} \circ P_{i''}[s_1^{i''},w_1] \circ Q_1[w_1,t_2]
\)
is disjoint from \(P_i\) for all \(i \in \intthree \setminus \{i''\}\) (See Figure \ref{fig3a:lemma}). {Hence we must have} \(s_1^{i''} \to v'\). {Now we obtain the final contradiction by observing that} the \((s_1,t_1)\)-path
\(
s_1 s_1^{i''} v' z_1 t_1
\)
is disjoint from \(Q_j\) for some \(j \in \{2,3\}\) (See Figure \ref{fig3b:lemma}).

\end{proof}

\begin{figure}[H]
\centering
\tikzset{pics/manuscriptdirection/.style={code={
  \fill[black] (0pt,0pt) -- (-6pt,2pt) -- (-6pt,-2pt) -- cycle;
}}}

\begin{subfigure}{0.45\textwidth}
\centering
\resizebox{\linewidth}{!}{%
\begin{tikzpicture}[
    font=\large,
    line cap=round,
    line join=round,
    pathline/.style={draw=black,line width=0.8pt},
    solidroute/.style={draw=black,line width=1.4pt,solid},
    dashedroute/.style={draw=black,line width=1.2pt,
        dash pattern=on 4pt off 2.5pt,line cap=butt},
    arrow/.style={-{Triangle[length=6pt,width=3pt,line width=0pt]},shorten >=0pt},
    starvertex/.style={shape=star,star points=5,star point ratio=2.35,
        draw=none,fill=black,minimum size=12pt,inner sep=0pt,outer sep=0pt},
    diamondvertex/.style={diamond,draw=none,fill=black,
        minimum size=7.5pt,inner sep=0pt,outer sep=0pt},
    blackvertex/.style={circle,draw=none,fill=black,
        minimum size=5.5pt,inner sep=0pt,outer sep=0pt},
    lab/.style={inner sep=1pt}
]
\node[starvertex]    (s1)   at ( 0.00,6.50) {};
\node[diamondvertex] (t1)   at ( 0.00,0.00) {};
\node[starvertex]    (s2)   at (-4.10,3.25) {};
\node[diamondvertex] (t2)   at ( 3.75,3.25) {};
\node[diamondvertex] (s11)  at (-0.79,5.72) {};
\node[diamondvertex] (s12)  at ( 0.00,5.68) {};
\node[diamondvertex] (s13)  at ( 0.78,5.68) {};
\node[diamondvertex] (v)    at (-1.00,1.08) {};
\node[diamondvertex] (vp)   at (-2.28,1.08) {};
\node[starvertex]    (z1)   at (-0.56,0.53) {};
\node[diamondvertex] (p2d)  at ( 0.00,1.14) {};
\node[starvertex]    (p2s)  at ( 0.00,0.58) {};
\node[diamondvertex] (p3d)  at ( 0.90,0.98) {};
\node[starvertex]    (p3s)  at ( 0.58,0.58) {};
\node[blackvertex]   (w1)   at ( 1.59,3.36) {};

\draw[pathline,arrow] (s1) -- (s12);
\draw[pathline] (s12) -- (p2d)
    pic[pos=0.70,sloped,allow upside down] {manuscriptdirection};
\draw[pathline] (p2d) -- (p2s);
\draw[pathline] (p2s) -- (t1)
    pic[pos=0.90,sloped,allow upside down] {manuscriptdirection};
\draw[pathline,arrow]
    (s1) .. controls (0.30,6.24) and (0.56,5.99) .. (s13);
\draw[pathline]
    (w1) .. controls (1.59,2.35) and (1.30,1.52) .. (p3d)
    pic[pos=0.70,sloped,allow upside down] {manuscriptdirection};
\draw[pathline]
    (p3d) .. controls (0.78,0.79) and (0.69,0.67) .. (p3s);
\draw[pathline]
    (p3s) .. controls (0.39,0.33) and (0.19,0.16) .. (t1)
    pic[pos=0.80,sloped,allow upside down] {manuscriptdirection};

\draw[pathline,arrow]
    (vp) .. controls (-1.85,0.66) and (-1.32,0.59) .. (z1);

\draw[dashedroute,arrow]
    (s2) .. controls (-2.66,2.69) and (-2.77,1.57) .. (vp)
    pic[pos=0.34,sloped,allow upside down] {manuscriptdirection};
\draw[dashedroute,arrow] (vp) -- (s13);
\draw[dashedroute,arrow]
    (s13) .. controls (1.36,4.91) and (1.57,4.14) .. (w1);
\draw[dashedroute,arrow]
    (w1) .. controls (2.25,3.91) and (2.38,4.08) .. (2.80,2.93)
    .. controls (3.03,2.36) and (3.43,2.91) .. (t2);

\draw[solidroute,arrow]
    (s1) .. controls (-0.30,6.24) and (-0.55,6.00) .. (s11);
\draw[solidroute,arrow]
    (s11) .. controls (-1.95,4.18) and (-1.91,2.61) .. (v);
\draw[solidroute,arrow]
    (v) .. controls (-0.86,0.89) and (-0.72,0.70) .. (z1);
\draw[solidroute,arrow]
    (z1) .. controls (-0.40,0.35) and (-0.19,0.16) .. (t1);

\node[lab,above=5pt] at (s1.north) {$s_1$};
\node[lab,below=5pt] at (t1.south) {$t_1$};
\node[lab,left=7pt] at (s2.west) {$s_2$};
\node[lab,right=5pt] at (t2.east) {$t_2$};
\node[lab,left=5pt,yshift=1pt] at (s11.center) {$s_1^1$};
\node[lab,left=5pt,yshift=-1pt] at (s12.center) {$s_1^2$};
\node[lab,right=5pt,yshift=1pt] at (s13.center) {$s_1^3$};
\node[lab,left=5pt] at (v.west) {$v$};
\node[lab,below left=3pt] at (vp.south west) {$v'$};
\node[lab,left=5pt] at (w1.west) {$w_1$};
\node[lab] at (-1.08,0.10) {$z_1=t_1^1$};
\node[lab] at (-1.78,4.83) {$P_1$};
\node[lab] at (-0.34,4.90) {$P_2$};
\node[lab] at ( 1.74,4.86) {$P_3$};
\node[lab] at (-3.12,1.90) {$Q_1$};

\node[starvertex] at (2.88,6.11) {};
\node[diamondvertex] at (2.88,5.59) {};
\node[lab,anchor=west] at (3.15,6.11) {$V_1$};
\node[lab,anchor=west] at (3.15,5.59) {$V_2$};
\end{tikzpicture}%
}
\caption{}
\label{fig3a:lemma}
\end{subfigure}
\hspace{4mm}
\begin{subfigure}{0.45\textwidth}
\centering
\resizebox{\linewidth}{!}{%
\begin{tikzpicture}[
    font=\large,
    line cap=round,
    line join=round,
    pathline/.style={draw=black,line width=0.8pt},
    solidroute/.style={draw=black,line width=1.4pt,solid},
    dashedroute/.style={draw=black,line width=1.2pt,
        dash pattern=on 4pt off 2.5pt,line cap=butt},
    arrow/.style={-{Triangle[length=6pt,width=3pt,line width=0pt]},shorten >=0pt},
    starvertex/.style={shape=star,star points=5,star point ratio=2.35,
        draw=none,fill=black,minimum size=12pt,inner sep=0pt,outer sep=0pt},
    diamondvertex/.style={diamond,draw=none,fill=black,
        minimum size=7.5pt,inner sep=0pt,outer sep=0pt},
    blackvertex/.style={circle,draw=none,fill=black,
        minimum size=5.5pt,inner sep=0pt,outer sep=0pt},
    lab/.style={inner sep=1pt}
]
\node[starvertex]    (s1)   at ( 0.00,6.50) {};
\node[diamondvertex] (t1)   at ( 0.00,0.00) {};
\node[starvertex]    (s2)   at (-4.10,3.25) {};
\node[diamondvertex] (t2)   at ( 3.75,3.25) {};
\node[diamondvertex] (s11)  at (-0.79,5.72) {};
\node[diamondvertex] (s12)  at ( 0.00,5.68) {};
\node[diamondvertex] (s13)  at ( 0.78,5.68) {};
\node[diamondvertex] (v)    at (-1.00,1.08) {};
\node[diamondvertex] (vp)   at (-2.28,1.08) {};
\node[starvertex]    (z1)   at (-0.56,0.53) {};
\node[diamondvertex] (p2d)  at ( 0.00,1.14) {};
\node[starvertex]    (p2s)  at ( 0.00,0.58) {};
\node[diamondvertex] (p3d)  at ( 0.90,0.98) {};
\node[starvertex]    (p3s)  at ( 0.58,0.58) {};
\node[blackvertex]   (w1)   at ( 1.59,3.36) {};

\draw[pathline,arrow]
    (s1) .. controls (-0.30,6.24) and (-0.55,6.00) .. (s11);
\draw[pathline]
    (s11) .. controls (-1.95,4.18) and (-1.91,2.61) .. (v)
    pic[pos=0.45,sloped,allow upside down] {manuscriptdirection};
\draw[pathline,arrow]
    (v) .. controls (-0.86,0.89) and (-0.72,0.70) .. (z1);
\draw[pathline,arrow] (s1) -- (s12);
\draw[pathline] (s12) -- (p2d)
    pic[pos=0.70,sloped,allow upside down] {manuscriptdirection};
\draw[pathline] (p2d) -- (p2s);
\draw[pathline] (p2s) -- (t1)
    pic[pos=0.90,sloped,allow upside down] {manuscriptdirection};
\draw[pathline]
    (s13) .. controls (1.36,4.91) and (1.57,4.14) .. (w1)
    pic[pos=0.50,sloped,allow upside down] {manuscriptdirection};
\draw[pathline]
    (w1) .. controls (1.59,2.35) and (1.30,1.52) .. (p3d)
    pic[pos=0.70,sloped,allow upside down] {manuscriptdirection};
\draw[pathline]
    (p3d) .. controls (0.78,0.79) and (0.69,0.67) .. (p3s);
\draw[pathline]
    (p3s) .. controls (0.39,0.33) and (0.19,0.16) .. (t1)
    pic[pos=0.80,sloped,allow upside down] {manuscriptdirection};

\draw[pathline]
    (s2) .. controls (-2.66,2.69) and (-2.77,1.57) .. (vp)
    pic[pos=0.34,sloped,allow upside down] {manuscriptdirection};
\draw[pathline]
    (w1) .. controls (2.25,3.91) and (2.38,4.08) .. (2.80,2.93)
    .. controls (3.03,2.36) and (3.43,2.91) .. (t2)
    pic[pos=0.80,sloped,allow upside down] {manuscriptdirection};

\draw[dashedroute,arrow]
    (s2) .. controls (-2.44,3.86) and (-2.13,3.96) .. (-1.10,2.95)
    pic[pos=0.30,sloped,allow upside down] {manuscriptdirection}
    .. controls (-0.37,2.14) and (0.12,3.15) .. (0.91,3.88)
    .. controls (1.65,4.65) and (2.47,4.19) .. (t2);

\draw[solidroute,arrow]
    (s1) .. controls (0.30,6.24) and (0.56,5.99) .. (s13);
\draw[solidroute,arrow] (s13) -- (vp);
\draw[solidroute,arrow]
    (vp) .. controls (-1.85,0.66) and (-1.32,0.59) .. (z1);
\draw[solidroute,arrow]
    (z1) .. controls (-0.40,0.35) and (-0.19,0.16) .. (t1);

\node[lab,above=5pt] at (s1.north) {$s_1$};
\node[lab,below=5pt] at (t1.south) {$t_1$};
\node[lab,left=7pt] at (s2.west) {$s_2$};
\node[lab,right=5pt] at (t2.east) {$t_2$};
\node[lab,left=5pt,yshift=1pt] at (s11.center) {$s_1^1$};
\node[lab,left=5pt,yshift=-1pt] at (s12.center) {$s_1^2$};
\node[lab,right=5pt,yshift=1pt] at (s13.center) {$s_1^3$};
\node[lab,left=5pt] at (v.west) {$v$};
\node[lab,below left=3pt] at (vp.south west) {$v'$};
\node[lab,left=5pt] at (w1.west) {$w_1$};
\node[lab] at (-1.08,0.10) {$z_1=t_1^1$};
\node[lab] at (-1.78,4.83) {$P_1$};
\node[lab] at (-0.34,4.90) {$P_2$};
\node[lab] at ( 1.74,4.86) {$P_3$};
\node[lab] at (-3.12,1.90) {$Q_1$};
\node[lab] at (-2.70,3.97) {$Q_j$};

\node[starvertex] at (2.88,6.11) {};
\node[diamondvertex] at (2.88,5.59) {};
\node[lab,anchor=west] at (3.15,6.11) {$V_1$};
\node[lab,anchor=west] at (3.15,5.59) {$V_2$};
\end{tikzpicture}%
}
\caption{}
\label{fig3b:lemma}
\end{subfigure}
\caption{Example with \(i'' = 3\): (a) \(v' \to s_1^3\); (b) \(s_1^3 \to v'\).}
\label{fig3:lemma}
\end{figure}
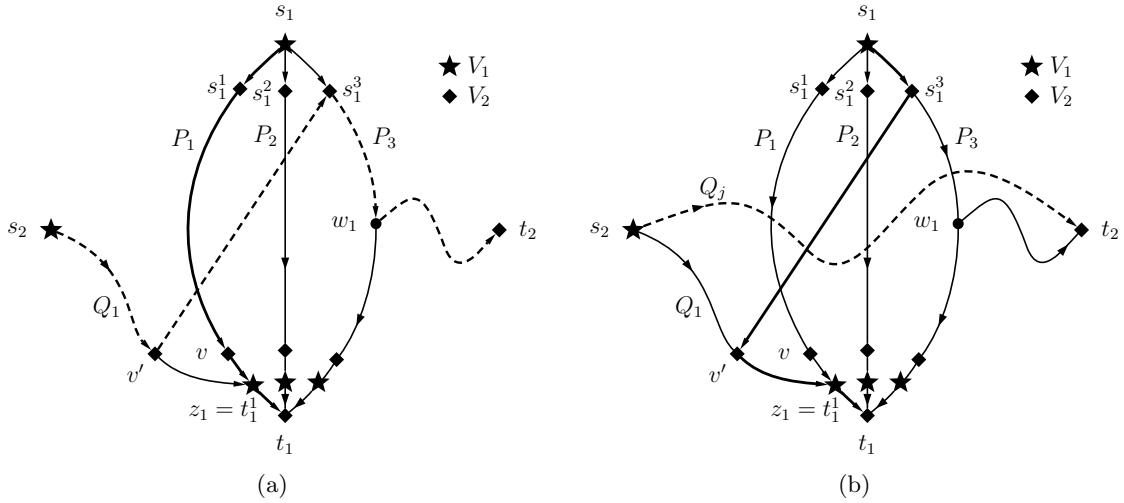

\section{Unilateral 2-linkages in split digraphs}\label{sec:splitunilateralsec}

Now we are ready to prove Theorem \ref{thm:mainsplit} which we recall below.

\ThmUnilateralsplit*
\begin{proof}

Let \(D=(V_1,V_2;A)\) be a \(4\)-strong split digraph, and let \(s_1,s_2,t_1,t_2\) be four distinct vertices of \(D\). Suppose that \(D\) is not unilaterally \(2\)-linked; that is,

\begin{equation}
(D,\{s_1,t_1\},\{s_2,t_2\}) \text{ is not good}.
\tag{$\ast$}
\label{eq:contradiction}
\end{equation}

By Menger's theorem, there exist four internally disjoint \((s_1,t_1)\)-paths \(P_1,P_2,P_3,P_4\) in \(D\), which can be chosen to be minimal. Let \(S = V(P_1) \cup V(P_2) \cup V(P_3) \cup V(P_4)\). Similarly, let \(Q_1,Q_2,Q_3,Q_4\) be internally disjoint minimal \((s_2,t_2)\)-paths in \(D\). We may assume that each of the path \(P_i\) for \(i \in \intfour\) intersects all the paths \(Q_i\), \( i \in \intfour\). Furthermore, let \(s_1^i(s_2^i)\) be the successor of \(s_1(s_2)\) on \(P_i(Q_i)\) for \(i \in \intfour\), and let \(t_1^i(t_2^i)\) be the predecessor of \(t_1(t_2)\) on \(P_i(Q_i)\) for \(i \in \intfour\).

Suppose that there exists some \(i \in \{1,2\}\) such that \(|\{s_i,t_i\} \cap V_1| = 0\), i.e., \(\{s_i,t_i\} \subseteq V_2\); then there is an arc between \(s_i\) and \(t_i\). Since \(D - \{s_i,t_i\}\) is 2-strong, there exists an \((s_{3-i},t_{3-i})\)-path in \(D - \{s_i,t_i\}\), which contradicts \eqref{eq:contradiction}.
Therefore, we may assume that \(|\{s_i,t_i\} \cap V_1| \geq 1\) for each \(i \in \{1,2\}\).
Next, we divide the proof into three cases based on the distribution of \(\{s_1,s_2,t_1,t_2\}\) in \(V_1\) and \(V_2\).

\setcounter{case}{0}
\vspace{2mm}
\mycase \label{case1} \(|\{s_i,t_i\} \cap V_1| = 1\) for each \(i \in \inttwo\).
\vspace{2mm}

In this case, the terminal pairs \(\{s_1,t_1\}\) and \(\{s_2,t_2\}\) have four distinct distributions with respect to the partition \((V_1,V_2)\). Since we do not care whether \(P_i\) is an \((s_i,t_i)\)-path or a \((t_i,s_i)\)-path, we may always relabel two of the terminals to achieve \(\{s_1,s_2\} \subseteq V_1\) and \(\{t_1,t_2\} \subseteq V_2\). Hence, w.l.o.g.,  we assume that \(\{s_1,s_2\} \subseteq V_1\) and \(\{t_1,t_2\} \subseteq V_2\).

\begin{claim}\label{claim:s_1 not in v(Q_1) V(Q_2) V(Q_3) V(Q_4)}
\(s_1 \notin V(Q_1) \cup V(Q_2) \cup V(Q_3) \cup V(Q_4)\)
\end{claim}

\begin{proof}
    Suppose \(s_1 \in V(Q_1) \cup V(Q_2) \cup V(Q_3) \cup V(Q_4)\). W.l.o.g., assume \(s_1 \in V(Q_1)\).
    Suppose that \(s_2 \notin S\). Then there exist at least three internally disjoint \((s_1,t_1)\)-paths in \(D-\{s_2,t_2\}\). W.l.o.g., let these three \((s_1,t_1)\)-paths be \(P_1, P_2, P_3\).
   Since \(D\) is 4-strong, there are at least two internally disjoint \((s_2,t_2)\)-paths in \(D-\{s_1,t_1\}\). W.l.o.g., let these two paths be \(Q_2\) and \(Q_3\). Let \(w\) be the last vertex on \(Q_2\) that lies in \(\bigcup_{i=1}^3 V(P_i)\). Then \(w\in V(P_{i_0})\) for some \(i_0\in\intthree\). Clearly, \(s_2^1\) and \(s_1^{i_0}\) are adjacent since both lie in \(V_2\). If \(s_2^1 \to s_1^{i_0}\), then the \((s_2,t_2)\)-path \(s_2s_2^1 \circ s_2^1s_1^{i_0} \circ P_{i_0}[s_1^{i_0},w] \circ Q_2[w,t_2]\) is disjoint from \(P_i\) for some \(i \in \intthree \setminus \{i_0\}\), which contradicts \eqref{eq:contradiction}. Hence, we may assume \(s_1^{i_0} \to s_2^1\). Similarly, \(t_1\) and \(s_2^1\) are adjacent as they are both in \(V_2\). If \(t_1 \to s_2^1\), then the \((t_1,s_1)\)-path \(t_1s_2^1 \circ Q_1[s_2^1,s_1]\) is disjoint from \(Q_j\) for any \(j\in\{2,3\}\), yielding a contradiction to \eqref{eq:contradiction}. Hence, we have \(s_2^1 \to t_1\). However, {then} the \((s_1,t_1)\)-path \(s_1s_1^{i_0}s_2^1t_1\) is disjoint from \(Q_j\) for some \(j \in \{2,3\}\), which again contradicts \eqref{eq:contradiction} (See Figure \ref{fig4a:thm}). Therefore, we may assume \(s_2 \in S\).

   W.l.o.g., assume \(s_2 \in V(P_1)\), since \(s_2 \in S\). Note that \(s_1^1, t_2 \in V_2\). If \(t_2 \to s_1^1\), then the \((t_2,s_2)\)-path \(t_2s_1^1 \circ P_1[s_1^1,s_2]\) is disjoint from \(P_i\) for some  \(i \in \intfour \setminus \{1\}\), which contradicts \eqref{eq:contradiction}. Thus we may assume \(s_1^1 \to t_2\).
   Recall that \(s_1^1, s_2^1 \in V_2\). If \(s_2^1 \to s_1^1\), then the \((s_2,t_2)\)-path \(s_2 \to s_2^1\to s_1^1 \to t_2\) is disjoint from \(P_i\) for some \(i \in \intfour \setminus \{1\}\), again contradicting \eqref{eq:contradiction}. Hence we have \(s_1^1 \to s_2^1\).
   We know that \(s_2^1, t_1 \in V_2\). If \(t_1 \to s_2^1\), then the \((t_1,s_1)\)-path \(t_1s_2^1 \circ Q_1[s_2^1,s_1]\) is disjoint from each \(Q_j\) with \(j\in\{2,3\}\), yielding a contradiction to \eqref{eq:contradiction}. Therefore, we may assume \(s_2^1 \to t_1\). 
   {Now} the \((s_1,t_1)\)-path \(s_1 \to s_1^1 \to s_2^1 \to t_1\) is disjoint from some \(Q_j\) where \(j\in\{2,3\}\) (Figure \ref{fig4b:thm}) {This final contradiction proves the claim}.
   \end{proof}

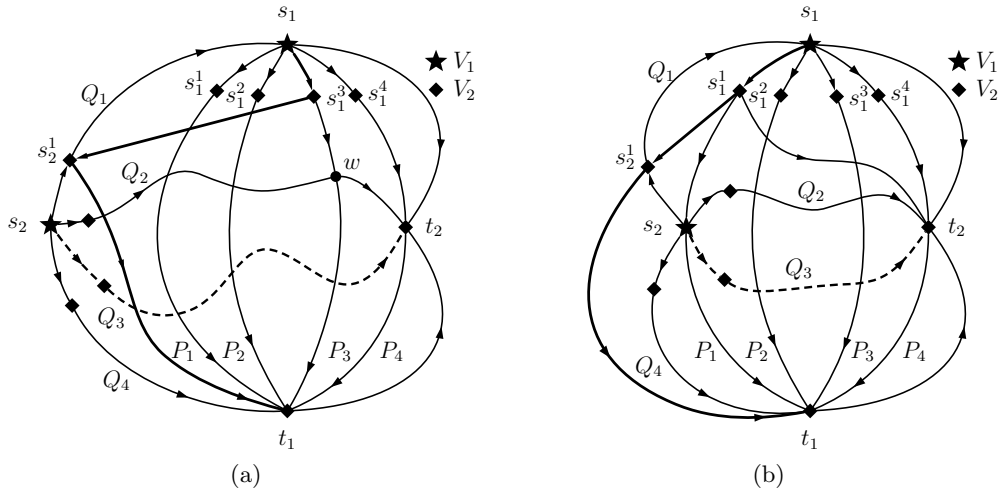
\begin{figure}[htbp]
\centering
\tikzset{pics/manuscriptdirection/.style={code={
  \fill[black] (0pt,0pt) -- (-6pt,2pt) -- (-6pt,-2pt) -- cycle;
}}}

\begin{subfigure}{0.4\textwidth}
\centering
\resizebox{\linewidth}{!}{%
\begin{tikzpicture}[
    font=\large,
    line cap=round,
    line join=round,
    pathline/.style={draw=black,line width=0.8pt},
    solidroute/.style={draw=black,line width=1.4pt,solid},
    dashedroute/.style={draw=black,line width=1.2pt,
        dash pattern=on 4pt off 2.5pt,line cap=butt},
    arrow/.style={-{Triangle[length=6pt,width=3pt,line width=0pt]},shorten >=0pt},
    starvertex/.style={shape=star,star points=5,star point ratio=2.35,
        draw=none,fill=black,minimum size=12pt,inner sep=0pt,outer sep=0pt},
    diamondvertex/.style={diamond,draw=none,fill=black,
        minimum size=7.5pt,inner sep=0pt,outer sep=0pt},
    blackvertex/.style={circle,draw=none,fill=black,
        minimum size=5.5pt,inner sep=0pt,outer sep=0pt},
    pics/figfourarrow/.style={code={
        \fill[black] (0pt,0pt) -- (-6pt,1.5pt) -- (-6pt,-1.5pt) -- cycle;
    }},
    lab/.style={inner sep=1pt}
]

\path[use as bounding box] (-4.95,-0.62) rectangle (3.48,7.02);
\node[starvertex]    (s1)  at ( 0.00,6.50) {};
\node[diamondvertex] (t1)  at ( 0.00,0.00) {};
\node[starvertex]    (s2)  at (-4.20,3.30) {};
\node[diamondvertex] (t2)  at ( 2.10,3.26) {};
\node[diamondvertex] (s11) at (-1.25,5.68) {};
\node[diamondvertex] (s12) at (-0.52,5.60) {};
\node[diamondvertex] (s13) at ( 0.47,5.58) {};
\node[diamondvertex] (s14) at ( 1.21,5.60) {};
\node[diamondvertex] (s21) at (-3.86,4.45) {};
\node[diamondvertex] (q2d) at (-3.53,3.38) {};
\node[diamondvertex] (q3d) at (-3.25,2.22) {};
\node[diamondvertex] (q4d) at (-3.82,1.87) {};
\node[blackvertex]   (w)   at ( 0.86,4.16) {};

\draw[pathline]
    (s1) .. controls (-0.46,6.35) and (-0.96,5.99) .. (s11)
    pic[pos=0.70,sloped,allow upside down] {manuscriptdirection};
\draw[pathline]
    (s11) .. controls (-2.90,3.65) and (-2.66,1.39) .. (t1)
    pic[pos=0.82,sloped,allow upside down] {manuscriptdirection};
\draw[pathline]
    (s1) .. controls (-0.18,6.18) and (-0.38,5.86) .. (s12)
    pic[pos=0.70,sloped,allow upside down] {manuscriptdirection};
\draw[pathline]
    (s12) .. controls (-1.30,3.49) and (-1.19,1.88) .. (t1)
    pic[pos=0.80,sloped,allow upside down] {manuscriptdirection};

\draw[pathline]
    (s13) .. controls (0.66,5.09) and (0.78,4.56) .. (w)
    pic[pos=0.60,sloped,allow upside down] {manuscriptdirection};
\draw[pathline]
    (w) .. controls (1.16,2.62) and (0.52,0.87) .. (t1)
    pic[pos=0.80,sloped,allow upside down] {manuscriptdirection};
\draw[pathline]
    (s1) .. controls (0.45,6.33) and (0.93,5.96) .. (s14)
    pic[pos=0.70,sloped,allow upside down] {manuscriptdirection};
\draw[pathline]
    (s14) .. controls (1.83,4.88) and (2.08,4.12) .. (t2)
    pic[pos=0.65,sloped,allow upside down] {manuscriptdirection};
\draw[pathline]
    (t2) .. controls (2.09,1.68) and (1.31,0.56) .. (t1)
    pic[pos=0.80,sloped,allow upside down] {manuscriptdirection};

\draw[pathline,arrow]
    (s2) .. controls (-4.12,3.70) and (-4.00,4.15) .. (s21);
\draw[pathline]
    (s21) .. controls (-3.11,5.96) and (-1.53,6.62) .. (s1)
    pic[pos=0.70,sloped,allow upside down] {manuscriptdirection};
\draw[pathline]
    (s1) .. controls (2.22,6.42) and (3.52,5.21) .. (t2)
    pic[pos=0.75,sloped,allow upside down] {manuscriptdirection};

\draw[pathline]
    (s2) .. controls (-3.98,3.32) and (-3.77,3.34) .. (q2d)
    pic[pos=0.88,sloped,allow upside down] {manuscriptdirection};
\draw[pathline]
    (q2d) .. controls (-2.65,3.61) and (-2.22,4.51) .. (-1.50,4.18)
    pic[pos=0.45,sloped,allow upside down] {manuscriptdirection}
    .. controls (-0.51,3.78) and (-0.32,3.86) .. (w);
\draw[pathline]
    (w) .. controls (1.24,4.16) and (1.68,3.80) .. (t2)
    pic[pos=0.55,sloped] {figfourarrow};

\draw[pathline]
    (s2) .. controls (-4.19,2.73) and (-4.03,2.22) .. (q4d)
    pic[pos=0.70,sloped,allow upside down] {manuscriptdirection};
\draw[pathline]
    (q4d) .. controls (-3.11,0.52) and (-1.36,-0.06) .. (t1)
    pic[pos=0.60,sloped,allow upside down] {manuscriptdirection};
\draw[pathline]
    (t1) .. controls (2.36,0.13) and (3.55,1.33) .. (t2)
    pic[pos=0.65,sloped,allow upside down] {manuscriptdirection};

\draw[dashedroute]
    (s2) .. controls (-3.94,2.94) and (-3.62,2.54) .. (q3d)
    pic[pos=0.65,sloped,allow upside down] {manuscriptdirection};
\draw[dashedroute]
    (q3d) .. controls (-2.77,1.77) and (-2.25,1.54) .. (-1.65,1.82)
    .. controls (-1.05,2.13) and (-0.82,2.96) .. (-0.29,2.86)
    .. controls (0.53,2.69) and (1.09,1.39) .. (t2)
    pic[pos=0.9,sloped,allow upside down] {manuscriptdirection};

\draw[solidroute,arrow]
    (s1) .. controls (0.20,6.22) and (0.38,5.90) .. (s13);
\draw[solidroute,arrow] (s13) -- (s21);
\draw[solidroute]
    (s21) .. controls (-2.67,2.81) and (-2.96,1.72) .. (-2.22,1.02)
    pic[pos=0.47,sloped] {figfourarrow}
    .. controls (-1.59,0.45) and (-0.64,0.17) .. (t1)
    pic[pos=0.70,sloped,allow upside down] {manuscriptdirection};

\node[lab,above=5pt] at (s1.north) {$s_1$};
\node[lab,below=5pt] at (t1.south) {$t_1$};
\node[lab,left=7pt] at (s2.west) {$s_2$};
\node[lab,right=5pt] at (t2.east) {$t_2$};
\node[lab,left=5pt,yshift=4pt] at (s11.center) {$s_1^1$};
\node[lab,left=5pt] at (s12.center) {$s_1^2$};
\node[lab,right=5pt,yshift=-1pt] at (s13.center) {$s_1^3$};
\node[lab,right=5pt,yshift=1pt] at (s14.center) {$s_1^4$};
\node[lab,left=5pt,yshift=5pt] at (s21.center) {$s_2^1$};
\node[lab] at (-1.85,1.06) {$P_1$};
\node[lab] at (-0.95,1.06) {$P_2$};
\node[lab] at ( 0.93,1.06) {$P_3$};
\node[lab] at ( 1.85,1.06) {$P_4$};

\node[lab,above right=1pt] at (w.north east) {$w$};
\node[lab] at (-3.42,5.61) {$Q_1$};
\node[lab] at (-2.70,4.22) {$Q_2$};
\node[lab] at (-3.12,1.64) {$Q_3$};
\node[lab] at (-3.03,0.53) {$Q_4$};

\node[starvertex] at (2.64,6.20) {};
\node[diamondvertex] at (2.64,5.70) {};
\node[lab,anchor=west] at (2.90,6.20) {$V_1$};
\node[lab,anchor=west] at (2.90,5.70) {$V_2$};
\end{tikzpicture}%
}
\caption{}
\label{fig4a:thm}
\end{subfigure}
\hspace{4mm}
\begin{subfigure}{0.4\textwidth}
\centering
\resizebox{\linewidth}{!}{%
\begin{tikzpicture}[
    font=\large,
    line cap=round,
    line join=round,
    pathline/.style={draw=black,line width=0.8pt},
    solidroute/.style={draw=black,line width=1.4pt,solid},
    dashedroute/.style={draw=black,line width=1.2pt,
        dash pattern=on 4pt off 2.5pt,line cap=butt},
    arrow/.style={-{Triangle[length=6pt,width=3pt,line width=0pt]},shorten >=0pt},
    starvertex/.style={shape=star,star points=5,star point ratio=2.35,
        draw=none,fill=black,minimum size=12pt,inner sep=0pt,outer sep=0pt},
    diamondvertex/.style={diamond,draw=none,fill=black,
        minimum size=7.5pt,inner sep=0pt,outer sep=0pt},
    blackvertex/.style={circle,draw=none,fill=black,
        minimum size=5.5pt,inner sep=0pt,outer sep=0pt},
    pics/figfourarrow/.style={code={
        \fill[black] (0pt,0pt) -- (-6pt,1.5pt) -- (-6pt,-1.5pt) -- cycle;
    }},
    lab/.style={inner sep=1pt}
]

\path[use as bounding box] (-4.95,-0.62) rectangle (3.48,7.02);
\node[starvertex]    (s1)  at ( 0.00,6.50) {};
\node[diamondvertex] (t1)  at ( 0.00,0.00) {};
\node[starvertex]    (s2)  at (-2.20,3.25) {};
\node[diamondvertex] (t2)  at ( 2.10,3.26) {};
\node[diamondvertex] (s11) at (-1.25,5.68) {};
\node[diamondvertex] (s12) at (-0.52,5.60) {};
\node[diamondvertex] (s13) at ( 0.47,5.58) {};
\node[diamondvertex] (s14) at ( 1.21,5.60) {};
\node[diamondvertex] (s21) at (-2.88,4.33) {};
\node[diamondvertex] (q2d) at (-1.42,3.90) {};
\node[diamondvertex] (q3d) at (-1.52,2.33) {};
\node[diamondvertex] (q4d) at (-2.77,2.16) {};

\draw[pathline]
    (s11) .. controls (-1.83,4.83) and (-2.18,4.08) .. (s2)
    pic[pos=0.55,sloped,allow upside down] {manuscriptdirection};
\draw[pathline]
    (s2) .. controls (-2.24,1.78) and (-1.40,0.56) .. (t1)
    pic[pos=0.80,sloped,allow upside down] {manuscriptdirection};
\draw[pathline]
    (s1) .. controls (-0.18,6.18) and (-0.38,5.86) .. (s12)
    pic[pos=0.70,sloped,allow upside down] {manuscriptdirection};
\draw[pathline]
    (s12) .. controls (-1.30,3.49) and (-1.19,1.88) .. (t1)
    pic[pos=0.80,sloped,allow upside down] {manuscriptdirection};

\draw[pathline,arrow]
    (s1) .. controls (0.20,6.22) and (0.38,5.90) .. (s13);
\draw[pathline]
    (s13) .. controls (1.16,3.70) and (1.16,1.69) .. (t1)
    pic[pos=0.80,sloped,allow upside down] {manuscriptdirection};
\draw[pathline]
    (s1) .. controls (0.45,6.33) and (0.93,5.96) .. (s14)
    pic[pos=0.70,sloped,allow upside down] {manuscriptdirection};
\draw[pathline]
    (s14) .. controls (1.83,4.88) and (2.08,4.12) .. (t2)
    pic[pos=0.65,sloped,allow upside down] {manuscriptdirection};
\draw[pathline]
    (t2) .. controls (2.09,1.68) and (1.31,0.56) .. (t1)
    pic[pos=0.80,sloped,allow upside down] {manuscriptdirection};

\draw[pathline,arrow]
    (s2) .. controls (-2.58,3.70) and (-2.80,3.96) .. (s21);
\draw[pathline]
    (s21) .. controls (-3.15,5.39) and (-2.10,6.77) .. (s1)
    pic[pos=0.70,sloped,allow upside down] {manuscriptdirection};
\draw[pathline]
    (s1) .. controls (2.22,6.42) and (3.52,5.21) .. (t2)
    pic[pos=0.75,sloped,allow upside down] {manuscriptdirection};

\draw[pathline]
    (s2) .. controls (-1.97,3.60) and (-1.71,3.91) .. (q2d)
    pic[pos=0.70,sloped,allow upside down] {manuscriptdirection};
\draw[pathline]
    (q2d) .. controls (-0.91,3.98) and (-0.23,3.43) .. (0.28,3.60)
    .. controls (1.15,3.85) and (1.42,4.06) .. (t2)
    pic[pos=0.75,sloped,allow upside down] {manuscriptdirection};

\draw[pathline]
    (s2) .. controls (-2.55,2.76) and (-2.68,2.48) .. (q4d)
    pic[pos=0.70,sloped,allow upside down] {manuscriptdirection};
\draw[pathline]
    (q4d) .. controls (-2.98,0.73) and (-1.66,-0.22) .. (t1)
    pic[pos=0.60,sloped,allow upside down] {manuscriptdirection};
\draw[pathline]
    (t1) .. controls (2.36,0.13) and (3.55,1.33) .. (t2)
    pic[pos=0.65,sloped,allow upside down] {manuscriptdirection};

\draw[pathline]
    (s11) .. controls (-0.89,4.88) and (-0.63,4.54) .. (0.10,4.45)
    pic[pos=0.94,sloped] {figfourarrow}
    .. controls (1.01,4.41) and (1.41,4.47) .. (t2);

\draw[dashedroute]
    (s2) .. controls (-2.02,2.87) and (-1.84,2.54) .. (q3d)
    pic[pos=0.65,sloped,allow upside down] {manuscriptdirection};
\draw[dashedroute]
    (q3d) .. controls (-1.03,1.97) and (-0.29,2.22) .. (0.69,2.25)
    .. controls (1.22,2.28) and (1.52,2.39) .. (t2)
    pic[pos=0.75,sloped,allow upside down] {manuscriptdirection};

\draw[solidroute,arrow]
    (s1) .. controls (-0.46,6.35) and (-0.96,5.99) .. (s11);
\draw[solidroute,arrow] (s11) -- (s21);
\draw[solidroute]
    (s21) .. controls (-4.33,2.66) and (-4.34,0.96) .. (-2.46,0.15)
    pic[pos=0.74,sloped,yscale=1.5] {figfourarrow}
    .. controls (-1.68,-0.20) and (-0.81,-0.15) .. (t1)
    pic[pos=0.70,sloped,allow upside down] {manuscriptdirection};
\node[lab,above=5pt] at (s1.north) {$s_1$};
\node[lab,below=5pt] at (t1.south) {$t_1$};
\node[lab,left=7pt] at (s2.west) {$s_2$};
\node[lab,right=5pt] at (t2.east) {$t_2$};
\node[lab,left=5pt,yshift=4pt] at (s11.center) {$s_1^1$};
\node[lab,left=5pt] at (s12.center) {$s_1^2$};
\node[lab,right=5pt,yshift=-1pt] at (s13.center) {$s_1^3$};
\node[lab,right=5pt,yshift=1pt] at (s14.center) {$s_1^4$};
\node[lab,left=5pt,yshift=5pt] at (s21.center) {$s_2^1$};
\node[lab] at (-1.85,1.06) {$P_1$};
\node[lab] at (-0.95,1.06) {$P_2$};
\node[lab] at ( 0.93,1.06) {$P_3$};
\node[lab] at ( 1.85,1.06) {$P_4$};

\node[lab] at (-2.64,6.02) {$Q_1$};
\node[lab] at ( 0.02,3.89) {$Q_2$};
\node[lab] at (-0.20,2.52) {$Q_3$};
\node[lab] at (-2.86,0.78) {$Q_4$};

\node[starvertex] at (2.64,6.20) {};
\node[diamondvertex] at (2.64,5.70) {};
\node[lab,anchor=west] at (2.90,6.20) {$V_1$};
\node[lab,anchor=west] at (2.90,5.70) {$V_2$};
\end{tikzpicture}%
}
\caption{}
\label{fig4b:thm}
\end{subfigure}
\caption{(a) \(s_2 \notin S\), \(i_0 = 3\); (b) \(s_2 \in V(P_1)\).}
\label{fig4:thm}
\end{figure}

By Claim \ref{claim:s_1 not in v(Q_1) V(Q_2) V(Q_3) V(Q_4)}, we obtain symmetrically that \(s_2 \notin S\). Therefore, 

\begin{itemize}
    \item There exist three internally disjoint \((s_1,t_1)\)-paths in \(D-\{s_2,t_2\}\).
    \item There exist three internally disjoint \((s_2,t_2)\)-paths in \(D-\{s_1,t_1\}\).
\end{itemize}

By Lemma \ref{Lemma:{s_1,s_2} in V_1 : (3,3)}, the triple \((D,\{s_1,t_1\},\{s_2,t_2\})\) is good, which contradicts \eqref{eq:contradiction}.

\vspace{2mm}
\mycase \label{case2} \(|\{s_i,t_i\} \cap V_1| =2\) for each \(i \in \inttwo\).
\vspace{2mm}

That is, \(\{s_1,t_1,s_2,t_2\} \subseteq V_1\). We now verify that for each \(i \in \{1,2\}\), there exist three internally disjoint \((s_i,t_i)\)-paths in \(D-\{s_{3-i},t_{3-i}\}\).

\begin{claim}\label{claim:case2--there is 3 (s_i,t_i)-paths in D-{s_{3-i},t_{3-i}}}
    There exist three internally disjoint \((s_i,t_i)\)-paths in \(D-\{s_{3-i},t_{3-i}\}\) for each \(i \in \inttwo\).
\end{claim} 

\begin{proof}
Since \(D\) is 4-strong, \(D-\{s_{3-i},t_{3-i}\}\) contains at least two internally disjoint \((s_i,t_i)\)-paths for each \(i \in \inttwo\).
Next, we verify that there exist three internally disjoint$(s_1,t_1)$-paths in $D-\{s_2,t_2\}$.

    If \(D-\{s_2,t_2\}\) has exactly two internally disjoint \((s_1,t_1)\)-paths, say \(P_2\) and \(P_3\). This implies that \(s_2\) and \(t_2\) lie on \(P_1\) and \(P_4\), respectively. W.l.o.g., assume \(s_2 \in V(P_1)\) and \(t_2 \in V(P_4)\). From the definition of split digraphs and \(\{s_1,t_1,s_2,t_2\} \subseteq V_1\), we obtain
\[
|P_1[s_1,s_2]|,\ |P_1[s_2,t_1]|,\ |P_4[s_1,t_2]|,\ |P_4[t_2,t_1]| \ge 3.
\]
Accordingly, we define four vertices as follows:
\begin{itemize}
    \item \(s_2^-\): the predecessor of \(s_2\) along \(P_1[s_1,s_2]\);
    \item \(s_2^+\): the successor of \(s_2\) along \(P_1[s_2,t_1]\);
    \item \(t_2^-\): the predecessor of \(t_2\) along \(P_4[s_1,t_2]\);
    \item \(t_2^+\): the successor of \(t_2\) along \(P_4[t_2,t_1]\).
\end{itemize}

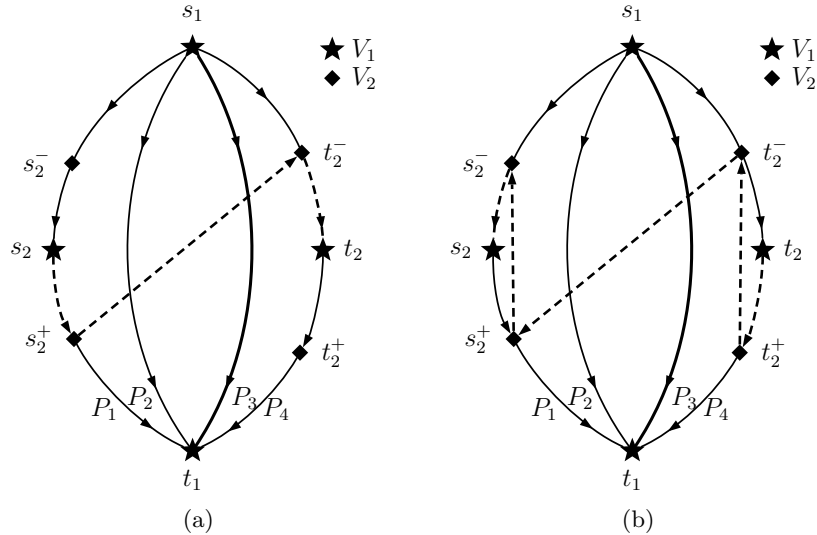
\begin{figure}[H]
\centering
\tikzset{pics/manuscriptdirection/.style={code={
  \fill[black] (0pt,0pt) -- (-6pt,2pt) -- (-6pt,-2pt) -- cycle;
}}}

\begin{subfigure}{0.33\textwidth}
\centering
\resizebox{\linewidth}{!}{%
\begin{tikzpicture}[
    font=\large,line cap=round,line join=round,
    pathline/.style={draw=black,line width=0.8pt},
    solidroute/.style={draw=black,line width=1.4pt,solid},
    dashedroute/.style={draw=black,line width=1.2pt,
        dash pattern=on 4pt off 2.5pt,line cap=butt},
    arrow/.style={-{Triangle[length=6pt,width=4pt,line width=0pt]},
        shorten >=0pt},
    starvertex/.style={shape=star,star points=5,star point ratio=2.35,
        draw=none,fill=black,minimum size=12pt,inner sep=0pt,outer sep=0pt},
    diamondvertex/.style={diamond,draw=none,fill=black,
        minimum size=7.5pt,inner sep=0pt,outer sep=0pt},
    pics/figfivearrow/.style={code={
        \fill[black] (0pt,0pt) -- (-6pt,2pt) -- (-6pt,-2pt) -- cycle;
    }},
    lab/.style={inner sep=1pt}
]
\path[use as bounding box] (-3.06,-0.64) rectangle (3.25,7.04);

\node[starvertex] (s1) at (0,6.50) {};
\node[starvertex] (t1) at (0,0) {};
\node[starvertex] (s2) at (-2.24,3.23) {};
\node[starvertex] (t2) at (2.10,3.23) {};
\node[diamondvertex] (sm) at (-1.94,4.63) {};
\node[diamondvertex] (sp) at (-1.91,1.80) {};
\node[diamondvertex] (tm) at (1.76,4.80) {};
\node[diamondvertex] (tp) at (1.73,1.58) {};

\draw[pathline]
    (s1) .. controls (-0.75,6.15) and (-1.56,5.43) .. (sm)
    pic[pos=0.65,sloped,allow upside down] {manuscriptdirection};
\draw[pathline]
    (sm) .. controls (-2.12,4.20) and (-2.22,3.75) .. (s2)
    pic[pos=0.86,sloped,allow upside down] {figfivearrow};
\draw[pathline]
    (sp) .. controls (-1.53,1.08) and (-0.77,0.34) .. (t1)
    pic[pos=0.75,sloped,allow upside down] {manuscriptdirection};

\draw[pathline]
    (s1) .. controls (-1.37,4.55) and (-1.37,1.86) .. (t1)
    pic[pos=0.25,sloped,allow upside down] {manuscriptdirection}
    pic[pos=0.84,sloped,allow upside down] {manuscriptdirection};

\draw[pathline]
    (s1) .. controls (0.73,6.23) and (1.44,5.53) .. (tm)
    pic[pos=0.60,sloped,allow upside down] {manuscriptdirection};
\draw[pathline,arrow]
    (t2) .. controls (2.10,2.61) and (1.91,1.91) .. (tp);
\draw[pathline]
    (tp) .. controls (1.34,0.91) and (0.69,0.31) .. (t1)
    pic[pos=0.75,sloped,allow upside down] {manuscriptdirection};

\draw[dashedroute,arrow]
    (s2) .. controls (-2.25,2.65) and (-2.14,2.19) .. (sp);
\draw[dashedroute,arrow] (sp) -- (tm);
\draw[dashedroute]
    (tm) .. controls (1.95,4.34) and (2.09,3.70) .. (t2)
    pic[pos=0.86,sloped,allow upside down] {figfivearrow};

\draw[solidroute]
    (s1) .. controls (1.25,4.52) and (1.25,1.87) .. (t1)
    pic[pos=0.25,sloped,allow upside down] {manuscriptdirection}
    pic[pos=0.84,sloped,allow upside down] {manuscriptdirection};

\node[lab,above=5pt] at (s1.north) {$s_1$};
\node[lab,below=5pt] at (t1.south) {$t_1$};
\node[lab,left=5pt] at (s2.west) {$s_2$};
\node[lab,right=5pt] at (t2.east) {$t_2$};
\node[lab,left=5pt] at (sm.west) {$s_2^-$};
\node[lab,left=5pt] at (sp.west) {$s_2^+$};
\node[lab,right=5pt] at (tm.east) {$t_2^-$};
\node[lab,right=5pt] at (tp.east) {$t_2^+$};
\node[lab] at (-1.42,0.67) {$P_1$};
\node[lab] at (-0.84,0.85) {$P_2$};
\node[lab] at (0.84,0.85) {$P_3$};
\node[lab] at (1.35,0.67) {$P_4$};

\node[starvertex] at (2.25,6.46) {};
\node[diamondvertex] at (2.25,6.00) {};
\node[lab,anchor=west] at (2.52,6.46) {$V_1$};
\node[lab,anchor=west] at (2.52,6.00) {$V_2$};
\end{tikzpicture}%
}
\caption{}
\label{fig5a:thm}
\end{subfigure}
\hspace{4mm}
\begin{subfigure}{0.33\textwidth}
\centering
\resizebox{\linewidth}{!}{%
\begin{tikzpicture}[
    font=\large,line cap=round,line join=round,
    pathline/.style={draw=black,line width=0.8pt},
    solidroute/.style={draw=black,line width=1.4pt,solid},
    dashedroute/.style={draw=black,line width=1.2pt,
        dash pattern=on 4pt off 2.5pt,line cap=butt},
    arrow/.style={-{Triangle[length=6pt,width=4pt,line width=0pt]},
        shorten >=0pt},
    starvertex/.style={shape=star,star points=5,star point ratio=2.35,
        draw=none,fill=black,minimum size=12pt,inner sep=0pt,outer sep=0pt},
    diamondvertex/.style={diamond,draw=none,fill=black,
        minimum size=7.5pt,inner sep=0pt,outer sep=0pt},
    pics/figfivearrow/.style={code={
        \fill[black] (0pt,0pt) -- (-6pt,2pt) -- (-6pt,-2pt) -- cycle;
    }},
    lab/.style={inner sep=1pt}
]
\path[use as bounding box] (-3.06,-0.64) rectangle (3.25,7.04);

\node[starvertex] (s1) at (0,6.50) {};
\node[starvertex] (t1) at (0,0) {};
\node[starvertex] (s2) at (-2.24,3.23) {};
\node[starvertex] (t2) at (2.10,3.23) {};
\node[diamondvertex] (sm) at (-1.94,4.63) {};
\node[diamondvertex] (sp) at (-1.91,1.80) {};
\node[diamondvertex] (tm) at (1.76,4.80) {};
\node[diamondvertex] (tp) at (1.73,1.58) {};

\draw[pathline]
    (s1) .. controls (-0.75,6.15) and (-1.56,5.43) .. (sm)
    pic[pos=0.65,sloped,allow upside down] {manuscriptdirection};
\draw[pathline,arrow]
    (s2) .. controls (-2.25,2.65) and (-2.14,2.19) .. (sp);
\draw[pathline]
    (sp) .. controls (-1.53,1.08) and (-0.77,0.34) .. (t1)
    pic[pos=0.75,sloped,allow upside down] {manuscriptdirection};

\draw[pathline]
    (s1) .. controls (-1.37,4.55) and (-1.37,1.86) .. (t1)
    pic[pos=0.25,sloped,allow upside down] {manuscriptdirection}
    pic[pos=0.84,sloped,allow upside down] {manuscriptdirection};

\draw[pathline]
    (s1) .. controls (0.73,6.23) and (1.44,5.53) .. (tm)
    pic[pos=0.60,sloped,allow upside down] {manuscriptdirection};
\draw[pathline]
    (tm) .. controls (1.95,4.34) and (2.09,3.70) .. (t2)
    pic[pos=0.86,sloped,allow upside down] {figfivearrow};
\draw[pathline]
    (tp) .. controls (1.34,0.91) and (0.69,0.31) .. (t1)
    pic[pos=0.75,sloped,allow upside down] {manuscriptdirection};

\draw[dashedroute,arrow]
    (t2) .. controls (2.10,2.61) and (1.91,1.91) .. (tp);
\draw[dashedroute,arrow] (tp) -- (tm);
\draw[dashedroute,arrow] (tm) -- (sp);
\draw[dashedroute,arrow] (sp) -- (sm);
\draw[dashedroute]
    (sm) .. controls (-2.12,4.20) and (-2.22,3.75) .. (s2)
    pic[pos=0.86,sloped,allow upside down] {figfivearrow};

\draw[solidroute]
    (s1) .. controls (1.25,4.52) and (1.25,1.87) .. (t1)
    pic[pos=0.25,sloped,allow upside down] {manuscriptdirection}
    pic[pos=0.84,sloped,allow upside down] {manuscriptdirection};

\node[lab,above=5pt] at (s1.north) {$s_1$};
\node[lab,below=5pt] at (t1.south) {$t_1$};
\node[lab,left=5pt] at (s2.west) {$s_2$};
\node[lab,right=5pt] at (t2.east) {$t_2$};
\node[lab,left=5pt] at (sm.west) {$s_2^-$};
\node[lab,left=5pt] at (sp.west) {$s_2^+$};
\node[lab,right=5pt] at (tm.east) {$t_2^-$};
\node[lab,right=5pt] at (tp.east) {$t_2^+$};
\node[lab] at (-1.42,0.67) {$P_1$};
\node[lab] at (-0.84,0.85) {$P_2$};
\node[lab] at (0.84,0.85) {$P_3$};
\node[lab] at (1.35,0.67) {$P_4$};

\node[starvertex] at (2.25,6.46) {};
\node[diamondvertex] at (2.25,6.00) {};
\node[lab,anchor=west] at (2.52,6.46) {$V_1$};
\node[lab,anchor=west] at (2.52,6.00) {$V_2$};
\end{tikzpicture}%
}
\caption{}
\label{fig5b:thm}
\end{subfigure}
\caption{\(D-\{s_2,t_2\}\) has exactly two internally disjoint \((s_1,t_1)\)-paths. (a) \(s_2^+ \to t_2^-\); (b) \(t_2^- \to s_2^+\).}
\label{fig5:thm}
\end{figure}

By the minimality of \(P_1\) and \(P_4\), we have \(s_2^+ \to s_2^-\) and \(t_2^+ \to t_2^-\). If \(s_2^+ \to t_2^-\), then the \((s_2,t_2)\)-path \(s_2 s_2^+ t_2^- t_2\) is disjoint from \(P_i\) for all \(i \in \{2,3\}\), which contradicts \eqref{eq:contradiction} (See Figure \ref{fig5a:thm} ). Consequently, \(t_2^- \to s_2^+\). However, the \((t_2,s_2)\)-path \(t_2 t_2^+ t_2^- s_2^+ s_2^- s_2\) is also disjoint from \(P_i\) for all \(i \in \{2,3\}\), yielding another contradiction to \eqref{eq:contradiction} (See Figure \ref{fig5b:thm}). Hence, we may assume that
\(D - \{s_2, t_2\}\) contains three internally disjoint \((s_1, t_1)\)-paths.
Since $\{s_1,t_1,s_2,t_2\}\subseteq V_1$ and \(D\) is \(4\)-strong, by symmetry, $D-\{s_1,t_1\}$ contains three internally disjoint $(s_2,t_2)$‑paths.\end{proof}

Combining Claim \ref{claim:case2--there is 3 (s_i,t_i)-paths in D-{s_{3-i},t_{3-i}}} and Lemma \ref{Lemma:{s_1,t_1} in V_1 : (3,3)}, \(D\) contains a pair of disjoint \((s_1,t_1)\)-path and \((s_2,t_2)\)-path, which contradicts \eqref{eq:contradiction}.

\vspace{2mm}
\mycase \label{case3} 
$|\{s_i, t_i\} \cap V_1| = 1$ and $|\{s_{3-i}, t_{3-i}\} \cap V_1| = 2$, {for some} $i \in \{1,2\}$.
\vspace{2mm}

W.l.o.g., {assume that} \(i=2\), \(s_2\in V_1\) and \(t_2\in V_2\) (The case \(t_2\in V_1\) and \(s_2\in V_2\) follows by reversing all arcs of \(D\) to get the digraph \(D'\), and the same argument applies).

\begin{claim}\label{claim: s_2 not in V(P_1),V(P_2),V(P_3), V(P_4)}
\(s_2 \notin S\)
\end{claim}

\begin{proof}
   Suppose that \(s_2 \in V(P_{j'})\) for some \(j' \in [4]\). Let \(s_2^+{\in V_2}\) denote the successor of \(s_2\) along \(P_{j'}[s_2,t_1]\).
    By the minimality of \(P_{j'}\), we have \(s_2^+ \to s_1^{j'}\), where \(|P_{j'}[s_1,s_2]| \geq 3\) and \(s_1^{j'} \in V_2\).
   If \(s_1^{j'} \to t_2\), then the \((s_2,t_2)\)-path \(s_2 s_2^+ s_1^{j'} t_2\) is disjoint from \(P_i\) for some \(i \in \intfour \setminus \{j'\}\). This forces \(t_2 \to s_1^{j'}\) {and then} the \((t_2,s_2)\)-path \(t_2s_1^{j'} \circ P_{j'}[s_1^{j'},s_2]\) is also disjoint from \(P_i\) for some \(i \in \intfour \setminus \{j'\}\){, a contradiction}. Consequently, \(s_2 \notin S\).
\end{proof}

By Claim \ref{claim: s_2 not in V(P_1),V(P_2),V(P_3), V(P_4)}, there exist at least three internally disjoint \((s_1,t_1)\)-paths in \(D-\{s_2,t_2\}\). If \(t_1\) or \(s_1\) does not belong to \(V(Q_1) \cup V(Q_2) \cup V(Q_3) \cup V(Q_4)\), then \(D-\{s_1,t_1\}\) contains at least three internally disjoint \((s_2,t_2)\)-paths. Since \(\{s_1,t_1\} \subseteq V_1\), {Lemma \ref{Lemma:{s_1,t_1} in V_1 : (3,3)} implies that $D$ }admits a pair of disjoint \((s_1,t_1)\)-path and \((s_2,t_2)\)-path, {contradicting}  \eqref{eq:contradiction}. Combining this with the fact that \(D\) is \(4\)-strong, we conclude that \(D-\{s_1,t_1\}\) has exactly two internally disjoint \((s_2,t_2)\)-paths.

This implies that \(s_1 \in V(Q_i)\) and \(t_1 \in V(Q_j)\) for distinct indices \(i,j \in \intfour\). W.l.o.g., assume \(s_1 \in V(Q_1)\) and \(t_1 \in V(Q_4)\). Consequently, \(Q_2\) and \(Q_3\) are also internally disjoint \((s_2,t_2)\)-paths in \(D-\{s_1,t_1\}\). Since \(s_1, s_2 \in V_1\), we {know that } \(s_1^i \in V_2\) for all \(i \in \intfour\) and \(s_2^4 \in V_2\).%

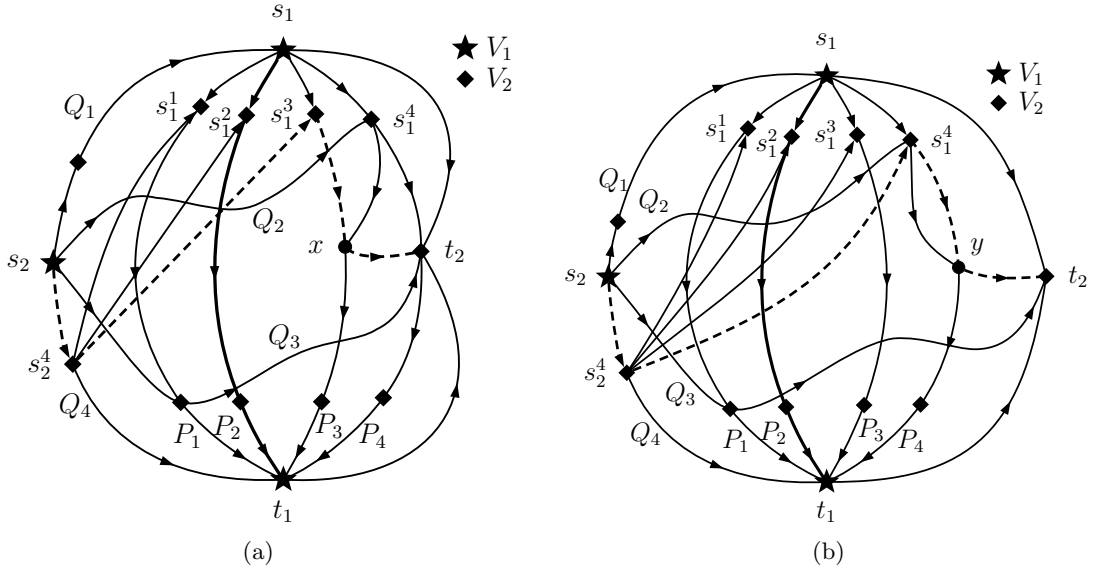
\begin{figure}[H]
\centering
\tikzset{pics/manuscriptdirection/.style={code={
  \fill[black] (0pt,0pt) -- (-6pt,2pt) -- (-6pt,-2pt) -- cycle;
}}}

\tikzset{
    figseven/.style={font=\large,line cap=round,line join=round},
    fsevenline/.style={draw=black,line width=0.8pt},
    fsevensolid/.style={draw=black,line width=1.4pt,solid},
    fsevendash/.style={draw=black,line width=1.2pt,
        dash pattern=on 4pt off 2.5pt,line cap=butt},
    fsevenarrow/.style={-{Triangle[length=6pt,width=4pt,line width=0pt]},
        shorten >=0pt},
    pics/figsevenarrow/.style={code={
        \fill[black] (0pt,0pt) -- (-6pt,2pt) -- (-6pt,-2pt) -- cycle;
    }},
    fsevenstar/.style={shape=star,star points=5,star point ratio=2.35,
        draw=none,fill=black,minimum size=12pt,inner sep=0pt,outer sep=0pt},
    fsevendiamond/.style={diamond,draw=none,fill=black,
        minimum size=7.5pt,inner sep=0pt,outer sep=0pt},
    fsevendot/.style={circle,draw=black,fill=black,
        minimum size=5.5pt,inner sep=0pt,outer sep=0pt},
    fsevenlabel/.style={inner sep=1pt}
}
\begin{subfigure}{0.43\textwidth}
\centering
\resizebox{\linewidth}{!}{%
\begin{tikzpicture}[figseven]
\path[use as bounding box] (-4.18,-0.64) rectangle (3.43,7.00);

\node[fsevenstar] (s1) at (0,6.40) {};
\node[fsevenstar] (t1) at (0,0) {};
\node[fsevenstar] (s2) at (-3.43,3.23) {};
\node[fsevendiamond] (t2) at (2.05,3.40) {};
\node[fsevendiamond] (s11) at (-1.23,5.56) {};
\node[fsevendiamond] (s12) at (-0.55,5.43) {};
\node[fsevendiamond] (s13) at (0.48,5.45) {};
\node[fsevendiamond] (s14) at (1.31,5.37) {};
\node[fsevendiamond] (s24) at (-3.14,1.72) {};
\node[fsevendiamond] (q1d) at (-3.06,4.73) {};
\node[fsevendiamond] (p1d) at (-1.53,1.15) {};
\node[fsevendiamond] (p2d) at (-0.64,1.16) {};
\node[fsevendiamond] (p3d) at (0.57,1.16) {};
\node[fsevendiamond] (p4d) at (1.49,1.22) {};
\node[fsevendot] (x) at (0.92,3.47) {};

\draw[fsevenline,fsevenarrow]
    (s1) .. controls (-0.45,6.20) and (-0.92,5.91) .. (s11);
\draw[fsevenline]
    (s11.center) .. controls (-2.61,3.95) and (-2.36,2.30) .. (p1d.center)
    pic[pos=0.55,sloped,allow upside down] {manuscriptdirection}
    .. controls (-1.10,0.62) and (-0.49,0.22) .. (t1.center)
    pic[pos=0.65,sloped,allow upside down] {manuscriptdirection};

\draw[fsevensolid,fsevenarrow]
    (s1) .. controls (-0.20,6.08) and (-0.41,5.72) .. (s12);
\draw[fsevensolid]
    (s12.center) .. controls (-1.17,3.90) and (-1.16,2.40) .. (p2d.center)
    pic[pos=0.55,sloped,allow upside down] {manuscriptdirection}
    .. controls (-0.46,0.69) and (-0.18,0.25) .. (t1.center)
    pic[pos=0.65,sloped,allow upside down] {manuscriptdirection};

\draw[fsevenline,fsevenarrow]
    (s1) .. controls (0.18,6.13) and (0.35,5.83) .. (s13);
\draw[fsevenline]
    (x.center) .. controls (0.97,2.62) and (0.85,1.86) .. (p3d.center)
    pic[pos=0.45,sloped,allow upside down] {manuscriptdirection}
    .. controls (0.38,0.66) and (0.14,0.23) .. (t1.center)
    pic[pos=0.65,sloped,allow upside down] {manuscriptdirection};

\draw[fsevenline]
    (s1.center) .. controls (0.45,6.20) and (0.96,5.82) .. (s14.center)
    pic[pos=0.72,sloped,allow upside down] {manuscriptdirection}
    .. controls (1.74,4.81) and (2.08,4.09) .. (t2.center)
    pic[pos=0.60,sloped,allow upside down] {manuscriptdirection}
    .. controls (2.11,2.46) and (1.89,1.74) .. (p4d.center)
    pic[pos=0.55,sloped,allow upside down] {manuscriptdirection}
    .. controls (1.15,0.75) and (0.52,0.22) .. (t1.center)
    pic[pos=0.70,sloped,allow upside down] {manuscriptdirection};

\draw[fsevenline]
    (s2.center) .. controls (-3.35,3.83) and (-3.23,4.32) .. (q1d.center)
    pic[pos=0.60,sloped,allow upside down] {manuscriptdirection}
    .. controls (-2.50,6.00) and (-1.70,6.45) .. (s1.center)
    pic[pos=0.65,sloped,allow upside down] {manuscriptdirection}
    .. controls (2.59,6.43) and (2.93,5.02) .. (t2.center)
    pic[pos=0.75,sloped,allow upside down] {manuscriptdirection};

\draw[fsevenline]
    (s2.center) .. controls (-2.66,4.12) and (-2.35,4.30) .. (-1.90,4.20)
    pic[pos=0.35,sloped,allow upside down] {manuscriptdirection}
    .. controls (-1.30,4.15) and (-0.77,3.86) .. (-0.28,4.16)
    .. controls (0.42,4.50) and (0.87,5.25) .. (s14.center)
    pic[pos=0.55,sloped,allow upside down] {manuscriptdirection}
    .. controls (1.66,4.51) and (1.23,3.89) .. (x.center)
    pic[pos=0.55,sloped,allow upside down] {manuscriptdirection};

\draw[fsevenline]
    (s2.center) .. controls (-2.73,2.43) and (-2.02,1.25) .. (p1d.center)
    pic[pos=0.30,sloped,allow upside down] {manuscriptdirection}
    .. controls (-0.99,1.03) and (-0.36,1.63) .. (0.18,1.86)
    pic[pos=0.50,sloped,allow upside down] {manuscriptdirection}
    .. controls (0.62,2.08) and (1.12,2.06) .. (1.39,2.24)
    .. controls (1.81,2.43) and (2.00,2.99) .. (t2.center)
    pic[pos=0.75,sloped,allow upside down] {manuscriptdirection};

\draw[fsevenline]
    (s24.center) .. controls (-2.66,0.47) and (-1.52,-0.08) .. (t1.center)
    pic[pos=0.60,sloped,allow upside down] {manuscriptdirection}
    .. controls (2.14,-0.13) and (3.42,1.20) .. (t2.center)
    pic[pos=0.65,sloped,allow upside down] {manuscriptdirection};

\draw[fsevenline,fsevenarrow]
    (s24) .. controls (-2.60,4.10) and (-2.05,4.72) .. (s11);
\draw[fsevenline,fsevenarrow] (s24) -- (s12);

\draw[fsevendash,fsevenarrow]
    (s2) .. controls (-3.40,2.61) and (-3.26,1.95) .. (s24);
\draw[fsevendash,fsevenarrow] (s24) -- (s13);
\draw[fsevendash]
    (s13.center) .. controls (0.78,4.63) and (0.89,3.99) .. (x.center)
    pic[pos=0.60,sloped,allow upside down] {figsevenarrow}
    .. controls (1.09,3.27) and (1.60,3.30) .. (t2.center)
    pic[pos=0.60,sloped,allow upside down] {figsevenarrow};

\node[fsevenlabel,above=5pt] at (s1.north) {$s_1$};
\node[fsevenlabel,below=5pt] at (t1.south) {$t_1$};
\node[fsevenlabel,left=5pt] at (s2.west) {$s_2$};
\node[fsevenlabel,right=5pt] at (t2.east) {$t_2$};
\node[fsevenlabel,left=4pt] at (s11.west) {$s_1^1$};
\node[fsevenlabel,left=1pt,yshift=-3pt] at (s12.west) {$s_1^2$};
\node[fsevenlabel,left=4pt] at (s13.west) {$s_1^3$};
\node[fsevenlabel,right=4pt] at (s14.east) {$s_1^4$};
\node[fsevenlabel,left=4pt] at (s24.west) {$s_2^4$};
\node[fsevenlabel,left=5pt] at (x.west) {$x$};
\node[fsevenlabel] at (-1.43,0.65) {$P_1$};
\node[fsevenlabel] at (-0.84,0.78) {$P_2$};
\node[fsevenlabel] at (0.68,0.80) {$P_3$};
\node[fsevenlabel] at (1.33,0.65) {$P_4$};
\node[fsevenlabel] at (-3.03,5.55) {$Q_1$};
\node[fsevenlabel] at (-0.22,3.85) {$Q_2$};
\node[fsevenlabel] at (0.02,2.13) {$Q_3$};
\node[fsevenlabel] at (-3.10,1.10) {$Q_4$};

\node[fsevenstar] at (2.71,6.43) {};
\node[fsevendiamond] at (2.71,5.97) {};
\node[fsevenlabel,anchor=west] at (2.98,6.43) {$V_1$};
\node[fsevenlabel,anchor=west] at (2.98,5.97) {$V_2$};
\end{tikzpicture}%
}
\caption{}
\label{fig7a:thm}
\end{subfigure}
\hspace{4mm}
\begin{subfigure}{0.45\textwidth}
\centering
\resizebox{\linewidth}{!}{%
\begin{tikzpicture}[figseven]
\path[use as bounding box] (-4.18,-0.64) rectangle (4.24,7.00);

\node[fsevenstar] (s1) at (0,6.40) {};
\node[fsevenstar] (t1) at (0,0) {};
\node[fsevenstar] (s2) at (-3.44,3.23) {};
\node[fsevendiamond] (t2) at (3.46,3.24) {};
\node[fsevendiamond] (s11) at (-1.24,5.57) {};
\node[fsevendiamond] (s12) at (-0.55,5.44) {};
\node[fsevendiamond] (s13) at (0.48,5.47) {};
\node[fsevendiamond] (s14) at (1.32,5.39) {};
\node[fsevendiamond] (s24) at (-3.15,1.72) {};
\node[fsevendiamond] (q1d) at (-3.29,4.10) {};
\node[fsevendiamond] (p1d) at (-1.52,1.15) {};
\node[fsevendiamond] (p2d) at (-0.64,1.18) {};
\node[fsevendiamond] (p3d) at (0.59,1.21) {};
\node[fsevendiamond] (p4d) at (1.48,1.22) {};
\node[fsevendot] (y) at (2.08,3.38) {};

\draw[fsevenline,fsevenarrow]
    (s1) .. controls (-0.45,6.20) and (-0.92,5.92) .. (s11);
\draw[fsevenline]
    (s11.center) .. controls (-2.61,3.95) and (-2.36,2.30) .. (p1d.center)
    pic[pos=0.55,sloped,allow upside down] {manuscriptdirection}
    .. controls (-1.10,0.62) and (-0.49,0.22) .. (t1.center)
    pic[pos=0.65,sloped,allow upside down] {manuscriptdirection};

\draw[fsevensolid,fsevenarrow]
    (s1) .. controls (-0.20,6.08) and (-0.41,5.73) .. (s12);
\draw[fsevensolid]
    (s12.center) .. controls (-1.17,3.91) and (-1.16,2.42) .. (p2d.center)
    pic[pos=0.55,sloped,allow upside down] {manuscriptdirection}
    .. controls (-0.46,0.70) and (-0.18,0.25) .. (t1.center)
    pic[pos=0.65,sloped,allow upside down] {manuscriptdirection};

\draw[fsevenline,fsevenarrow]
    (s1) .. controls (0.18,6.13) and (0.35,5.85) .. (s13);
\draw[fsevenline]
    (s13.center) .. controls (1.08,3.96) and (1.09,2.57) .. (p3d.center)
    pic[pos=0.55,sloped,allow upside down] {manuscriptdirection}
    .. controls (0.38,0.66) and (0.14,0.23) .. (t1.center)
    pic[pos=0.65,sloped,allow upside down] {manuscriptdirection};

\draw[fsevenline,fsevenarrow]
    (s1) .. controls (0.45,6.20) and (0.97,5.84) .. (s14);
\draw[fsevenline]
    (y.center) .. controls (2.12,2.55) and (1.85,1.74) .. (p4d.center)
    pic[pos=0.55,sloped,allow upside down] {manuscriptdirection}
    .. controls (1.12,0.70) and (0.45,0.20) .. (t1.center)
    pic[pos=0.65,sloped,allow upside down] {manuscriptdirection};

\draw[fsevenline]
    (s2.center) .. controls (-3.43,3.54) and (-3.37,3.82) .. (q1d.center)
    pic[pos=0.60,sloped,allow upside down] {manuscriptdirection}
    .. controls (-2.80,5.91) and (-1.75,6.57) .. (s1.center)
    pic[pos=0.65,sloped,allow upside down] {manuscriptdirection}
    .. controls (2.29,6.24) and (2.83,5.61) .. (t2.center)
    pic[pos=0.75,sloped,allow upside down] {manuscriptdirection};

\draw[fsevenline]
    (s2.center) .. controls (-2.66,4.12) and (-2.35,4.30) .. (-1.90,4.20)
    pic[pos=0.35,sloped,allow upside down] {manuscriptdirection}
    .. controls (-1.25,4.06) and (-0.90,3.97) .. (-0.34,4.22)
    .. controls (0.27,4.44) and (0.84,5.19) .. (s14.center)
    pic[pos=0.55,sloped,allow upside down] {manuscriptdirection}
    .. controls (1.39,4.22) and (1.19,3.93) .. (y.center)
    pic[pos=0.55,sloped,allow upside down] {manuscriptdirection};

\draw[fsevenline]
    (s2.center) .. controls (-2.73,2.43) and (-2.02,1.25) .. (p1d.center)
    pic[pos=0.30,sloped,allow upside down] {manuscriptdirection}
    .. controls (-0.88,1.01) and (0.20,2.12) .. (1.05,2.26)
    pic[pos=0.50,sloped,allow upside down] {manuscriptdirection}
    .. controls (1.53,2.30) and (2.00,2.09) .. (2.34,2.09)
    .. controls (2.92,2.09) and (3.23,2.56) .. (t2.center)
    pic[pos=0.75,sloped,allow upside down] {manuscriptdirection};

\draw[fsevenline]
    (s24.center) .. controls (-2.66,0.47) and (-1.52,-0.08) .. (t1.center)
    pic[pos=0.60,sloped,allow upside down] {manuscriptdirection}
    .. controls (2.35,-0.03) and (3.13,0.88) .. (t2.center)
    pic[pos=0.65,sloped,allow upside down] {manuscriptdirection};

\draw[fsevenline,fsevenarrow]
    (s24) .. controls (-2.33,2.97) and (-1.59,4.40) .. (s11);
\draw[fsevenline,fsevenarrow]
    (s24) .. controls (-2.10,2.97) and (-1.15,4.03) .. (s12);
\draw[fsevenline,fsevenarrow]
    (s24) .. controls (-1.33,3.14) and (-0.48,3.66) .. (s13);

\draw[fsevendash,fsevenarrow]
    (s2) .. controls (-3.40,2.61) and (-3.26,1.95) .. (s24);
\draw[fsevendash,fsevenarrow]
    (s24) .. controls (-0.89,2.61) and (0.32,3.19) .. (s14);
\draw[fsevendash]
    (s14.center) .. controls (1.75,4.88) and (2.01,4.18) .. (y.center)
    pic[pos=0.60,sloped,allow upside down] {figsevenarrow}
    .. controls (2.49,3.17) and (2.97,3.20) .. (t2.center)
    pic[pos=0.60,sloped,allow upside down] {figsevenarrow};

\node[fsevenlabel,above=5pt] at (s1.north) {$s_1$};
\node[fsevenlabel,below=5pt] at (t1.south) {$t_1$};
\node[fsevenlabel,left=5pt] at (s2.west) {$s_2$};
\node[fsevenlabel,right=5pt] at (t2.east) {$t_2$};
\node[fsevenlabel,left=4pt] at (s11.west) {$s_1^1$};
\node[fsevenlabel,left=1pt,yshift=-3pt] at (s12.west) {$s_1^2$};
\node[fsevenlabel,left=4pt] at (s13.west) {$s_1^3$};
\node[fsevenlabel,right=4pt] at (s14.east) {$s_1^4$};
\node[fsevenlabel,left=4pt] at (s24.west) {$s_2^4$};
\node[fsevenlabel,above right=3pt] at (y.north east) {$y$};
\node[fsevenlabel] at (-1.43,0.65) {$P_1$};
\node[fsevenlabel] at (-0.84,0.78) {$P_2$};
\node[fsevenlabel] at (0.70,0.82) {$P_3$};
\node[fsevenlabel] at (1.33,0.65) {$P_4$};
\node[fsevenlabel] at (-3.35,4.75) {$Q_1$};
\node[fsevenlabel] at (-2.72,4.42) {$Q_2$};
\node[fsevenlabel] at (-2.30,1.38) {$Q_3$};
\node[fsevenlabel] at (-2.85,0.75) {$Q_4$};

\node[fsevenstar] at (2.71,6.43) {};
\node[fsevendiamond] at (2.71,5.97) {};
\node[fsevenlabel,anchor=west] at (2.98,6.43) {$V_1$};
\node[fsevenlabel,anchor=west] at (2.98,5.97) {$V_2$};
\end{tikzpicture}%
}
\caption{}
\label{fig7b:thm}
\end{subfigure}
\caption{\(s_2 \notin S\). \(D-\{s_1,t_1\}\) has exactly two internally disjoint \((s_2,t_2)\)-paths. (a) \(t_2 \in V(P_{k'})\) for some \(k' \in \intfour\); (b) \(t_2 \notin V(P_i)\) for all \(i \in \intfour\).}
\label{fig7:thm}
\end{figure}

Suppose that \(t_2 \in V(P_{k'})\) for some \(k' \in \intfour\). 
If \(s_1^i \to s_2^4\) for some \(i \in \intfour \setminus \{k'\}\), then the \((s_1,t_1)\)-path \(s_1s_1^i \circ s_1^is_2^4 \circ Q_4[s_2^4,t_1]\) is disjoint from \(Q_j\) for some \(j \in \{2,3\}\), which contradicts \eqref{eq:contradiction}. Therefore, 
\begin{equation}\label{equ:s_2^4 -> s_1^i}
s_2^4 \to s_1^i \text{ for every } i \in \intfour \setminus \{k'\}.
\end{equation}
Furthermore, if \(s_1^i \to t_2\) for some \(i \in \intfour \setminus \{k'\}\), then the \((s_2,t_2)\)-path \(s_2 s_2^4 s_1^i t_2\) is disjoint from \(P_j\) for some \(j \in \intfour \setminus \{k'\}\), again yielding a contradiction to \eqref{eq:contradiction}. Thus, we have \(t_2 \to s_1^i\) for every \(i \in \intfour \setminus \{k'\}\).

Let \(x\) be the last vertex of \(Q_2\) contained in \(S \setminus V(P_{k'})\). Then \(x \in V(P_{l'})\) for some \(l' \in \intfour \setminus \{k'\}\). The \((s_2,t_2)\)-path
\(
s_2s_2^4 \circ s_2^4s_1^{l'} \circ P_{l'}[s_1^{l'},x] \circ Q_2[x,t_2]
\)
is disjoint from \(P_i\) for some \(i \in \intfour \setminus \{k',l'\}\), which contradicts \eqref{eq:contradiction} (See Figure \ref{fig7a:thm}).
Consequently, we may assume \(t_2 \notin S\). %

Combining \(t_2 \notin V(P_i)\) for all \(i \in \intfour\) with Claim \ref{claim: s_2 not in V(P_1),V(P_2),V(P_3), V(P_4)}, $P_1,\ldots{},P_4$ form a set of four internally disjoint \((s_1,t_1)\)-paths in \(D-\{s_2,t_2\}\). Let \(y\) be the last vertex of \(Q_2\) that lies in \(S\), and assume \(y \in V(P_{i'})\) for some \(i' \in \intfour\). Similarly to \eqref{equ:s_2^4 -> s_1^i}, %
we deduce that \(s_2^4 \to s_1^i\) and \(t_2 \to s_1^i\) hold for all \(i \in \intfour\). Then the \((s_2,t_2)\)-path
\(
s_2s_2^4 \circ s_2^4s_1^{i'} \circ P_{i'}[s_1^{i'},y] \circ Q_2[y,t_2]
\)
is disjoint from \(P_i\) for some \(i \in \intfour \setminus \{i'\}\), {contradicting} \eqref{eq:contradiction} (See Figure \ref{fig7b:thm}).  This completes the proof of the theorem.

\end{proof}

The next result shows that the bound in Theorem \ref{thm:mainsplit} is best possible.

\begin{theorem}\label{thm:main}
For every integer $n\geq 8$, there is an $n$-vertex $3$-strong split digraph that is not unilaterally $2$-linked. 
\end{theorem}

\begin{proof}
Let $m=n-8$ and $X=\{x_1,\ldots,x_m\}$ (possibly empty), and set
\[
I=\{s_1,s_2\}\cup X,\; 
T=\{t_1,t_2,v_1,v_2,v_3,v_4\},\;
U=\{v_1,v_2\} \;\text{and}\; W=\{v_3,v_4\}.
\]
We construct a digraph $D_n$ with vertex set $I\cup T$ as follows. 

\emph{Arcs inside $T$.} The six vertices of $T$ induce a tournament $H$ defined by
\begin{align}
&t_1\to\{t_2,v_1,v_2\},\;\{v_3,v_4\}\to t_1,\;
\{v_1,v_2\}\to t_2\to\{v_3,v_4\}, \label{eq:T1}\\[-2pt]
&v_1\to v_2,\; v_4\to v_3\;\text{and}\;
v_1\to v_4\to v_2\to v_3\to v_1.\label{eq:T2}
\end{align}

\emph{Arcs incident with $s_1,s_2$.} There is no arc between $s_1$ and $s_2$, and we add exactly
\begin{align}
&s_1\to\{v_1,v_2,t_2\},\; \{v_3,v_4,t_2\}\to s_1,\label{eq:s1}\\[-2pt]
&s_2\to\{t_1,v_3,v_4\},\; \{t_1,v_1,v_2\}\to s_2.\label{eq:s2}
\end{align}
Thus $s_1t_2,t_2s_1,s_2t_1,t_1s_2$ are all arcs.

\emph{Arcs incident with $X$.} For every $x\in X$, add exactly
\begin{equation}\label{eq:extra}
\{t_1,v_1,v_2\}\to x\;\text{and}\; x\to\{t_2,v_1,v_2\}.
\end{equation}
There are no arcs among vertices of $I$, and there are no other arcs. Hence $I$ is independent and $T$ induces a tournament; in particular, $D_n$ is a split digraph. The construction is illustrated schematically
in Figure~\ref{fig+construction}.

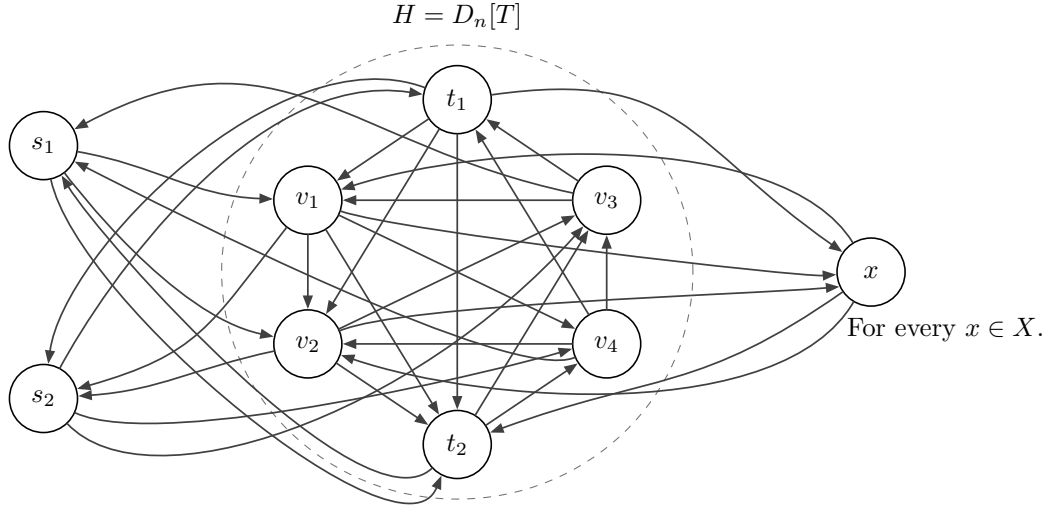
\begin{figure}[htbp]
\centering
\begin{tikzpicture}[
  scale=0.76,
  >={Triangle[length=1.8mm,width=1.3mm]},
  vert/.style={
    circle,
    draw=black,
    fill=white,
    line width=0.6pt,
    minimum size=9mm,
    inner sep=0pt,
    font=\normalsize
  },
  arc/.style={
    ->,
    draw=black!75,
    line width=0.65pt,
    line cap=rect
  },
  redarc/.style={
    ->,
    draw=red!75!black,
    line width=1pt,
    line cap=rect
  },
  bluearc/.style={
    ->,
    draw=blue!70!black,
    line width=1pt,
    line cap=rect
  },
  lab/.style={font=\normalsize}
]

\node[vert] (s1) at (-7.2, 2.2) {$s_1$};
\node[vert] (s2) at (-7.2,-2.2) {$s_2$};

\node[vert] (t1) at (0,3.0) {$t_1$};
\node[vert] (t2) at (0,-3.0) {$t_2$};

\node[vert] (v1) at (-2.6, 1.25) {$v_1$};
\node[vert] (v2) at (-2.6,-1.25) {$v_2$};
\node[vert] (v3) at (2.6, 1.25) {$v_3$};
\node[vert] (v4) at (2.6,-1.25) {$v_4$};

\node[vert] (x) at (7.2,0) {$x$};

\draw[dashed,black!55] (0,0)
  ellipse [x radius=4.1cm,y radius=3.95cm];

\node[lab] at (0,4.45) {$H=D_n[T]$};

\node[lab] at (8.5,-1)
  {For every $x\in X$.};

\draw[arc] (t1) -- (t2);
\draw[arc] (t1) -- (v1);
\draw[arc] (t1) -- (v2);

\draw[arc] (v3) -- (t1);
\draw[arc] (v4) -- (t1);

\draw[arc] (v1) -- (t2);
\draw[arc] (v2) -- (t2);

\draw[arc] (t2) -- (v3);
\draw[arc] (t2) -- (v4);

\draw[arc] (v1) -- (v2);
\draw[arc] (v4) -- (v3);

\draw[arc] (v1) -- (v4);
\draw[arc] (v2) -- (v3);

\draw[arc] (v3) -- (v1);
\draw[arc] (v4) -- (v2);

\draw[arc]
  (s1) to[out=-10,in=178,looseness=1.02] (v1);

\draw[arc]
  (s1) .. controls (-6.70,1.55)
               and (-4.95,-0.75) .. (v2);

\draw[arc]
  (s1) .. controls (-6.60,-0.55)
               and (-0.95,-5.20) .. (t2.243);

\draw[arc]
  (v3) to[out=168,in=28,looseness=1.18] (s1);

\draw[arc]
  (v4) .. controls (0.95,-1.95)
               and (-4.95,1.05) .. (s1);

\draw[arc]
  (t2) .. controls (-1.65,-4.55)
               and (-6.20,0.55) .. (s1);

\draw[arc]
  (t1) .. controls (-3.10,4.15)
               and (-6.85,0.35) .. (s2);

\draw[arc]
  (v1) to[out=232,in=16,looseness=1.22] (s2);

\draw[arc]
  (v2) to[out=192,in=4,looseness=1.04] (s2);

\draw[arc]
  (s2) .. controls (-6.15,-0.25)
               and (-3.55,3.55) .. (t1);

\draw[arc]
  (s2) .. controls (-5.4,-4.1)
               and (-0.8,-2.55) .. (v3);

\draw[arc]
  (s2) .. controls (-5.10,-3.15)
               and (0.95,-1.65) .. (v4.188);

\draw[arc]
  (t1) to[out=8,in=145,looseness=1.28] (x);

\draw[arc]
  (v1) .. controls (-0.95,0.7)
               and (5.95,-0.09) .. (x);

\draw[arc]
  (v2) .. controls (-0.90,-0.55)
               and (4.95,-0.40) .. (x.205);

\draw[arc]
  (x) .. controls (5.65,2.50)
             and (0.5,2.25) .. (v1);

\draw[arc]
  (x) .. controls (5.65,-2.65)
             and (0.5,-2.30) .. (v2);

\draw[arc]
  (x) to[out=-145,in=20,looseness=1.18] (t2);

\end{tikzpicture}
\vspace{-0.6cm}
\caption{A schematic illustration of an
$n$-vertex 3-strong split digraph $D_n$, on $n\geq 8$ vertices that is not unilaterally 2-linked. The vertex $x$ represents an arbitrary vertex of $X$.}

\label{fig+construction}
\end{figure}

\begin{claim}
$D_n$ is $3$-strong.
\end{claim}

First observe that every vertex of the tournament $H$ has both in-degree and out-degree at least $2$ in $H$. We use the elementary fact that a tournament on at most five vertices with minimum in-degree and minimum out-degree at least $1$ is strong. Indeed, in a non-strong tournament the initial and final strong components are distinct. Neither can be a singleton under these degree assumptions; as nontrivial strong tournaments have at least three vertices, there would be at least six vertices.

It follows that $H$ is $2$-strong: deletion of any one vertex leaves a tournament on five vertices with positive in-degree and out-degree at every remaining vertex. Similarly, after deleting two vertices of $H$, the resulting four-vertex tournament fails to be strong precisely when some remaining vertex has in-degree or out-degree zero. By \eqref{eq:T1}--\eqref{eq:T2}, this occurs for exactly the five pairs
\begin{equation}\label{eq:exceptions}
\mathcal{F}=\bigl\{\{v_1,t_2\},\{v_1,t_1\},\{v_3,t_2\},\{v_3,v_4\},\{v_3,t_1\}\bigr\}.
\end{equation}
For completeness, the exceptional pairs are precisely the in- or out-neighbourhoods of size two in $H$: $N_H^-(t_1)=N_H^+(t_2)=\{v_3,v_4\}$, $N_H^-(v_1)=\{t_1,v_3\}$, $N_H^+(v_2)=\{t_2,v_3\}$, $N_H^+(v_3)=\{t_1,v_1\}$, and $N_H^-(v_4)=\{t_2,v_1\}$.

Let $D_0=D_n[T\cup\{s_1,s_2\}]$, the same eight-vertex digraph for every $m$. For each $S\in\mathcal F$, the following table exhibits a directed cycle in $D_0-S$. The cycle is spanning unless a vertex appears in the last column; such a remaining vertex has an in-neighbour and an out-neighbour on the displayed cycle, and hence $D_0-S$ is strong.
\begin{table}[htbp]
\centering
\small
\renewcommand{\arraystretch}{1.2}
\begin{tabular}{@{}clc@{}}
\toprule
$S$ & Directed cycle in $D_0-S$ & Omitted vertex \\
\midrule
$\{v_1,t_2\}$ &
$v_4\to v_2\to v_3\to t_1\to s_2\to v_4$ &
$s_1$ \\

$\{v_1,t_1\}$ &
$t_2\to v_4\to v_2\to v_3\to s_1\to t_2$ &
$s_2$ \\

$\{v_3,t_2\}$ &
$t_1\to v_1\to v_4\to v_2\to s_2\to t_1$ &
$s_1$ \\

$\{v_3,v_4\}$ &
$t_1\to t_2\to s_1\to v_1\to v_2\to s_2\to t_1$ &
none \\

$\{v_3,t_1\}$ &
$v_2\to t_2\to v_4\to s_1\to v_1\to v_2$ &
$s_2$ \\
\bottomrule
\end{tabular}
\end{table}
Consequently, $D_0-S$ is strong whenever $S\subseteq T$ and $|S|\le 2$: either $H-S$ is strong or $S\in\mathcal F$. If $S$ contains at least one vertex outside $T$, then $|S\cap T|\le 1$, so $H-S$ is strong. In this case each remaining vertex of $\{s_1,s_2\}$ has at least one in-neighbour and one out-neighbour in $H-S$, since it originally has three of each. Thus $D_0-S$ is also strong.

Now let $S\subseteq V(D_n)$ with $|S|\le2$. The subdigraph induced by $V(D_0)\setminus S$ is strong by the preceding paragraph. By \eqref{eq:extra}, every $x\in X\setminus S$ has three
in-neighbours and three out-neighbours in $T$, so at least one of each survives in $D_0-S$. Consequently, all vertices of $D_n-S$ belong to one strong component. This proves that $D_n$ is $3$-strong. Moreover, $d^+_{D_n}(s_1)=3$, and therefore deletion of its three out-neighbours destroys strong connectivity. Hence $\kappa(D_n)=3$.

\begin{claim}
    $D_n$ is not unilaterally $2$-linked.
\end{claim}

Consider the terminal pairs $\{s_1,t_1\}$ and $\{s_2,t_2\}$. Suppose, to the contrary, that there are vertex-disjoint directed paths $P,Q$ joining these two pairs, respectively, with either orientation allowed.

The path $P$ is contained in $D_n-\{s_2,t_2\}$. In this digraph both $s_1$ and $t_1$ have all their out-neighbours in $U\cup X$, while all their in-neighbours lie in $W$. Thus $P$ must reach $W$ after leaving its first terminal. The only arcs entering $W$ from $\{s_1,t_1\}\cup U\cup X$ are
\begin{equation}\label{eq:forward}
v_1v_4\;\text{and}\; v_2v_3.
\end{equation}
Therefore $P$ contains at least one of these two arcs, regardless of its orientation.

Similarly, $Q$ is contained in $D_n-\{s_1,t_1\}$. In this digraph both $s_2$ and $t_2$ have all their out-neighbours in $W$, while all their in-neighbours lie in $U\cup X$. Since no vertex of $W$ sends an arc to $X$, $Q$ must contain at least one of the two arcs
\begin{equation}\label{eq:backward}
v_3v_1\;\text{and}\; v_4v_2.
\end{equation}
But every arc in \eqref{eq:forward} shares an endpoint with every arc in \eqref{eq:backward}. Hence $P$ and $Q$ share a vertex, contradicting their assumed disjointness. The digraph $D_n$ is therefore not unilaterally $2$-linked. This completes the proof of the theorem.
\end{proof}

\section{Unilateral 2-linkages in  semicomplete multipartite digraphs}\label{sec:smdsec}

In this section, we turn to semicomplete multipartite digraphs. For the unilateral $k$-linkage problem, we can make use of the corresponding result for the $k$-linkage problem to obtain the following.

\begin{corollary}
The unilateral $k$-linkage problem is polynomial for semicomplete multipartite digraphs.
\end{corollary}

\begin{proof}
By (C), the $k$-linkage problem can be solved in polynomial time for semicomplete multipartite digraphs.
Let $\{s_1,t_1\},\ldots,\{s_k,t_k\}$ be the given terminal pairs.
For each $i\in[k]$, a unilateral linkage may use either an $(s_i,t_i)$-path or a $(t_i,s_i)$-path. Thus, there are exactly $2^k$ possible choices of directions for the $k$ terminal pairs.

For each of these $2^k$ choices, we solve the corresponding ordinary $k$-linkage problem. The original unilateral $k$-linkage problem has
a solution if and only if at least one of these $2^k$ ordinary $k$-linkage problems has a solution. Since $k$ is fixed, $2^k$ is constant. Hence the unilateral $k$-linkage problem is polynomial time solvable.
\end{proof}

For $6$-strong semicomplete multipartite digraphs, Chen et al. \cite{chen2linksplit} obtained the following result.

\begin{proposition}\cite{chen2linksplit}\label{6-strong-2-linked}
Every 6-strong semicomplete multipartite digraph is 2-linked.
\end{proposition}

Since every $2$-linked digraph is unilaterally $2$-linked, Proposition~\ref{6-strong-2-linked} also gives a sufficient condition for unilateral $2$-linkage. We show in Theorem~\ref{the3.6} that every
$5$-strong semicomplete multipartite digraph is unilaterally $2$-linked, and this strong-connectivity condition is best possible. We first give an example showing that $4$-strong does not suffice.

\begin{proposition}\label{proposition+counterexample}
    There exists a 4-strong semicomplete multipartite digraph which is not unilaterally 2-linked.
\end{proposition}
\begin{proof}

Let
\[
A=\{s_1,t_1\},
B=\{s_2,t_2\},
X=\{x_1,x_2\}\;\text{and}\;
Y=\{y_1,y_2\}.
\]
Now we construct a digraph $D$ with $V(D)=A\cup B\cup X\cup Y$, where the partite classes of $D$ are $A, B, \{x_1\}, \{x_2\}, \{y_1\}, \{y_2\}$. The arc set of $D$
 consists exactly of the following arcs (see Figure \ref{Counterexample-4-strong-not-unilaterally-2-linked}):
\[
A\leftrightarrow B,
A\to X, X\to B, B\to Y, Y\to A,
x_1\leftrightarrow x_2,
y_1\leftrightarrow y_2
\;\text{and}\;
x_1\to y_1\to x_2\to y_2\to x_1.
\] 
It follows from the construction that $D$ is a semicomplete multipartite digraph.

We first show that $D$ is $4$-strong. Let
$S\subseteq V(D)$ with $|S|\leq 3$. Suppose first that both $A\setminus S$ and $B\setminus S$ are nonempty.
Since $A\leftrightarrow B$, all vertices in
$(A\cup B)\setminus S$ lie in the same strong component. Moreover, for
every $x\in X\setminus S$,
$A\setminus S\to x\to B\setminus S$,
and for every $y\in Y\setminus S$,
$B\setminus S\to y\to A\setminus S$.
Hence every vertex of $D-S$ belongs to the same strong component, and so $D-S$ is strong.

Suppose next that $A\subseteq S$. Since $|A|=2$ and $|S|\le3$, at most one vertex is deleted from $B\cup X\cup Y$. In particular,
$B\setminus S\ne\varnothing$,
$X\setminus S\ne\varnothing$ and
$Y\setminus S\ne\varnothing$.
The subdigraphs induced by $X\setminus S$ and $Y\setminus S$ are strong, and
$X\setminus S\to B\setminus S\to Y\setminus S$.
Furthermore, the two arcs
$y_1x_2, y_2x_1$
are vertex-disjoint, and hence at least one of these two arcs remains in
$D-S$. Thus there is also a directed path from $Y\setminus S$ to $X\setminus S$.
Consequently, $D-S$ is strong. The case $B\subseteq S$ is analogous. Since $|A\cup B|=4>|S|$, these cases cover all possibilities.
Therefore $D$ is $4$-strong.

We next show that $D$ is not unilaterally $2$-linked.
Consider the two terminal pairs
$\{s_1,t_1\}$ and $\{s_2,t_2\}$.
On the contrary, suppose that there are two vertex-disjoint
directed paths $P$ and $Q$, where $P$ joins $s_1$ and $t_1$, and $Q$ joins $s_2$ and $t_2$.
Since $P$ and $Q$ are vertex-disjoint,
$P\subseteq D-B
\;\text{and}\;Q\subseteq D-A$.
In $D-B$, every directed path from one vertex of $A$ to the
other must leave $A$ through $X$ and return to $A$ through $Y$.
Hence $P$ must use an arc from $X$ to $Y$. The only such arcs are $x_1y_1$ and $x_2y_2$.
Similarly, in $D-A$, every directed path from one vertex of $B$ to the other must use an arc from $Y$ to $X$. The only such arcs are $y_1x_2$ and $y_2x_1$.
However, each of the two $X\to Y$ arcs has a common endpoint with
each of the two $Y\to X$ arcs. Therefore $P$ and $Q$ cannot be
vertex-disjoint, a contradiction.
Hence $D$ is not unilaterally $2$-linked.
\end{proof}

\begin{figure}[htbp]
\centering
\resizebox{0.71\linewidth}{!}{%
\begin{tikzpicture}[
  vertex/.style={circle,draw=black,fill=white,text=black,line width=.85pt,minimum size=10.5mm,inner sep=0pt,outer sep=0pt,font=\Large},
  box/.style={rounded corners=2.5pt,draw=black,fill=none,text=black,line width=.8pt,inner xsep=7pt,inner ysep=4pt,font=\Large},
  arc/.style={-{Latex[length=2.4mm,width=1.55mm]},draw=black,line width=.78pt,shorten <=.2pt,shorten >=.7pt},
  redarc/.style={arc,draw=black,line width=1.4pt},
  line cap=round,line join=round]

\node[vertex] (s1) at (-5,1.5) {$s_1$};
\node[vertex] (t1) at (-5,-1.5) {$t_1$};
\node[vertex] (s2) at (5,1.5) {$s_2$};
\node[vertex] (t2) at (5,-1.5) {$t_2$};
\node[vertex] (x1) at (-1.5,4) {$x_1$};
\node[vertex] (x2) at (1.5,4) {$x_2$};
\node[vertex] (y1) at (-1.5,-4) {$y_1$};
\node[vertex] (y2) at (1.5,-4) {$y_2$};
\draw[arc] (s1.19.42) .. controls (-1.79,2.63) and (1.67,2.17) .. (s2.168.7);
\draw[arc] (s2.202.7) .. controls (1.86,0.19) and (-1.71,0.64) .. (s1.345.42);
\draw[arc] (t1.340.2) .. controls (-1.8,-2.65) and (2.01,-3.13) .. (t2.208.55);
\draw[arc] (t2.157.55) .. controls (1.86,-0.2) and (-1.7,-0.67) .. (t1.14.2);
\draw[arc] (s1.2.42) .. controls (-1.45,1.65) and (2.26,0.76) .. (t2.140.55);
\draw[arc] (t2.174.55) .. controls (1.47,-1.16) and (-2.65,-1.16) .. (s1.311.42);
\draw[arc] (t1.357.2) .. controls (-1.45,-1.67) and (2.27,-0.77) .. (s2.219.7);
\draw[arc] (s2.185.7) .. controls (1.47,1.15) and (-1.96,0.34) .. (t1.31.2);
\draw[arc] (s1.53.42) .. controls (-4.13,2.67) and (-2.84,3.42) .. (x1.203.54);
\draw[arc] (s1.36.42) .. controls (-3.09,2.91) and (-0.76,3.31) .. (x2.197.04);
\draw[arc] (t1.65.2) .. controls (-4.07,0.51) and (-2.28,1.92) .. (x1.249.53);
\draw[arc] (t1.48.2) .. controls (-3.07,0.66) and (-0.47,1.87) .. (x2.227.24);
\draw[arc] (x1.347.36) .. controls (0,3.66) and (3.92,2.1) .. (s2.134.7);
\draw[arc] (x1.330.36) .. controls (1.02,2.57) and (4.18,1.28) .. (t2.106.55);
\draw[arc] (x2.336.46) .. controls (2.84,3.42) and (4.13,2.67) .. (s2.126.58);
\draw[arc] (x2.292.97) .. controls (2.36,1.96) and (3.77,0.35) .. (t2.123.55);
\draw[arc] (s2.236.7) .. controls (3.41,-0.92) and (1.23,-3.03) .. (y1.19.64);
\draw[arc] (s2.253.7) .. controls (4.38,-0.63) and (3.15,-2.52) .. (y2.42.03);
\draw[arc] (t2.191.55) .. controls (2.68,-1.97) and (0.4,-2.59) .. (y1.36.64);
\draw[arc] (t2.233.42) .. controls (4.13,-2.67) and (2.84,-3.42) .. (y2.23.54);
\draw[arc] (y1.131.47) .. controls (-2.97,-2.34) and (-4.12,-0.53) .. (s1.293.47);
\draw[arc] (y1.162.46) .. controls (-2.89,-3.56) and (-4.14,-2.68) .. (t1.306.2);
\draw[arc] (y2.129.16) .. controls (-0.33,-1.76) and (-2.53,-0.02) .. (s1.328.42);
\draw[arc] (y2.150.96) .. controls (-0.57,-2.85) and (-3.1,-2.92) .. (t1.323.2);
\draw[arc] (x1.45) .. controls (-0.52,4.98) and (0.52,4.98) .. (x2.135);
\draw[arc] (x2.162) .. controls (0.53,4.32) and (-0.53,4.32) .. (x1.18);
\draw[arc] (y1.342) .. controls (-0.53,-4.32) and (0.53,-4.32) .. (y2.198);
\draw[arc] (y2.225) .. controls (0.52,-4.98) and (-0.52,-4.98) .. (y1.315);
\draw[redarc] (x1.267.28) .. controls (-1.63,1.28) and (-1.5,-1.28) .. (y1.90);
\draw[redarc] (y1.60.44) .. controls (-0.07,-1.47) and (0.78,1.18) .. (x2.255.72);
\draw[redarc] (x2.272.72) .. controls (1.63,1.28) and (1.5,-1.28) .. (y2.90);
\draw[redarc] (y2.112.16) .. controls (0.4,-1.31) and (-0.78,1.18) .. (x1.284.28);

\node[box] at (-7.75,0) {$A=\{s_1,t_1\}$};
\node[box] at (7.75,0) {$B=\{s_2,t_2\}$};
\node[box] at (0,5.65) {$X=\{x_1,x_2\}$};
\node[box] at (0,-5.65) {$Y=\{y_1,y_2\}$};
\end{tikzpicture}%
}
\caption{A $4$-strong semicomplete multipartite digraph that is not unilaterally $2$-linked.}
\label{Counterexample-4-strong-not-unilaterally-2-linked}
\end{figure}
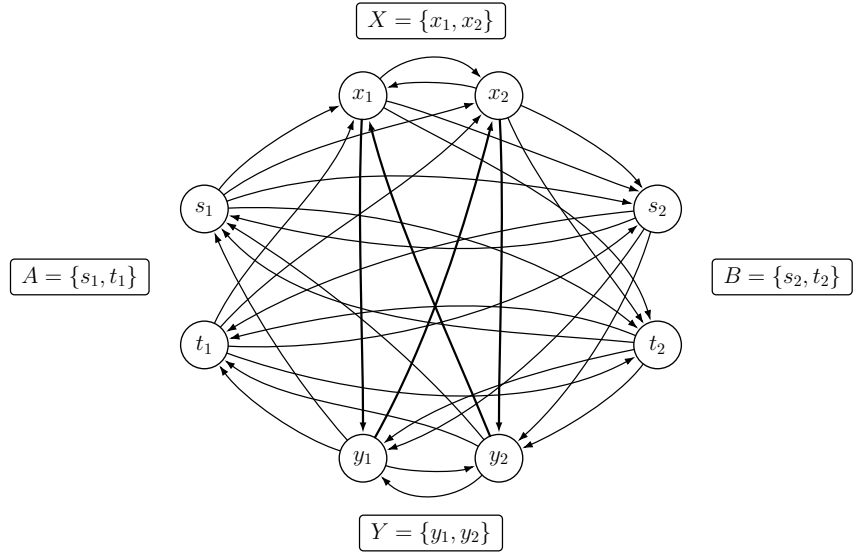

The digraph  in Proposition \ref{proposition+counterexample} has only eight vertices. So far, we have not found any counterexamples with more vertices. This leads us to the following problem.

\begin{problem}
Is every 4-strong semicomplete multipartite digraph on at least 9 vertices unilaterally 2-linked?

\end{problem}

Proposition \ref{proposition+counterexample} shows that the 5-strong connectivity condition in the following theorem is best possible.

\ThmUnilateralsmd*

\begin{proof}
Let $D$ be a $5$-strong semicomplete multipartite digraph, and let $s_1,t_1,s_2,t_2$ be four distinct vertices of $D$. For each $i\in[2]$, the digraph $D-\{s_{3-i},t_{3-i}\}$ is $3$-strong. Hence, by Menger's theorem, $D-\{s_{3-i},t_{3-i}\}$ contains three internally disjoint $(s_i,t_i)$-paths for each $i\in[2]$.

On the contrary, suppose that
$(D,\{s_1,t_1\},\{s_2,t_2\})$ is not good. We first observe that, for each $i\in[2]$, the digraph
$D-\{s_{3-i},t_{3-i}\}$ contains neither an $(s_i,t_i)$-path nor a $(t_i,s_i)$-path of length at most three. Otherwise, let $R$ be such a path. Since $R$ has at most two internal vertices, and the three $(s_{3-i},t_{3-i})$-paths are internally disjoint, at least one of them is vertex-disjoint from $R$. This contradicts the assumption that $(D,\{s_1,t_1\},\{s_2,t_2\})$ is not good.
Hence $s_1$ and $t_1$ are nonadjacent, and therefore they belong to the same partite class, say $A$. Similarly, $s_2$ and $t_2$ belong to the same partite class, say $B$, where possibly $A=B$.

For every vertex $v\in V(D)\setminus(A\cup\{s_2,t_2\})$, the vertex $v$ is adjacent to both $s_1$ and $t_1$. Since neither
$s_1vt_1$ nor $t_1vs_1$ is a directed path, either
$\{s_1,t_1\}\to v$ with no reverse arc, or
$v\to\{s_1,t_1\}$ with no reverse arc.
Choose three internally disjoint $(s_1,t_1)$-paths
$P_1,P_2,P_3$ in $D-\{s_2,t_2\}$ such that
$\sum_{i=1}^3 |P_i|$ is minimum.
Let $v$ be an internal vertex
of $P_i$ with $v\notin A$. If $v$ is dominated by both $s_1$ and $t_1$, then $v$ must be the first internal vertex of $P_i$; otherwise the arc $s_1v$ yields a shorter $(s_1,t_1)$-path.
Similarly, if $v$ dominates both $s_1$ and $t_1$, then $v$ must be the last internal vertex of $P_i$. Hence every internal vertex of $P_i$, except the first and the last, lies in $A$. Since $A$ is independent, there is at most one such middle
vertex. Moreover, the first and the last internal vertices of $P_i$ do not belong to $A$. Since $P_i$ has length at least four, it follows that \[\text{$P_i=s_1x_im_iy_it_1$, where $m_i\in A$ and $x_i,y_i\notin A$.}\] The internal vertices
of $P_1,P_2,P_3$ are pairwise distinct.

Applying the same argument to $s_2$ and $t_2$, choose three internally disjoint $(s_2,t_2)$-paths $Q_1,Q_2,Q_3$ in
$D-\{s_1,t_1\}$ such that $\sum_{j=1}^3 |Q_j|$ is minimum. Then each $Q_j$ has length four, and its unique middle internal vertex belongs to $B$.

Since $(D,\{s_1,t_1\},\{s_2,t_2\})$ is not good, we have $V(P_i)\cap V(Q_j)\neq\varnothing$ for all  $i,j\in[3]$. As $P_i\subseteq D-\{s_2,t_2\}$ and $Q_j\subseteq D-\{s_1,t_1\}$, every such intersection is an internal vertex of both paths. Since each $P_i$ has exactly three
internal vertices and $Q_1,Q_2,Q_3$ are internally disjoint, we have $|V(P_i)\cap V(Q_j)|=1$ for all $i,j\in[3]$. Let
\[
X=\{x_1,x_2,x_3\}, 
M=\{m_1,m_2,m_3\}\;\text{and}\;
Y=\{y_1,y_2,y_3\}.
\]
Hence $\bigcup_{j=1}^3\bigl(V(Q_j)\setminus\{s_2,t_2\}\bigr)
=X\cup M\cup Y$.

We next show that no $Q_j$ contains an arc of the form
$x_im_h$ or $m_hy_i$. Suppose that some $Q_j$ contains
$x_im_h$. Then $i\ne h$, and
$s_1x_im_hy_ht_1$
is an $(s_1,t_1)$-path. The vertices $x_i$ and $m_h$ lie on
$Q_j$, while $y_h$ lies on at most one of the paths
$Q_\ell$, $\ell\in[3]\setminus\{j\}$. Hence $s_1x_im_hy_ht_1$ is vertex-disjoint from one of
$Q_1,Q_2,Q_3$, a contradiction. Similarly, if some $Q_j$
contains $m_hy_i$, then
$s_1x_hm_hy_it_1$
is disjoint from one of the paths $Q_1,Q_2,Q_3$, again a
contradiction.

We claim that $A=B$. Otherwise, suppose that $A\ne B$. The middle internal vertex of each
$Q_j$ belongs to $B$, and hence does not belong to $M$. We claim
that no $Q_j$ contains two vertices of $M$. Otherwise, these two
vertices are the first and last internal vertices of $Q_j$, and
its middle internal vertex lies in $X\cup Y$. If the middle
vertex lies in $X$, then $Q_j$ contains an arc of the form
$x_im_h$; if it lies in $Y$, then $Q_j$ contains an arc of the
form $m_hy_i$. Both are impossible.

Thus each $Q_j$ contains at most one vertex of $M$. Since all three vertices of $M$ occur on $Q_1,Q_2,Q_3$, each $Q_j$ contains exactly one vertex of $M$. If the unique vertex of $M$ on $Q_j$ is the first internal vertex, then the second internal vertex belongs to $X$ since $Q_j$ contains no arc of the form $m_h y_i$. The third internal vertex also belongs to $X$, since
otherwise $Q_j$ contains an arc $x_iy_h$, and
$s_1x_iy_ht_1$ is a forbidden path of length three. Hence such a
$Q_j$ has internal-vertex type $(M,X,X)$. Similarly, if the unique vertex of $M$ on $Q_j$ is the last internal vertex, then $Q_j$ has type
$(Y,Y,M)$.

Among $Q_1,Q_2,Q_3$, at least two have the same type.
Two paths of type $(M,X,X)$ require four distinct vertices of
$X$, while two paths of type $(Y,Y,M)$ require four distinct
vertices of $Y$. Both are impossible since $|X|=|Y|=3$. Thus, $A=B$.

The middle internal vertex of each
$Q_j$ belongs to $A$, and hence belongs to $M$. Its first and
last internal vertices therefore lie in $X\cup Y$. Since no $Q_j$ contains an arc of the form $x_im_h$ or
$m_hy_i$, the first internal vertex of $Q_j$ cannot lie in $X$
and the last internal vertex cannot lie in $Y$. Thus every
$Q_j$ has the form
$Q_j=s_2y_im_hx_kt_2$.
The indices $i,h,k$ are pairwise distinct. Relabelling the
indices if necessary, we may assume that
\[
\text{$Q_1=s_2y_1m_2x_3t_2$,
$Q_2=s_2y_2m_3x_1t_2$ and
$Q_3=s_2y_3m_1x_2t_2$.}
\]

The vertices $x_1$ and $m_2$ belong to distinct partite classes
and hence are adjacent. If $x_1\to m_2$, then
$s_1x_1m_2y_2t_1$ is disjoint from $Q_3$, a contradiction.
Otherwise, $m_2\to x_1$, and
$s_2y_1m_2x_1t_2$ is disjoint from
$P_3=s_1x_3m_3y_3t_1$, again a contradiction.
Thus $(D,\{s_1,t_1\},\{s_2,t_2\})$ is good. Since
$s_1,t_1,s_2,t_2$ were chosen arbitrarily, $D$ is unilaterally $2$-linked.
\end{proof}

In view of Proposition \ref{6-strong-2-linked} and Theorem \ref{the3.6}, it is natural to ask whether 5-strong connectivity already guarantees 2-linkage.

\begin{problem}
    Is every 5-strong semicomplete multipartite digraph 2-linked?
\end{problem}

\section{Unilateral linkages in co-bipartite digraphs}\label{sec:cobiplinksec}

Recall that a digraph $D=(V,A)$ is co-bipartite if one can partition its vertex set into two sets $V_1,V_2$ so that each of the digraphs $D[V_1],D[V_2]$ are semicomplete.
Chudnovsky, Scott and Seymour \cite{chudnovskyJCT135} showed that the 2-linkage problem is polynomial for co-bipartite digraphs and hence so is the unilateral 2-linkage problem. We are not aware of any results on sufficient connectivities for co-bipartite digraphs being 2-linked.
The digraph in Figure \ref{Counterexample-4-strong-not-unilaterally-2-linked} is also co-bipartite, implying that not all 4-strong co-bipartite digraphs are unilaterally 2-linked.

\Thmunilateralcobip*

\begin{proof}
The proof reduces the problem to Theorem \ref{split digraph:theorem}, by displaying a nice connection between 2-linkages in split digraphs and unilateral 2-linkages in co-bipartite digraphs.

Let \(D\) be a \(6\)-strong co-bipartite digraph and \(V(D)=A\mathbin{\cup} B\), where $A\cap B=\emptyset$ and both $A$ and $B$ induce semicomplete digraphs. On the contrary, suppose that the pairs \((s_1,t_1)\) and \((s_2,t_2)\) have no
unilateral linkage.

Neither \(s_1\) and \(t_1\) nor \(s_2\) and \(t_2\) can be adjacent. Otherwise, suppose that
$s_1$ and $t_1$ are adjacent. Since $D-\{s_1,t_1\}$ is strong,
there exists a $(s_2,t_2)$-path in $D-\{s_1,t_1\}$.
Together with the arc between $s_1$ and $t_1$, we obtain
a unilateral linkage, a contradiction.
 Thus, we may assume that $s_1,s_2\in A$ and $t_1,t_2\in B$.

Since \(D\) is \(6\)-strong, both \(D-\{s_2,t_2\}\) and \(D-\{s_1,t_1\}\) are \(4\)-strong. By Menger's theorem, there exist four internally disjoint \((s_1,t_1)\)-paths \(P_1,\ldots,P_4\) in \(D-\{s_2,t_2\}\) and four internally disjoint \((t_2,s_2)\)-paths \(Q_1,\ldots,Q_4\) in \(D-\{s_1,t_1\}\).
Choose $P_1,\ldots,P_4$ and $Q_1,\ldots,Q_4$ so that
the total length of each set of paths is minimum. Thus, if $u$ precedes
$v$ on any of these paths and they are not consecutive,
then $uv$ is not an arc of $D$. Since the two terminal
pairs are not unilaterally linked, every $P_i$ intersects
every $Q_j$.

We first show that every \(P_i\) has the form \(A^*B^*\). Suppose that some \(P_i\) contains an arc \(xy\), where \(x\in B\) and \(y\in A\) are internal vertices. Since \(D[B]\) is semicomplete, either \(x\to t_1\) or \(t_1\to x\). Then $t_1\to x$, for otherwise we can obtain a shorter $(s_1,t_1)$-path. Similarly, we have \(y\to s_1\). Thus \(t_1xys_1\) is a \((t_1,s_1)\)-path with only two internal vertices. Since \(Q_1,\ldots,Q_4\) are internally disjoint and avoid \(s_1,t_1\), at least one of them is disjoint from \(t_1xys_1\), a contradiction. Hence every \(P_i\) has the form \(A^*B^*\). Similarly,
each \(Q_j\) has the form \(B^*A^*\).

Let
\[
A_i=(V(P_i)\cap A)\setminus\{s_1\},\;
B_i=(V(P_i)\cap B)\setminus\{t_1\},
\]
and
\[
X_j=(V(Q_j)\cap A)\setminus\{s_2\},\;
Y_j=(V(Q_j)\cap B)\setminus\{t_2\}.
\]
Assign the following types:

\begin{center}
\renewcommand{\arraystretch}{1.3}
\begin{tabular}{cccl}
\hline
Path & Type $+$ & Type $-$ & Type 0\\
\hline
$P_i$ & $s_2\to u$ for some $u\in A_i$
      & $v\to t_2$ for some $v\in B_i$ 
      &  $A_i\mapsto s_2$ \mbox{ and} $t_2\mapsto B_i$\\
$Q_j$ & $u\to s_1$ for some $u\in X_j$
      & $t_1\to v$ for some $v\in Y_j$
      & $s_1\mapsto X_j$ \mbox{ and } $Y_j\mapsto t_1$\\
\hline
\end{tabular}
\end{center}

No path can have both type $+$ and type $-$. Otherwise, suppose that
$P_i$ has both types. Then there exist $u\in A_i$
and $v\in B_i$ such that $s_2\to u$ and $v\to t_2$.
Since $P_i$ has the form $A^*B^*$, the $(s_2,t_2)$-path
$Q=s_2u\circ P_i[u,v]\circ vt_2$ is disjoint from
$P_k$ for every $k\ne i$, a contradiction. The argument for $Q_j$ is similar.

Choose one intersection vertex $w_{i,j}$  for each of the 16 pairs \(P_i,Q_j\), and form a
\(4\times4\) matrix whose entry is \(A\) or \(B\), according to whether $w_{i,j}\in A$  or $w_{i,j}\in B$.  Label its rows and columns by the corresponding path types $+$,$-$,$0$.

Using the definition of these three types, we now establish two constraints on this matrix.\\

\textbf{(i)} \emph{In any column, there is at most one \(A\)-entry in a
row of type \(-\) or \(0\), and at most one \(B\)-entry in a row of type
\(+\) or \(0\). 
Similarly, in any row, there is at most one $A$-entry in a column of type $-$ or $0$, and at most one $B$-entry in a column of type $+$ or $0$.}\\

For the first assertion, suppose that the entry in row $i$ and column $j$ is $A$, and that row $i$ has type $-$ or $0$. Since $P_i$ is not of type $+$,  we have $s_2\not\to w_{i,j}$. As $s_2,w_{i,j}\in A$ and $D[A]$ is semicomplete, we have $w_{i,j}\to s_2$. By minimality, $w_{i,j}$ must be the last internal vertex of $Q_j$; otherwise, we can obtain a shorter $(t_2,s_2)$-path by replacing $Q_j[w_{i,j},s_2]$ with the arc $w_{i,j}s_2$. Since $P_1,\ldots,P_4$ are internally disjoint, at most one of them contains this vertex. Hence column $j$ has at most one $A$-entry among the rows of type $-$ or $0$. The assertion for $B$-entries in a column follows similarly. The corresponding assertions for rows follow by the same argument, with $s_1,t_1$ in place of $s_2,t_2$.\\

\textbf{(ii)} \emph{An \(A\)-entry in a negative row forces its entire
column to contain only that \(A\)-entry. A \(B\)-entry in a positive
row forces its entire column to contain only that \(B\)-entry.
Similarly, an $A$-entry in a negative column is the only $A$-entry in its row, and a $B$-entry in a positive column is the only $B$-entry in its row.} 

For example, suppose that \(w_{i,j}\in A_i\cap X_j\), where \(P_i\) has type \(-\).
By (i), \(w_{i,j}\) is the last internal vertex of \(Q_j\). If \(Q_j\) has
another internal \(A\)-vertex, then the   predecessor \(z\) of $w_{i,j}$ on $Q_j$ is also in \(A\), and
minimality gives \(s_2\to z\). Choose \(v\in B_i\) with \(v\to t_2\).
Then $s_2\to z\to w_{i,j}\,P_i[w_{i,j},v]\to t_2$
is a directed walk using only the interior of \(P_i\) and one
additional vertex. It contains an $(s_2,t_2)$-path disjoint from another \(P_k\), a contradiction. Thus \(Q_j\) has only one internal
\(A\)-vertex. The other three assertions use the same predecessor or
successor argument.\\

We now extract the consequence of these constraints.

\begin{claim}\label{cl:1}
There is at most one row of type $0$ and at most one column of type $0$.
\end{claim}

\begin{proof}
Suppose there were two rows of type $0$. Then (i) would force their entries in every column to
be one \(A\) and one \(B\). Any additional positive row would then
be entirely \(A\), and any additional negative row entirely \(B\).

If one of each occurs, (ii) excludes both negative and positive
columns, leaving four neutral columns, contrary to (i). If both
additional rows are positive, there are no negative columns and
at most one neutral column. Hence there are at least three positive
columns. By (ii), each neutral row has at most one \(B\): if it has
a \(B\) in a positive column this follows directly; otherwise its
only possible \(B\) is in the neutral column. But the two neutral
rows together contain four \(B\)'s, a contradiction. The case of two
negative rows is symmetric.

By transposition, there is also at most one neutral column.
\end{proof}

\begin{claim}\label{cl:2}
If some row is positive (negative), then no other row is negative (positive).
\end{claim}
\begin{proof}

A \(B\)-entry in a positive row forces all other entries of its
column to be \(A\), by (ii). Then (i) permits at most one negative
or neutral row. Dually, an \(A\)-entry in a negative row permits at
most one positive or neutral row.

Thus, whenever both weak sign classes have size at least two,
positive rows are entirely \(A\) and negative rows entirely \(B\).
In the exceptional \(1+3\) cases the same conclusion follows: an
entirely \(A\) positive row excludes \(A\)-entries from negative
rows by (ii), and dually. These constant rows exclude both negative
and positive columns, a contradiction.

The same holds for columns.
\end{proof}

By Claims \ref{cl:1} and \ref{cl:2} at least three rows have one
common nonzero sign, and at least three columns have one common
nonzero sign. Their signs must agree: three positive rows and three
negative columns would form a \(3\times3\) submatrix with at most
three \(A\)'s by the row part of (ii), and at most three \(B\)'s by
its column part---impossible.

\begin{samepage}
After reversing all arcs and relabelling if necessary, we may
therefore assume
\[
\begin{aligned}
&\text{every }P_i\text{ has type }+\text{ or }0,\\
&\text{every }Q_j\text{ has type }+\text{ or }0,
\end{aligned}
\]
with at least three positive paths in each family.
\end{samepage}

Consequently,
$t_2\to B_i$ for every $i$ and
$Y_j\to t_1$ for every $j$.
Minimality of the paths $P_1,\ldots{},P_4$ and $Q_1,\ldots{},Q_4$ now implies that \textbf{every intersection in \(B\) is
simultaneously the last internal vertex of its \(P_i\) and the first
internal vertex of its \(Q_j\)}. 
Let $I$ be the set of common vertices of the two path families
that lie in $B$. Each $P_i$ and each $Q_j$ contains at most one vertex of $I$. Note that $I$ may be empty.
We construct a split digraph $H$ on
$A\cup I\cup\{t_1,t_2\}$ by replacing the remaining
$B$-segments of the paths with arcs as follows:

\begin{itemize}
\item A \(P_i\)-suffix becomes \(u_i\to t_1\), or \(u_i\to x\to t_1\)
when it contains \(x\in I\).
\item A \(Q_j\)-prefix becomes \(t_2\to v_j\), or \(t_2\to x\to v_j\)
when it contains \(x\in I\).
\end{itemize}
Here \(u_i\) is the last \(A\)-vertex of \(P_i\), possibly \(s_1\),
and \(v_j\) is the first \(A\)-vertex of \(Q_j\), possibly \(s_2\). Note that if some $B_i$ has no intersection with any of the $Y_j$'s, then we contract the path $P_i[u_i,t_1]$ into the arc $u_i\to t_1$ and similarly, if some $Y_j$ has no intersection with any $B_i$, then we contract the path $Q_j[t_2,v_j]$ into the arc $t_2\to v_j$. Finally, if $B_i$ does intersect some $Y_j$, then this is in the vertex $w_{i,j}$ which is both the first internal vertex of $Q_j$ and the last internal vertex of $P_i$ so $Q_j$ contains the arc $t_2\to w_{i,j}$, $P_i$ contains the arc $w_{i,j}\to t_1$ and we replace $Q_j[w_{i,j},v_j]$ by the arc $w_{i,j}\to v_j$ as well as $P_i[u_i,w_{i,j}]$ by the arc $u_i\to w_{i,j}$.

All suppressed segments have mutually disjoint interiors: their
cross-family intersections in \(B\) occur only at the retained
vertices of \(I\).
Retain the original arcs among the vertices of \(A\cup\{t_1,t_2\}\). Make these
vertices together with \(t_1,t_2\) semicomplete by adding, for every $z\in A$, the arcs $t_1\to z$ and $z\to t_2$ wherever $D$ has no adjacency between $z$ and $t_1$, respectively $t_2$.
The completion arcs cannot be used by disjoint \(s_1\to t_1,t_2\to s_2\)
paths: they leave the sink terminal \(t_1\) or enter the source
terminal \(t_2\). Every usable compressed arc lifts to its original
segment. Finally, retain the original adjacency between \(t_1,t_2\), and put no arcs
inside \(I\). Thus \(H\) is a split digraph.\\

The images of the four \(P_i\)'s and four \(Q_j\)'s remain internally
disjoint within their respective families. At least three paths in
each family contain an internal \(A\)-vertex, so compression
preserves four distinct paths even if the remaining path becomes
a single arc.

The local split-digraph theorem therefore gives disjoint \(s_1\to t_1\)
and \(t_2\to s_2\) paths in \(H\). Expanding arcs corresponding to contracted subpaths
preserves vertex-disjointness, giving the required unilateral
linkage in \(D\), a contradiction.
\end{proof}

\begin{problem}
    Is every 5-strong co-bipartite digraph unilaterally 2-linked?
\end{problem}

\section{Unilateral linkages in digraphs of independence number 2}\label{sec:alpha=2linksec}

We now turn our attention to general digraphs of independence number 2. This class of digraphs contains the co-bipartite digraphs but their structure is considerably more difficult to work with.
The purpose is to prove the following. We have not tried to optimize the connectivity but the current proof does not give a better bound.

\Thmunilateralalpha*

To prove Theorem~\ref{thm:alpha-two-unilateral}, we first establish the following lemma about co-bipartite digraphs.

\begin{lemma}\label{lem:oneway-seven-six}
Let $H$ be a co-bipartite digraph with $V(H)=V_1\mathbin{\dot\cup}V_2$ such that
$H[V_1]$ and $H[V_2]$ are semicomplete such that no arc goes from $V_2$ to $V_1$.
Let $s_1,s_2\in V_1$ and $t_1,t_2\in V_2$ be four distinct vertices.
Suppose that $H-\{s_2,t_2\}$ contains seven internally vertex-disjoint
$(s_1,t_1)$-paths and $H-\{s_1,t_1\}$ contains six internally
vertex-disjoint $(s_2,t_2)$-paths.
Then $H$ contains vertex-disjoint $(s_1,t_1)$- and $(s_2,t_2)$-paths.
\end{lemma}

\begin{proof}
Let $\mathcal P=\{P_1,\ldots,P_7\}$ be a family of seven
internally vertex-disjoint $(s_1,t_1)$-paths in
$H-\{s_2,t_2\}$, and let $\mathcal Q=\{Q_1,\ldots,Q_6\}$
be a family of six internally vertex-disjoint
$(s_2,t_2)$-paths in $H-\{s_1,t_1\}$.
Since no arc goes from $V_2$ to $V_1$, each of these paths contains
exactly one arc from $V_1$ to $V_2$, which we call its \emph{crossing arc}.
The part before this arc lies entirely in $V_1$, and the part after
it lies entirely in $V_2$.

Let $\mathcal F$ be either of the two families, and let $t$ be its
common last vertex, so $t=t_1$ for $\mathcal P$ and $t=t_2$ for
$\mathcal Q$.
Let $Y_{\mathcal F}$ be the set of heads of its crossing arcs,
excluding $t$.
These heads are distinct within $\mathcal F$, since its paths are
internally disjoint.
We call a crossing arc $xy$ \emph{good} if $y=t$ or
$d^+_{H[Y_{\mathcal F}]}(y)\geq2$, and \emph{very good} if
$y=t$ or $d^+_{H[Y_{\mathcal F}]}(y)\geq3$.
Thus every very good crossing arc is also good.
These definitions are always taken with respect to the family
containing the original path.
Note that several crossing arcs may have head $t$; we count these
arcs separately, according to their original paths.

We record the number of good and very good crossing arcs.
For every nonnegative integer $r$, a semicomplete digraph
has at most $2r+1$ vertices of outdegree at most $r$.
Indeed, if there are $m$ such vertices, their induced subdigraph
has at least $\binom m2$ arcs, whereas the sum of their outdegrees
is at most $rm$.
Consequently, a family of $m$ paths has at least $m-3$ good
crossing arcs and at least $m-5$ very good crossing arcs.
We shall also use the following equality case: if five vertices
have outdegree at most two, then each has outdegree exactly two.
This follows because their induced subdigraph already has at
least $\binom52=10$ arcs, which equals the upper bound on the
sum of their outdegrees.
In particular, if exactly five crossing arcs of a family are
not very good, then all crossing arcs of that family are good.

We next explain why these degree conditions are useful.
Suppose that $y\in Y_{\mathcal F}$ has at least $r$ out-neighbours
in $H[Y_{\mathcal F}]$.
There are then $r+1$ internally disjoint $(y,t)$-paths in $H[V_2]$.
One is the part of its own original path from $y$ to $t$.
For each of $r$ out-neighbours $z$, another is the arc $yz$
followed by the part of the original path from $z$ to $t$.
These paths are internally disjoint because the original paths
in $\mathcal F$ are internally disjoint.
They also avoid the last vertex of the other family.

It follows that if $xy$ is a very good crossing arc of
$\mathcal P$, $uv$ is a good crossing arc of $\mathcal Q$,
and $y\ne v$, then $H[V_2]$ contains vertex-disjoint
$(y,t_1)$- and $(v,t_2)$-paths.
To see this when $y\ne t_1$ and $v\ne t_2$, there are at least four
internally disjoint $(y,t_1)$-paths and at least three internally
disjoint $(v,t_2)$-paths as above.
Deleting $v$ destroys at most one path of the first family,
and deleting $y$ destroys at most one path of the second.
By Theorem~\ref{thm:local}, the semicomplete digraph $H[V_2]$ contains
vertex-disjoint $(y,t_1)$- and $(v,t_2)$-paths.
If $y=t_1$ or $v=t_2$, use the trivial path at that vertex and an
original path segment for the other pair.
The same argument applies with the two families interchanged.

It remains to find disjoint path segments in $V_1$ leading to
crossing arcs whose heads satisfy these conditions.
We prove the following two assertions.

\vskip0.1in

{\bf (i)} \emph{Three very good crossing arcs in one family and three
good crossing arcs in the other are sufficient.}

\vskip0.1in
By symmetry, suppose that the three very good crossing arcs
belong to $\mathcal P$.
If one of the six chosen crossing arcs starts at its own first
terminal, we can use that arc together with an original path
segment in $V_1$ from the other family.
The latter can be chosen to end at a crossing arc with a different
head: a nonterminal head occurs at most once in that family, and
the terminal heads $t_1,t_2$ cannot be shared by the two families.
The two segments in $V_1$ are disjoint because every original path
avoids the other two terminals.
The preceding argument in $H[V_2]$ completes the required linkage.

We may therefore assume that none of the chosen crossing arcs
starts at $s_1$ or $s_2$.
Let $w_1$ and $w_2$ be two new vertices, and let $Z$ be the set
of heads of the six chosen crossing arcs, excluding $t_1$ and $t_2$.
We construct a split digraph $F$ whose independent set is
$I=\{s_1,s_2\}\cup Z$ and whose semicomplete part is
$C=(V_1\setminus\{s_1,s_2\})\cup\{w_1,w_2\}$.
Retain all arcs of $H[V_1\setminus\{s_1,s_2\}]$ and all original arcs
from $s_1,s_2$ into $V_1\setminus\{s_1,s_2\}$.
For each chosen crossing arc $xy$ with $y\in Z$, add $xy$ to $F$.
Also add $yw_1$ if $y$ is the head of a chosen crossing arc
of $\mathcal P$, and add $yw_2$ if it is the head of a chosen
crossing arc of $\mathcal Q$.
A chosen crossing arc $xt_1$ of $\mathcal P$ is represented by
$xw_1$, and a chosen crossing arc $xt_2$ of $\mathcal Q$ is
represented by $xw_2$.
Finally, add $w_1 z$ and $w_2 z$ for every
$z\in V_1\setminus\{s_1,s_2\}$, together with one arc between
$w_1$ and $w_2$.
There are no other arcs in $F$.
In particular, $I$ is independent, $F[C]$ is semicomplete,
and no arc enters $s_1$ or $s_2$.

Each of the three selected $\mathcal P$-paths now gives an
$(s_1,w_1)$-path: follow its original part in $V_1$ and then use
$xyw_1$, or $xw_1$ when the crossing head was $t_1$.
These three paths are internally disjoint and avoid $s_2$ and $w_2$.
Similarly, the three selected $\mathcal Q$-paths give three
internally disjoint $(s_2,w_2)$-paths avoiding $s_1$ and $w_1$.
By Lemma~\ref{Lemma:{s_1,s_2} in V_1 : (3,3)}, we can find vertex-disjoint $(s_1,w_1)$- and $(s_2,w_2)$-paths in $F$.

Neither of these two paths can use an arc leaving $w_1$
or $w_2$, since each avoids the other pair of terminals.
Moreover, a vertex of $Z$ has no out-neighbours other than
$w_1$ and $w_2$.
Thus the first path consists of an original path segment in $V_1$
followed by a chosen crossing arc and the last arc into $w_1$;
if its last arc is $xw_1$ with $x\in V_1$, we replace it by $xt_1$.
The second path has the corresponding form with $w_2$ and $t_2$.
Their crossing heads are distinct, because a shared nonterminal
head is represented by a single vertex of $Z$, and $t_1,t_2$ occur
only for their respective families.
The first head is the head of a selected very good
$\mathcal P$-crossing arc, and the second is the head of a selected
good $\mathcal Q$-crossing arc.
This remains true even if the incoming crossing arc used by a
new path came from the other family.
We can therefore append the disjoint path segments in $H[V_2]$
established above.
Since $V_1$ and $V_2$ are disjoint, this gives the required linkage
in $H$.

\vskip0.1in

{\bf(ii)} \emph{Two very good crossing arcs in one family and three
good crossing arcs in the other are sufficient if no chosen head
from one family equals a chosen head from the other.}

\vskip0.1in

Again, we may suppose that the two very good crossing arcs belong
to $\mathcal P$.
If a chosen crossing arc starts at its own first terminal, the
direct argument in (i) applies.
Otherwise, let $T$ be a digraph on $V_1\cup\{w_3,w_4\}$,
where $w_3$ and $w_4$ are new vertices.
Retain $H[V_1]$, and add $xw_3$ for each chosen
$\mathcal P$-crossing tail $x$ and $uw_4$ for each chosen
$\mathcal Q$-crossing tail $u$.
Make $T$ semicomplete by adding arcs from $w_3$ and $w_4$ into
$V_1$ and one arc between $w_3$ and $w_4$.
The chosen original path segments give two internally disjoint
$(s_1,w_3)$-paths avoiding $s_2$ and $w_4$ and three internally
disjoint $(s_2,w_4)$-paths avoiding $s_1$ and $w_3$.
The crossing tails in each family are distinct, so these counts
are preserved.

By Theorem~\ref{thm:local}, $T$ has disjoint $(s_1,w_3)$- and
$(s_2,w_4)$-paths.
As before, these paths cannot use arcs leaving the two new
terminals. Replace their last arcs by the corresponding chosen crossing arcs in $H$.
The resulting heads are distinct by the assumption in (ii).
Then we append the two vertex-disjoint path segments in $H[V_2]$ to obtain vertex-disjoint $(s_1,t_1)$- and $(s_2,t_2)$-paths.

\smallskip
We now apply (i) and (ii) to the original seven and six paths.
The family $\mathcal P$ has at least two very good crossing arcs,
and $\mathcal Q$ has at least three good crossing arcs.
If $\mathcal P$ has at least three very good crossing arcs,
then the desired vertex-disjoint paths exist by (i).

Suppose that $\mathcal P$ has exactly two very good crossing arcs.
Its five remaining crossing heads have outdegree at most two
in $H[Y_{\mathcal P}]$.
Since these five heads induce at least $\binom{5}{2}=10$ arcs
and each has outdegree at most two in $H[Y_{\mathcal P}]$,
they must all have outdegree exactly two.
Thus all seven $\mathcal P$-crossing arcs are good.
If $\mathcal Q$ has at least two very good crossing arcs, choose
two of them.
Since $\mathcal Q$ avoids $t_1$, each of their heads occurs at most
once among the $\mathcal P$-crossing heads.
Discard the at most two $\mathcal P$-crossing arcs with a
coincident head.
At least five good crossing arcs remain, so we may select three
and apply (ii), with the roles of the families interchanged.

Finally, suppose that $\mathcal Q$ has fewer than two very good
crossing arcs.
It has at least $6-5=1$, so it has exactly one.
Its five other crossing heads again force all six
$\mathcal Q$-crossing arcs to be good.
Choose the two very good $\mathcal P$-crossing arcs, and discard
the at most two $\mathcal Q$-crossing arcs whose heads coincide
with either of their heads.
At least four good crossing arcs remain.
We choose three of these good crossing arcs.
Then it follows from (ii) that the desired vertex-disjoint paths exist.
This completes the proof.
\end{proof}

We now prove Theorem~\ref{thm:alpha-two-unilateral} by applying Lemmas~\ref{lem:degree-two-selection} and~\ref{lem:oneway-seven-six}.%
To use Lemma~\ref{lem:oneway-seven-six}, we shall delete the first and last internal vertices of several paths from one family. %

\begin{proof}[\textbf{Proof of Theorem~\ref{thm:alpha-two-unilateral}}]
Suppose, for a contradiction, that $D$ is a $16$-strong digraph with 
$\alpha(D)\leq2$, and four distinct vertices
$s_1,t_1,s_2,t_2$ such that $D$ has  no unilateral linkage between  $(s_1,s_2)$ and $(t_1,t_2)$.
Every directed path in either direction between either terminal
pair, avoiding the other pair, has at least $16$ vertices.
Indeed, deleting a path with fewer than $16$ vertices leaves
a strong digraph containing the other two terminals, and a path
between those terminals would complete a unilateral linkage.
In particular, $s_1$ and $t_1$ are nonadjacent, as are $s_2$ and $t_2$.

We first choose the two path families.
Both $D-\{s_2,t_2\}$ and $D-\{s_1,t_1\}$ are $14$-strong.
By Menger's theorem, there are fourteen internally disjoint
$(s_1,t_1)$-paths $P_1,\ldots,P_{14}$ avoiding $s_2$ and $t_2$,
and fourteen internally disjoint $(s_2,t_2)$-paths
$Q_1,\ldots,Q_{14}$ avoiding $s_1$ and $t_1$.
We may assume that every $P_i$ is minimal, this means that no arc
goes from a vertex of $P_i$ to a later nonconsecutive vertex.
This is achieved by repeatedly replacing a subpath by a forward
shortcut, which preserves internal disjointness.

Let $G$ be the bipartite graph whose two parts consist of
vertices representing the $(s_1,t_1)$-paths $P_1,\ldots,P_{14}$
and the $(s_2,t_2)$-paths $Q_1,\ldots,Q_{14}$, respectively. The vertices representing $P_i$ and $Q_j$ are adjacent in $G$ if and only if $Q_j$ contains the first or last internal vertex of $P_i$. For convenience, we refer to the paths $P_i$ and $Q_j$ as $P$-paths and $Q$-paths, respectively. 
Each such internal vertex belongs to at most one $Q$-path,
since the $Q$-paths are internally disjoint and the $P$-paths
avoid $s_2$ and $t_2$.
Thus every vertex in the $P$-part of $G$ has degree at most two.
Applying Lemma~\ref{lem:degree-two-selection} with $N=14$ and $p=7$, we can find seven $P$-paths whose first and last internal vertices together meet at most eight $Q$-paths.
After relabelling, retain these paths as $P_1,\ldots,P_7$ and
retain six of the remaining $Q$-paths as $Q_1,\ldots,Q_6$.

Let $E$ consist of the first and last internal vertices of the seven retained $P$-paths, and let
\[
 S=\left(\bigcup_{i=1}^{7}V(P_i)\right)\setminus E.
\]
Every retained $Q$-path avoids $E$.
Also, each $P_i$ has at least $16$ vertices, so $|E|=14$ and
each $P_i$ retains internal vertices in $S$.

We next describe the structure of $D[S]$.
For every $z\in S\setminus\{s_1,t_1\}$, the $P$-path containing $z$ is minimal, and this implies that $s_1z,zt_1\notin A(D)$.
Here the first and last internal vertices were removed precisely
to ensure that either such arc would be a forward shortcut.
Since $s_1$ and $t_1$ are nonadjacent and $\alpha(D)\leq2$, the vertex
$z$ is adjacent to at least one of them.
Its only possible adjacencies to these terminals are therefore
$zs_1$ and $t_1z$.
They cannot both occur, because $t_1zs_1$ would be a directed
path with only three vertices, contrary to the first paragraph
of the proof.
Consequently, $S=V_1\mathbin{\dot\cup}V_2$, where
\begin{align*}
 V_1=\{s_1\}\cup\{z\in S\setminus\{s_1,t_1\}:
                    z\text{ is nonadjacent to }t_1\}
                    \end{align*}                   
and
 \begin{align*}
 V_2=\{t_1\}\cup\{z\in S\setminus\{s_1,t_1\}:
                    z\text{ is nonadjacent to }s_1\}.
\end{align*}
Moreover, every vertex of $V_1\setminus\{s_1\}$ dominates $s_1$,
and $t_1$ dominates every vertex of $V_2\setminus\{t_1\}$.

Both $D[V_1]$ and $D[V_2]$ are semicomplete.
For example, a nonadjacent pair in $V_1$, together with $t_1$,
would be an independent set of size three.
The argument for $V_2$ is the same.
Moreover, no arc goes from $V_2$ to $V_1$.
If $xy\in A(D)$ for some $x\in V_2\setminus\{t_1\}$ and
$y\in V_1\setminus\{s_1\}$, then $t_1xys_1$ is a directed
$(t_1,s_1)$-path avoiding $s_2,t_2$ with only four vertices. This contradicts the fact that every such path has at least $16$ vertices.
Arcs from $t_1$ into $V_1$ or from $V_2$ into $s_1$ are excluded by the definitions of the two parts.

Every retained $Q_j$ meets every retained $P_i$, since otherwise
they themselves form a linkage.
All these intersections lie in $S\setminus\{s_1,t_1\}$.
We claim that, along any $Q_j$, every vertex in $V_1$ precedes
every vertex in $V_2$, whenever both parts are met.
Otherwise, some $u\in V(Q_j)\cap V_2$ precedes some
$v\in V(Q_j)\cap V_1$, and $t_1u\circ Q_j[u,v]\circ vs_1$
is a directed $(t_1,s_1)$-path.
Its internal vertices all lie in the interior of $Q_j$.
It is therefore disjoint from any other retained $Q$-path,
so we again obtain a unilateral linkage, a contradiction. Thus, along each $Q_j$, all vertices in $V_1$ occur before all vertices in $V_2$, whenever both sets are met.

We now construct the auxiliary digraph $H$ on
$S\cup\{s_2,t_2\}$.
Retain all arcs of $D[S]$.
For each retained $P_i$, replace its first two arcs by one arc
and replace its last two arcs by one arc; the removed internal
vertices are exactly its two vertices in $E$.
For each retained $Q_j$, replace every segment between
consecutive vertices of $S\cup\{s_2,t_2\}$ by one arc.
Thus the interior of each such segment contains no vertex
of $S\cup\{s_2,t_2\}$.
Let $V_1'=V_1\cup\{s_2\}$ and $V_2'=V_2\cup\{t_2\}$.
Make $H[V_1']$ and $H[V_2']$ semicomplete by adding the arcs
$xs_2$ for $x\in V_1$ and $t_2y$ for $y\in V_2$ whenever the
corresponding pair is nonadjacent.

No arc of $H$ goes from $V_2'$ to $V_1'$.
This is already true for arcs retained from $D[S]$.
An arc replacing an end segment of a $P$-path starts at
$s_1\in V_1'$ or ends at $t_1\in V_2'$, so it cannot go from
$V_2'$ to $V_1'$.
The arcs replacing segments of a $Q$-path have the same
property by the order proved above, with $s_2$ first in $V_1'$
and $t_2$ last in $V_2'$.
Finally, all arcs added to make the parts semicomplete lie
within the parts.

After compression, the seven $P$-paths are internally disjoint $(s_1,t_1)$-paths in $H-\{s_2,t_2\}$, and the six $Q$-paths are internally disjoint $(s_2,t_2)$-paths in $H-\{s_1,t_1\}$.
These paths remain distinct after compression.
Each $P_i$ retains at least fourteen vertices, and each $Q_j$
meets all seven $P$-paths in distinct retained internal vertices.
By Lemma~\ref{lem:oneway-seven-six}, we have vertex-disjoint $(s_1,t_1)$- and $(s_2,t_2)$-paths in $H$.

It remains to recover these paths in $D$.
Neither path can use an arc added only to make the two parts semicomplete: such an arc enters $s_2$ or leaves $t_2$, whereas the first path avoids both vertices and the second starts at $s_2$ and ends at $t_2$.
Every other arc is an original arc or represents a compressed path segment.

The interiors of all compressed segments lie outside $V(H)$ and are pairwise disjoint.
Indeed, the end segments of the $P$-paths use distinct vertices
of $E$, all retained $Q$-paths avoid $E$, and the interiors of
segments from the $Q$-paths are pairwise disjoint because these
paths are internally disjoint.
If an auxiliary arc has more than one representation, choose
any one of them.
After we replace each auxiliary arc by the path segment it
represents, the resulting $(s_1,t_1)$- and $(s_2,t_2)$-paths lie in $D$ and remain vertex-disjoint.
This contradiction completes the proof.
\end{proof}

\section{Remarks and further open problems}\label{sec:remarks}

For every choice of fixed positive integers $r,k$, Chudnovsky, Scott and Seymour \cite{chudnovskyJCT135} gave a polynomial algorithm for the $k$-linkage problem in digraphs whose vertex set can be covered by $r$ semicomplete digraphs, in particular the case $r=2$ which are the co-bipartite digraphs. However, it was mentioned in \cite{chudnovskyAM270} that the complexity 2-linkage problem for digraphs of independence number 2 was open and as far as we know it is still open (see also \cite{gomesLIPScs334}). Hence the following problem is relevant to ask.

\begin{problem}
    What is the complexity of the unilateral 2-linkage problem for digraphs of independence number 2?
\end{problem}

In the case of arc-disjoint linkages there is a polynomial algorithm for $k$-linkage in digraphs of independence number $\alpha$ for every fixed $k$ and $\alpha$ \cite{fradkinJCT110}.\\

The following conjectures are open already for $\alpha=2$.

\begin{conjecture}
    For every positive integer $\alpha$ there exists an integer $K_{\alpha}$ such that every $K_{\alpha}$-strong digraph $D$ with $\alpha(D)\leq\alpha$ is 2-linked.
\end{conjecture}

\begin{conjecture}
    For every positive integer $\alpha$ there exists an integer $K^u_{\alpha}$ such that every $K^u_{\alpha}$-strong digraph $D$ with $\alpha(D)\leq\alpha$ is unilaterally 2-linked.
\end{conjecture}

By Theorem \ref{thm:alpha-two-unilateral} $K_2^u$ exists and we have $K_2^u\leq 16.$\\

As mentioned in the introduction, the hamiltonian cycle problem is NP-complete for split digraphs \cite{bangJGT108}. For semicomplete multipartite digraphs the hamiltonian cycle problem is polynomial \cite{bangJGT29}. We do not know the complexity for co-bipartite digraphs.

\begin{conjecture}
    The hamiltonian cycle problem is polynomial for co-bipartite digraphs.
\end{conjecture}

For the larger class of digraphs with independence number 2 the complexity of the problem of finding an oriented hamiltonian cycle with the minimum number of backwards arcs was mentioned as an open problem (Problem 4.5) in \cite{gerkearxiv}.

\section*{Acknowledgments}
The support of the National Natural
Science Foundation of China (Grant No. 12571373) and the Natural Science Foundation of Shandong Province
(Grant No. ZR2025MS05) is greatly acknowledged.
\vskip0.2in

\noindent{}{\bf Declaration about the use of AI:} We used AI-assisted tools in the construction of the infinite class of 3-strong split digraphs which are not unilaterally 2-linked, starting from our own construction on 8 vertices.         It also constructed the 4-strong semicomplete multipartite digraph in Figure \ref{Counterexample-4-strong-not-unilaterally-2-linked}. It was AI assistance that suggested the reductions in the proof of Theorem \ref{thm:6-strongcobip} as well as the proofs of Theorems \ref{thm:u2lnknoK} and \ref{thm:alpha-two-unilateral}. The authors take full responsibility for all results and proofs in the paper.

\bibliographystyle{plainnat} %

\end{document}